\documentclass[a4paper,11pt]{article}
\usepackage[english]{babel}
\usepackage[utf8]{inputenc}
\usepackage{paralist,url,verbatim,anysize,enumitem}
\usepackage{amscd}
\usepackage{fancyhdr}
\usepackage{mathrsfs, mathtools}
\usepackage{cancel} %for strike out text
\usepackage{xfrac,faktor}
\usepackage[e]{esvect}
\usepackage{amsmath,amsthm,amssymb,amsfonts}
\usepackage{graphicx, graphics}
\usepackage[bookmarksopen=true]{hyperref}
\hypersetup{colorlinks=true, citecolor=blue}
\usepackage{bbm}
\usepackage[font=small,labelfont=bf]{caption}
\usepackage{subcaption}
\usepackage{bm} %bold
\usepackage[symbol*]{footmisc} %per autori
\usepackage{tikz} %for the command circled

\theoremstyle{plain}
\newtheorem{theorem}{\bf Theorem}[]

\newtheorem{corollary}[theorem]{Corollary}
\newtheorem{lemma}[theorem]{Lemma}
\newtheorem{proposition}[theorem]{Proposition}

\newtheorem{thmnonumber}{\bf Main Theorem}

\theoremstyle{definition}
\newtheorem{example}[theorem]{Example}
\newtheorem{nonexample}[theorem]{Non-Example}
\newtheorem{remark}[theorem]{Remark}

\newtheorem{definition}[theorem]{Definition}

\newcommand{\NN}{\mathbb{N}}
\newcommand{\ZZ}{\mathbb{Z}}
\newcommand{\RR}{\mathbb{R}}

\newcommand{\C}{\mathcal{C}}

\newcommand{\mm}{\mathfrak{m}}

\newcommand{\supp}{\mathrm{supp}}
\newcommand{\init}{\mathrm{in}}
\newcommand{\height}{\mathrm{height}}
\newcommand{\Hom}{\mathrm{Hom}}

\newcommand{\mmu}{{\bm \mu}}

\newcommand\eqdef{\mathrel{\overset{\makebox[0pt]{\mbox{\normalfont\tiny\sffamily def}}}{=}}}

\newcommand\iffdef{\mathrel{\overset{\makebox[0pt]{\mbox{\normalfont\tiny\sffamily def}}}{\Longleftrightarrow}}}

\definecolor{mypink}{RGB}{215, 5, 234}
\definecolor{lemonchiffon}{RGB}{255, 250, 205}

\renewcommand{\textdagger}{$**$}
\renewcommand{\textdaggerdbl}{$**$}
\renewcommand{\textparagraph}{$*$}
\DefineFNsymbolsTM{otherfnsymbols}{%
  \textparagraph \mathparagraph
  \textdagger    \dagger
  \textdaggerdbl \ddagger
  \textasteriskcentered *
  \textbardbl    \|%
}%
\setfnsymbol{otherfnsymbols}

\begin{document}

\author{Bruno Benedetti \thanks{Supported by NSF Grant  1855165, ``Geometric Combinatorics and Discrete Morse Theory''.}\\
\small Dept. of Mathematics\\
\small University of Miami\\
\small bruno@math.miami.edu
\and Lisa Seccia$\phantom{i}^{**}$\\
\small Dip. di Matematica\\
\small Univ. degli Studi di Genova\\
\small seccia@dima.unige.it
\and Matteo Varbaro \thanks{Supported by 100021 2019 ADS FRA, ``Algebra commutativa: ideali iniziali square-free, singolarit\`a in caratteristica mista e aspetti computazionali".}\\
\small Dip. di Matematica\\
\small Univ.  degli Studi di Genova\\
\small varbaro@dima.unige.it}

\date{April 12, 2021}
\title{Hamiltonian paths, unit-interval complexes, \\and determinantal facet ideals}
\maketitle

\begin{abstract} We study $d$-dimensional generalizations of three mutually related topics in graph theory: Hamiltonian paths, (unit) interval graphs, and binomial edge ideals. 
We provide partial  high-dimensional generalizations of Ore and P\'osa's sufficient conditions for a graph to be Hamiltonian. 
We introduce a hierarchy of combinatorial properties for simplicial complexes that generalize unit-interval, interval, and co-comparability graphs. 
We connect these properties to the already existing notions of determinantal facet ideals and (tight and weak) Hamiltonian paths in simplicial complexes. 
Some important consequences of our work are:
\begin{compactenum}[(1)]
\item Every unit-interval strongly-connected $d$-dimensional simplicial complex is traceable. \\
(This extends the well-known result ``unit-interval connected graphs are traceable''.)
\item Every unit-interval $d$-complex that remains strongly connected after the deletion of~$d$ or less vertices, is Hamiltonian. \\
(This extends the fact that  ``unit-interval $2$-connected graphs are Hamiltonian''.)
\item Unit-interval complexes are characterized, among traceable complexes, by the property that the minors defining their determinantal facet ideal form a Gr\"obner basis for a diagonal term order which is compatible with the traceability of the complex.\\(This corrects a recent theorem by Ene et al., extends a result by Herzog and others, and partially answers a question by Almousa--Vandebogert.)
\item Only the $d$-skeleton of the simplex has a determinantal facet ideal with linear resolution. \\
(This extends the result by Kiani and Saeedi-Madani that ``only the complete graph has a binomial edge ideal with linear resolution''.)
\item The determinantal facet ideals of all under-closed and semi-closed complexes have a square-free initial ideal with respect to lex. In characteristic $p$, they are even F-pure.\\ 

\end{compactenum}
\end{abstract}

\vskip2mm

\section*{Introduction}
The first Combinatorics paper in History is apparently Leonhard Euler's 1735 solution of the K\"onigsberg bridge problem. In that article, Euler introduced the notion of graph, and studied cycles (now called `Eulerian') that touch all edges exactly once. Euler proved that the graphs admitting them, are exactly those graphs with all vertices of even degree. \emph{Hamiltonian} cycles are instead cycles that touch all \emph{vertices} exactly once; they are named after sir William Rowan Hamilton, who in 1857 invented a puzzle game which asked to find one such cycle in the icosahedron. %The two properties are independent. %and have an interesting interplay if one pass to the line graph (if $G$ is Eulerian, its line graph is both Eulerian and Hamiltonian; if $G$ is Hamiltonian, its line graph is Hamiltonian, but need not be Eulerian.) 
Unlike for the Eulerian case, figuring out if a graph admits a Hamiltonian cycle or not is a hard problem, now known to be NP-complete \cite{Karp}. 

Even if simple characterizations are off the table, in the 1950s and 1960s Dirac, Ore, P\'osa and others were able to obtain simple conditions on the vertex degrees (in the spirit of Euler's work) that are \emph{sufficient} for a graph to admit Hamiltonian cycles \cite{Dirac, Ore, Posa}. Ore's theorem, for example, says, ``{\it Any graph with $n$ vertices such that 
$\deg u + \deg v \ge n$ for all non-adjacent vertices $u,v$, admits a Hamiltonian cycle}''. Ore's condition is far from being necessary: In any cycle, no matter how large, one has $\deg u + \deg v =4$ for all $u,v$.

In the same years, the two papers \cite{LekBol} and \cite{GilmoreHoffman} initiated the study of \emph{unit-interval graphs}. This very famous class  consists, as the name suggests, of all intersection graphs of a bunch of length-one open intervals  on the real line. (That is, we place a node in the middle of each interval, and we connect two nodes with an arc if and only if the corresponding intervals overlap). Bertossi's theorem says that if they are connected, such graphs always admit \emph{Hamiltonian paths}, i.e. paths that touch all vertices once \cite{Bertossi}. Chen--Chang--Chang's theorem states that $2$-connected unit-interval graphs admit Hamiltonian cycles \cite{ChChCh}. For these results, the length-one request can be weakened to ``pairwise not-nested'', but it cannot be dismissed: Within the larger world of  \emph{interval graphs}, one encounters connected graphs such as $K_{1,3}$  that do not admit Hamiltonian paths, and also $2$-connected graphs like the $G_5$ of Remark~\ref{rem:G5} that do not admit Hamiltonian cycles.

In the 1970s, the work of Stanley and Reisner established a fundamental bridge between Combinatorics and Commutative Algebra, namely, a natural bijection between labeled simplicial complexes on $n$ vertices and radical monomial ideals in a polynomial ring with $n$ variables. This correspondence lead Stanley to prove the famous Upper Bound Theorem for triangulated spheres~\cite{StanleyUBT}. After this success, many authors have investigated  ways to encode graphs into monomial ideals. In 2010, Herzog et al.~\cite{HerzogEtAl} first considered a natural way to encode graphs into binomial ideals, the so-called \emph{binomial edge ideals}. The catch is that all such binomial edge ideals are radical \cite{HerzogEtAl}. In the process, Herzog et al. re-discovered unit-interval graphs, characterizing them as the graphs whose binomial edge ideals have quadratic Gr\"obner bases with respect to a diagonal term order \cite[Theorem 1.1]{HerzogEtAl}. 

So far, we sketched three graph-theoretic topics from three different centuries: Hamiltonian paths, (unit) interval graphs, binomial edge ideals. In the last years, there has been an increasing interest in expanding these three notions to higher dimensions. Specifically:
\begin{compactitem}
\item Katona--Kierstead \cite{KatKie} and many others \cite{HS, Ketal, RSR} have studied  ``tight Hamiltonian paths'' and ``loose Hamiltonian paths'' in $d$-dimensional simplicial complexes; both notions for $d=1$ boil down to ordinary Hamiltonian paths.  The good news is that extremal combinatorics provides a non-trivial way to extend Dirac's theorem for $d$-complexes with a very large number of vertices that satisfy certain ridge-degree conditions. The bad news is that already Ore and P\'osa's theorems seem very hard to extend. 
\item Ene et al. \cite{EneEtAl} introduced  ``determinantal facet ideals'', which directly generalize binomial edge ideals, and ``closed $d$-complexes'', which generalize  `unit-interval graphs'. The good news is that the definitions are rather natural. The bad news is that determinantal facet ideals are not radical in general (see Example \ref{ex:notradical}), and they are hard to manipulate; alas,  the two main results of the paper \cite{EneEtAl} are  incorrect, cf.~Remark \ref{rem:correction}.
\end{compactitem}

\medskip 
In the present paper we take a new, unified look at these approaches. In Chapter 1, we introduce a notion of `weakly-Hamiltonian paths' for $d$-dimensional simplicial complexes that for $d=1$ also boils down to ordinary Hamiltonian paths. This weaker notion enables us to obtain a first, partial extension of Dirac, Ore and P\'osa's theorem to higher dimensions:

\begin{thmnonumber}[Higher-dimensional Ore and Dirac, cf.~Proposition~\ref{prop:Ore1} and Corollary \ref{cor:Dirac}] Let $\Delta$ be any traceable $d$-complex on $n>2d$ vertices. If in some labeling that makes $\Delta$ traceable the two $(d-1)$-faces $\sigma$ and $\tau$ formed by the first $d$ and the last $d$ vertices, respectively, have facet degrees summing up to at least $n$, then $\Delta$ admits a weakly-Hamiltonian cycle.\\
In particular, if in a traceable pure $d$-complex with $n$ vertices, every $(d-1)$-face belongs to  at least $\frac{n}{2}$ facets, then the complex admits a weakly-Hamiltonian cycle.
\end{thmnonumber}

\begin{thmnonumber}[Higher-dimensional P\'osa, cf.~Proposition~\ref{prop:PosaD}] \label{prop:PosaD} Let $\Delta$ be any traceable pure $d$-complex on $n$ vertices, $n > 2d$. Suppose that with any labeling in which $\Delta$ has a weakly-Hamiltonian path, $\Delta$ is traceable.
 Let $\sigma_1, \sigma_2, \ldots, \sigma_{s}$ be the $(d-1)$-faces of  $\Delta^*$, ordered so that $d_1 \le d_2 \le \ldots \le d_{s}$, where $d_i \eqdef d_{\sigma_i}$ is  the number of $d$-faces containing $\sigma_i$. If for every $d \le  k < \frac{n}{2}$ one has  $ d_{k-d+1} >k$,
then $\Delta$ admits a weakly-Hamiltonian cycle.
\end{thmnonumber}

As you can see these results are conditional: `Traceability', i.e. the existence of a \emph{tight} Hamiltonian path, must be known a priori, in order to infer the existence of a \emph{weakly-}Hamiltonian cycle. This sounds like a bad deal, but  in the one-dimensional case our results above still immediately imply the original theorems by Ore and P\'osa for graphs. Moreover, since no extremal combinatorics is used in the proof, there is an advantage: Main Theorems I and II do not require the number of vertices to be extremely large. On the contrary: In the two-dimensional case, they already apply to complexes with five vertices. %Also, %We do not know if it is possible to remove or weaken the traceability assumption from Main Theorems I and II.

\medskip
In Chapter 2, we introduce a hierarchy of four natural properties that progressively weaken (for strongly-connected complexes) the notion of ``closed $d$-complexes'', as originally proposed in \cite{EneEtAl}. We introduce ``unit-interval'', ``under-closed'', and ``weakly-closed'' complexes, as natural combinatorial higher-dimensional generalizations of unit-interval graphs, of interval graphs, and of co-comparability graphs, respectively. The forth property, called ``semi-closed'', is intermediate between ``under-closed'' and ``weakly-closed''; it is also defined very naturally, but it seems to be new already for graphs. We will see its algebraic consequence in Main Theorem VI below. 
The main goal of Chapter 2 is to connect this hierarchy to the notions of Chapter 1:

\begin{thmnonumber}[Higher-dimensional Bertossi, Theorem \ref{thm:CTSC}] Every unit-interval strongly-connected $d$-dimensional simplicial complex is traceable.
\end{thmnonumber}

\begin{thmnonumber}[Higher-dimensional Chen--Chan--Chang, Theorem  \ref{thm:Hi0}] Every unit-interval $d$-dimensional simplicial complex that remains strongly connected after the deletion of~$d$ or less vertices, however chosen, is Hamiltonian.
\end{thmnonumber}

Finally, Chapter 3 is dedicated to the connection with commutative algebra. For a homogeneous ideal of polynomials, having a square-free Gr\"obner degeneration is a strong and desirable property. In 2020, Conca and the third author proved Herzog's conjecture that if a homogeneous ideal $I$ has a square-free initial ideal $\operatorname{in}(I)$, then the extremal Betti numbers of $I$ and $\operatorname{in}(I)$ are the same \cite{CV}. This allows us to infer the depth, the Castelnuovo--Mumford regularity, and many other invariants of the ideals $I$ with squarefree initial-ideal, simply by computing these invariants on the initial ideal --- which is a much simpler task, because the aforementioned Stanley--Reisner correspondence activates techniques from combinatorial topology. 
Building on  the very recent work of the second author \cite{Se2}, we are able to revise one of the results claimed in Ene et al \cite{EneEtAl} as follows:

\begin{thmnonumber}[Theorem \ref{t:a-c-gb} and \ref{t:a-c-gb1}]
A strongly-connected $d$-dimensional simplicial complex $\Delta$ is unit-interval if and only if the complex is traceable and with respect to the same labeling, the minors defining the determinantal facet ideal of $\Delta$ form a Gr\"obner basis with respect to any diagonal term order.
\end{thmnonumber}

We conclude our work with a result that provides a broad class of determinantal facet ideals that are radical, and even $F$-pure (if the characteristic is positive):

\begin{thmnonumber}[Theorem \ref{t:s-c-f}] The determinantal facet ideals of all semi-closed complexes are radical. Indeed, they have a square-free initial ideal with respect to any diagonal term order. Moreover, in characteristic $p>0$, the quotients by these ideals are all $F$-pure.
\end{thmnonumber}

The proof relies once again on the recent work by the second author \cite{Se2}. Since all shifted complexes are under-closed, and in particular semi-closed, Theorem \ref{t:s-c-f} immediately implies that the determinantal facet ideals of shifted complexes admit a square-free Gr\"obner degeneration and, in positive characteristic, define $F$-pure rings. 
As a consequence of Main Theorem VI, we can extend to all dimensions the result by Kiani and Saeedi-Madani that ``among all graphs, only complete graphs have a binomial edge ideal with a linear resolution'' \cite{SMK}.
Namely, we prove that among all $d$-dimensional simplicial complexes with $n$ vertices, only the $d$-skeleta of simplices have a determinantal facet ideal with a linear resolution (Corollary \ref{cor:SMK}).

\section*{Notation}
Throughout $d, n$ are positive integers, with $d < n$. We denote by $\Sigma^d$ the $d$-simplex, and by $\Sigma^d_n$ the $d$-skeleton of $\Sigma^{n-1}$. %the simplex with vertex set $[n]$.
We write each face of  $\Sigma^d_n$ by listing its vertices in increasing order. We describe simplicial complexes by listing their facets in any order, e.g. $ \Delta = 123, 124, 235$. For any $d$-face $F=a_0a_1\cdots a_d$ of $\Sigma_n^d$, we call \emph{gap of $F$} the integer $\operatorname{gap}(F) \eqdef a_d - a_0 - d$, which counts the integers $i$ strictly between $a_0$ and $a_d$ that are not present in $F$.
For each $i$ in $\{1, 2, \ldots, n-d\}$, we call $H_i$ the $d$-face of $\Sigma_n^d$ with vertices $i, i+1, \ldots, i+d$. Clearly, $H_1, H_2, \ldots, H_{n-d}$ are exactly those faces of $\Sigma_n^d$ that have gap zero. With abuse of notation, we extend the definition of $H_i$ also to $i \in \{n-d+1, \ldots, n\}$ using ``congruence modulo $n$''. Namely, by ``$n+1$'' we mean vertex $1$, by ``$n+2$'' we mean vertex $2$, and so on. So $H_n$ will be the $d$-face adjacent to $H_1$ and of vertices $\{n, 1, 2, 3, \cdots, d\}$, which we write down in increasing order, so $H_n =123\cdots d \, n$. Note that $\operatorname{gap}(H_i)>0$ when $i>n-d$.

\begin{definition}[traceable, Hamiltonian] \label{def:1}
A  complex $\Delta$ is \emph{(tight-) traceable} if it has a labeling such that $H_1, \ldots, H_{n-d}$ are in $\Delta$. %In this case, we call $H_1, \ldots,H_{n-d}$ a  \emph{Hamiltonian path}. 
It is  \emph{(tight-) Hamiltonian} if it has a labeling such that all of $H_1, \ldots, H_n$ are in $\Delta$. %In this case, we call $H_{1}, \ldots, H_{n}$ a \emph{Hamiltonian cycle}.
\end{definition}

Clearly, Hamiltonian implies traceable. For $d=1$, Definition~\ref{def:1} boils down to the classical notions of traceable and Hamiltonian graphs, that is, graphs that admits a Hamiltonian path and a Hamiltonian cycle, respectively. In fact, nobody prevents us from relabeling the vertices in the order in which we encounter them along such path (or cycle).

Recall that two facets of a pure simplicial $d$-complex are \emph{adjacent} if their intersection has cardinality $d$, or equivalently, dimension $d-1$. For example, each $H_i$ is adjacent to $H_{i+1}$. The \emph{dual graph} of a pure simplicial $d$-complex $\Delta$ has nodes corresponding to the facets of $\Delta$; two nodes are connected by an arc if and only if the corresponding facets of $\Delta$ are adjacent. 
A pure simplicial $d$-complex $\Delta$ is \emph{strongly-connected} if its dual graph is connected. For $d \ge 1$, every strongly-connected $d$-complex is connected, and when $d=1$ the two notions coincide.  According to our convention, all strongly-connected simplicial complexes are pure.

\begin{remark} \label{rem:HnotSC} The statement ``the dual graph of any Hamiltonian $d$-complex is Hamiltonian" holds true only for $d=1$: For example, the Hamiltonian  simplicial complex
\[\Delta_1= 123, 234, 345, 456, 567, 678, 789, 189, 129, 147 \] 
is not even strongly connected, because the facet $147$ is isolated in the dual graph. The deletion of vertex $1$ from $\Delta_1$ yields a simplicial complex that is not even pure.
\end{remark}

%Every traceable $d$-dimensional complex is obviously connected, but not necessarily strongly-connected, as shown by
%$\Delta = \{ 123, 234, 345, 456, 567 \} \cup \{147\}$.
%Strongly-connected complexes need not be traceable, as shown by the claw graph. 

\newpage

\section{Weakly-traceable/Hamiltonian complexes and ridge degrees}
%\section{An Ore-type and a P\'osa-type resul
In this section, we introduce two weaker notions of traceability and Hamiltonicity that first appeared in \cite{Ketal}, and we study their nontrivial relationship with the ``ridge degree'', i.e. how many $d$-faces contain any given $(d-1)$-face. This relationship has a long history, beginning in 1952 with one of the most classical results in graph theory, due to Gabriel Dirac \cite{Dirac}, the son of Nobel Prize physicist Paul Dirac:

\begin{theorem}[Dirac \cite{Dirac}] Let $G$ be a graph with $n$ vertices. 
If $\deg v \ge \frac{n}{2}$ for every vertex~$v$, then $G$ is Hamiltonian. 
\end{theorem}

Later {\O}ystein Ore \cite{Ore} improved Dirac's result and extended it to traceable graphs: %, Posa, and Chvatal:

\begin{theorem}[Ore \cite{Ore}] \label{thm:OreGraphs} Let $G$ be a graph with $n$ vertices. 
\begin{compactenum}[ \rm (A) ] 
\item  If   $\deg u + \deg v \ge n$ for all non-adjacent vertices $u,v$, the graph $G$ is Hamiltonian. 
\item If  $\deg u + \deg v \ge n-1$ for all non-adjacent vertices $u,v$, the graph $G$ is traceable.
\end{compactenum}
\end{theorem}

Two years later P\'osa extended Ore's condition (A) much further:

\begin{theorem}[P\'osa \cite{Posa}] \label{thm:PosaGraphs} Let $G$ be a graph with $n$ vertices. Order the vertices $v_1, \ldots, v_n$ so that the respective degrees are weakly increasing, $d_1 \le d_2 \le \ldots \le d_n$. 
\begin{compactenum}[ \rm (C) ] 
\item If for every $k < \frac{n}{2}$ one has $d_k > k$, the graph $G$ is Hamiltonian. 
\end{compactenum}
\end{theorem}

These theorems have been generalized in five main directions, over the course of more than a hundred papers (see also Li \cite{Li} for a survey with a different perspective than ours):
\begin{compactenum}
\item %Bondy \cite{Bondy}, Chv\'atal \cite{Chvatal}, and Bondy--Chv\'atal \cite{BondyChvatal} 
Bondy and Chv\'atal  \cite{BondyA, Bondy, Chvatal, BondyChvatal}  weakened the antecedent in the implication (C) of P\'osa's theorem (see \cite{Farrugia} for an application to self-complementary graphs);
\item Bondy \cite{BondyB} strengthened the conclusion of Ore's theorem, from Hamiltonian to \emph{pancyclic} (=containing cycles of length $\ell$ for any $3 \le \ell \le n$); later Schmeichel--Hakimi \cite{SchmeichelHakimi} showed that P\'osa and Chv\'atal's theorems can be strengthened in the same direction;
\item Fan  \cite{Fan} showed that for $2$-connected graphs, it suffices to check Ore's condition for vertices $u$ and $v$ at distance $2$; and even more generally, it suffices to check that for any two vertices at distance two, at least one of them has degree $\ge \frac{n}{2}$.  With these weaker assumptions he was still able to achieve a pancyclicity conclusion. See  \cite{BedrossianChenSchelp}, \cite{LiLiFeng}, \cite{ChaoSongZhang} for recent extensions of Fan's work.
\item A forth line of generalizations of Ore's theorem involved requiring certain vertex sets to have large neighborhood unions, rather than large degrees: Compare Broersma--van den Heuvel--Veldman  \cite{BHV} and Chen--Schelp \cite{ChenSchelp}. 
\end{compactenum}
Here we are interested in the \emph{fifth} main direction, namely, the generalization to higher dimensions. This is historically a rather difficult task: As of today, no straightforward extension of Ore's theorem or of P\'osa's theorem  is known. However, some elegant positive results were obtained in 1999 by Katona and Kierstead \cite{KatKie}, who applied extremal graph theory to generalize Dirac's theorem to simplicial complexes with a huge number of vertices. Building on the work by Katona and Kierstead \cite{KatKie}, R\"odl, Szemer\'edi, and Ruci\`nski \cite{RSR} were able in 2008 to prove the following `extremal' version of Dirac's theorem:

 \begin{theorem}[{R\"odl--Szemer\'edi--Ruci\'nski \cite{RSR}}] \label{thm:RSR}
For all integers $d \ge 2$ and for every $\varepsilon > 0$ there exists a (very large) integer $n_0$ such that every 
$d$-dimensional simplicial complex $\Delta$ with more than $n_0$ vertices, and such that every $(d-1)$-face of $\Delta$ is in at least
$n ( \frac{1}{2} + \varepsilon )$
facets, is Hamiltonian.
\end{theorem}

Now we are ready to introduce the main definition of the present section. Recall that two facets of a pure simplicial $d$-complex are \emph{incident} if their intersection is nonempty. 

\begin{definition}[weakly-traceable, weakly-Hamiltonian] \label{def:WH}
A $d$-dimensional simplicial complex $\Delta$ is \emph{weakly-traceable} if  if it has a labeling such that $\Delta$ contains faces $H_{i_1}, \ldots, H_{i_k}$ from $\{H_1, \ldots, H_{n-d}\}$ that altogether cover all vertices, and such that $H_{i_j}$ is incident to $H_{i_{j+1}}$ for each $j \in \{1, \ldots, k-1\}$. In this case, we call the list $H_{i_1}, \ldots, H_{i_k}$ a \emph{weakly-Hamiltonian path}.\\
A $d$-dimensional simplicial complex $\Delta$ is \emph{weakly-Hamiltonian} if it has a labeling such that $\Delta$ contains faces $H_{i_1}, \ldots, H_{i_k}$ from $H_1, \ldots, H_n$ that altogether cover all vertices, such that $H_{i_j}$ is incident to $H_{i_{j+1}}$ for each $j \in \{1, \ldots, k-1\}$, and in addition $H_{i_k}$ is incident to $H_{i_1}$. In this case, we call the list $H_{i_1}, \ldots, H_{i_k}$ a \emph{weakly-Hamiltonian cycle}.
\end{definition}

\begin{remark} These notions are not new. For what we called ``weakly-Hamiltonian'', Keevash et al. \cite{Ketal} use the term ``generic Hamiltonian''. Their paper  \cite{Ketal} focuses however on the stronger notion of ``loose-Hamiltonian'' complexes, which are weakly-Hamiltonian complexes where all of the intersections $H_{i_j} \cap H_{i_{j+1}}$ consist of a single point (with possibly one exception). By definition, all Hamiltonian complexes are loose-Hamiltonian, and all loose-Hamiltonian complexes are weakly-Hamiltonian. For $d=1$ all these different notions converge: ``Weakly-Hamiltonian $1$-complexes'' are simply ``graphs with a Hamiltonian cycle'', and  ``weakly-traceable $1$-complexes'' are ``graphs with a Hamiltonian path''.  In 2010 Han--Schacht  \cite{HS} and independently Keevash et al.  \cite{Ketal} proved the following extension of Theorem \ref{thm:RSR} above:
\begin{theorem}[{Han-Schacht \cite{HS}, Keevash et al. \cite{Ketal}}]
For all integers $d \ge 2$ and for every $\varepsilon > 0$ there exists a (very large) integer $n_0$ such that every 
$d$-dimensional simplicial complex $\Delta$ with more than $n_0$ vertices, and such that every $(d-1)$-face of $\Delta$ is in at least
$n (\frac{1}{2d} + \varepsilon)$
facets, is loose-Hamiltonian, and in particular weakly-Hamiltonian.
\end{theorem}
\end{remark}

\begin{remark} In Definition \ref{def:WH}, note that if $\Delta$ is weakly-traceable, necessarily $i_1 = 1$ and $i_k = n-d$, because otherwise $1$ and $n$ would not be covered. So equivalently, in Def.~\ref{def:WH} we could demand \[ \{i_2, \ldots, i_{k-1}\} \subset \{2, \ldots, n-d-1\}.\] Note also that  if a labeling $v_1, \ldots, v_n$ makes $\Delta$ (weakly-) traceable, so does the ``reverse labeling'' $v_{n}, \ldots, v_{1}$.
As for Hamiltonian complexes: If a labeling $v_1, \ldots, v_n$ makes $\Delta$ weakly-Hamiltonian, so does its reverse, and also $v_{i_1}, \ldots, v_{i_n}$, where $(i_1, \ldots, i_n)$ is any cyclic permutation of $(1, \ldots, n)$. So we may assume that $i_1=1$. Or we may assume that $i_k=n-d$. But as the next remark shows, we cannot assume both.
\end{remark}

\begin{remark} \label{rem:surprise}
When $d>1$, not all weakly-Hamiltonian $d$-complexes are weakly-traceable. For $d=2$, a simple counterexample is given by
\[ \Delta_0 = 123, 156, 345.\] 
The weakly-Hamiltonian cycle is of course $H_1, H_3, H_5$. Any labeling that makes $\Delta_0$ weakly-Hamiltonian is either the reverse or a cyclic shift (or both) of the labeling above. For parity reasons, in any labeling that makes $\Delta_0$ weakly-Hamiltonian, only one of $H_1$ and $H_4$ is in $\Delta_0$. \end{remark} 

\begin{remark} \label{rem:An}
Weakly-traceable complexes are obviously connected. Weakly-Hamiltonian complexes are even $2$-connected, in the sense that the deletion of any vertex leaves them connected.  The converses are well-known to be false already for $d=1$. In fact, let $n \ge 4$. Let $A_{n-2}$ be the edge-less graph on $n- 2$ vertices. Let $x, y$ be two new vertices. The ``suspension''
 \[ \operatorname{susp}(A_{n-2}) \ \eqdef \  A_{n-2} \: \cup \:  \{ x \ast v \, : \: v \in A_{n-2}\}  \: \cup \: \{ y \ast v \, : \: v \in A_{n-2}\} \]
is a $2$-connected graph on $n$ vertices that is not Hamiltonian for $n \ge 5$, and not even traceable for $n \ge 6$. In higher dimensions, the $\Delta^d_3$ of Lemma \ref{lem:WheelComplex} is $d$-connected, but neither weakly-traceable nor weakly-Hamiltonian. 
\end{remark}

We start with a few Lemmas that are easy, and possibly already known; we include nonetheless a proof for the sake of completeness. For the following lemma, a \emph{subword} of a word is a subsequence formed by \emph{consecutive} letters of a word: So for us ``word'' is a subword of ``subword'', whereas ``sword'' is not.

\begin{lemma} Let $d \ge 2$. If a $d$-complex $\Delta$ is weakly-Hamiltonian (resp. weakly traceable), then for any $k \in \{1,  \ldots, d\}$ the $k$-skeleton of $\Delta$ is weakly-Hamiltonian (resp. weakly-traceable). 
\end{lemma}

\begin{proof}
Given a weakly-Hamiltonian path/cycle, replace any $d$-face $H_1$ with its $(k+1)$-letter subwords, ordered lexicographically. The result, up to canceling possible redundancies, will be a weakly-Hamiltonian path/cycle for the $k$-skeleton. 
\end{proof}

For example: if $d=3$ and $k=1$, suppose that a $3$-complex on $8$ vertices admits the Hamiltonian path \[1234, \quad 2345, \quad 5678.\]  
Then the $1$-skeleton admits the Hamiltonian path
\[ 12, \, 23, \, 34, \quad \cancel{ 23, } \, \, \cancel{ 34, } \, 45, \quad \, 56, \, 67, \, 78.\]
The next Lemma is an analog to the fact that Hamiltonian complexes are traceable. 

\begin{lemma} \label{lem:deletion} Let $\Delta$ be a $d$-dimensional complex that has a weakly-Hamiltonian cycle $H_{i_1}, \ldots, H_{i_k}$, with $k \ge 3$. For any $j$ in $\{1, \ldots, k\}$, let $m_j$ be the number of vertices of $H_{i_j}$ that are neither contained in $H_{i_{j-1}}$ nor in  $H_{i_{j+1}}$ (where by convention $i_{k+1} \eqdef i_1$).
\begin{compactitem}
\item If $m_j >0$, the deletion of  those $m_j$ vertices from $\Delta$ yields a weakly-traceable complex. 
\item If $m_j=0$, and in addition  $H_{i_{j-1}}$ and  $H_{i_{j+1}}$ are disjoint, then $\Delta$ itself is weakly-traceable.
\end{compactitem}
\end{lemma}

\begin{proof} Fix $j$ in $\{1, \ldots, k\}$. If $m_j >0$, the $m_j$ vertices that belong to $H_{i_j}$ and to no other facet of the cycle are labeled consecutively. So up to relabeling the vertices cyclically, we can assume that they are the vertices $n-m_j + 1, \, n- m_j + 2, \ldots, n-1, n.$ Thus the facet in the cycle they all belong to is the last one, $H_{i_k}$. Now let $D$ be the complex obtained from $\Delta$ by deleting these $m_j$ vertices. It is easy to see that 
\[ H_{1}=H_{i_1}, H_{i_2}, \ldots, H_{i_{k-1}}\]
is a weakly-Hamiltonian path for $D$. \\The case $m_j=0$ is similar: Up to relabeling the vertices cyclically, $i_{j+1}=1$ and thus $j=k$. By assumption  $H_{i_{k-1}}$ and  $H_{1}$ are disjoint. But since $m_k =0$, and vertex $n$ does not belong to $H_1$, it must belong to $H_{i_{k-1}}$. Therefore $H_{i_{k-1}}=H_{n-d}$. So
\[ H_{1}=H_{i_1}, H_{i_2}, \ldots, H_{i_{k-1}}\]
is a weakly-Hamiltonian path for $\Delta$ itself.
\end{proof}

The next Lemma can be viewed as a $d$-dimensional extension of the fact that the cone over the vertex set of a graph $G$ is a Hamiltonian graph if and only if the starting graph $G$ is traceable.
 
\begin{lemma} \label{lem:buciodiculo2} Let $\Delta$ be any $d$-complex on $n$ vertices.  Let $\Sigma^{d-1}$ be the $(d-1)$-simplex. 
Let $\Gamma$ be the $d$-complex obtained by adding to $\Delta$ a  $d$-face $v \ast \Sigma^{d-1}$ for every vertex $v$ in $\Delta$. Then 
\[ \Delta \textrm{ is weakly-traceable }  \Longleftrightarrow \Gamma \textrm{ is weakly-Hamiltonian. }\] 
\end{lemma}

\begin{proof} 
``$\Rightarrow$'': If $H_{i_1}, \ldots, H_{i_k}$ is a list of facets proving that $\Delta$ is weakly-traceable, then the list $H_{i_1}, \ldots, H_{i_k}, H_n, H_{n+1}$ shows that $\Gamma$ is weakly-Hamiltonian.\\
 ``$\Leftarrow$'':  Pick a labeling that makes $\Gamma$ weakly-Hamiltonian. By how the complex $\Gamma$ is constructed, the vertices of $\Sigma^{d-1}$ must be labeled consecutively; so without loss, we may assume that they are $n+1, \ldots, n+d$. Take a weakly-Hamiltonian cycle for $\Gamma$ and delete from the list all the $d$-faces containing any vertex whose label exceeds $n$. 
\end{proof}

\begin{remark}  \label{rem:trouble} The following statements are valid \textbf{only for $\mathbf{d=1}$}.
\begin{compactenum}[(i)]
\item ``$\Delta$ is weakly-traceable $ \Longleftrightarrow$  $\Delta \:  \cup  \: \: \: w \ast \operatorname{(d-1)-skel}(\Delta)$  is weakly-Hamiltonian.''
\item  ``Deleting a single vertex from a weakly-Hamiltonian $d$-complex yields a weakly-traceable complex.''
\item ``Deleting (the interior of) any of the $H_i$'s from a weakly-Hamiltonian $d$-complex yields a weakly-traceable complex.'' 
\end{compactenum}
Simple counterexamples in higher dimensions are:
\begin{compactenum}[(i)]
\item $\Delta_1=126, 234, 456, 489, 678$ is not weakly-traceable, yet
$\Delta_2 \eqdef  \Delta_1 \, \cup \, \left( 10 \ast \, \operatorname{2-skel}(\Delta_1) \right)$ admits  the weakly\--Hamiltonian cycle $234, \; 456, \; 678, \; 89 \, 10, \; 12\, 10$.
This is a counterexample to ``$\Leftarrow$''. $\, $ In contrast, the direction ``$\Rightarrow$'' holds in all dimensions.
\item If from the $\Delta_2$ above we delete vertex $10$, we get back to $\Delta_1$, not weakly-traceable.
\item $\Delta_3=1234, \; 2345, \; 5678, \; 167 \,10, \; 189\, 10$ is weakly-Hamiltonian, but the deletion of (the interior) of $5678$ yields a complex that is not weakly-traceable.
\end{compactenum}
\end{remark}

Our first non-trivial result is an ``Ore-type result'': We shall see later that in some sense it extends `most' of the proof of Ore's theorem  \ref{thm:OreGraphs}, part (A), to all dimensions.

\begin{definition} Let  $\Delta$ be a pure $d$-dimensional simplicial complex, and let $\sigma$ be any $(d-1)$-face of $\Delta$. The \emph{degree} $d_{\sigma}$ of $\sigma$ is the number of $d$-faces of $\Delta$ containing $\sigma$.
\end{definition}

\begin{proposition}\label{prop:Ore1}
Let $\Delta$ be a traceable $d$-dimensional  simplicial complex on $n$ vertices, $n > 2d$.
If in some labeling that makes $\Delta$ traceable the two $(d-1)$-faces $\sigma$ and $\tau$ formed by the first $d$ and the last $d$ vertices, respectively, satisfy 
$d_{\sigma} + d_{\tau} \ge n$, then $\Delta$ is weakly-Hamiltonian.  
\end{proposition}

\begin{proof}Since $n > 2d$, the two faces $\sigma$ and $\tau$ are disjoint.  Let $J \eqdef \{d+2, d+3, \ldots, n-d\}$.  For every $i$ in $J$, which has cardinality $n-2d-1$, consider the two $d$-faces of $\Sigma^d_n$ 
\[ S_i \eqdef \sigma \ast i \quad \textrm{ and } \quad T_i \eqdef (i-1) \ast \tau.\] 

\begin{figure}[htbp]
\begin{center}
\includegraphics[scale=0.64]{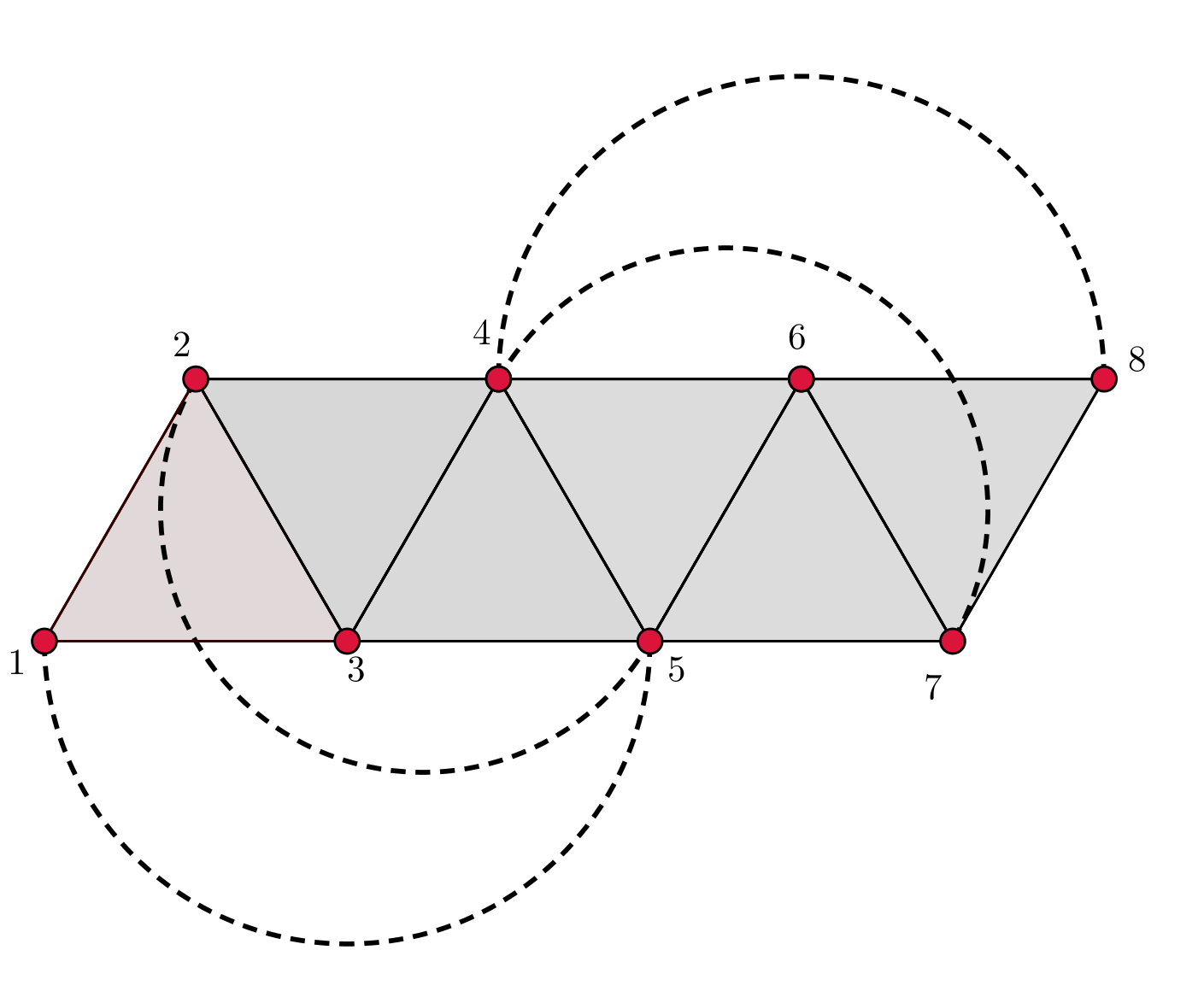} 
\includegraphics[scale=0.64]{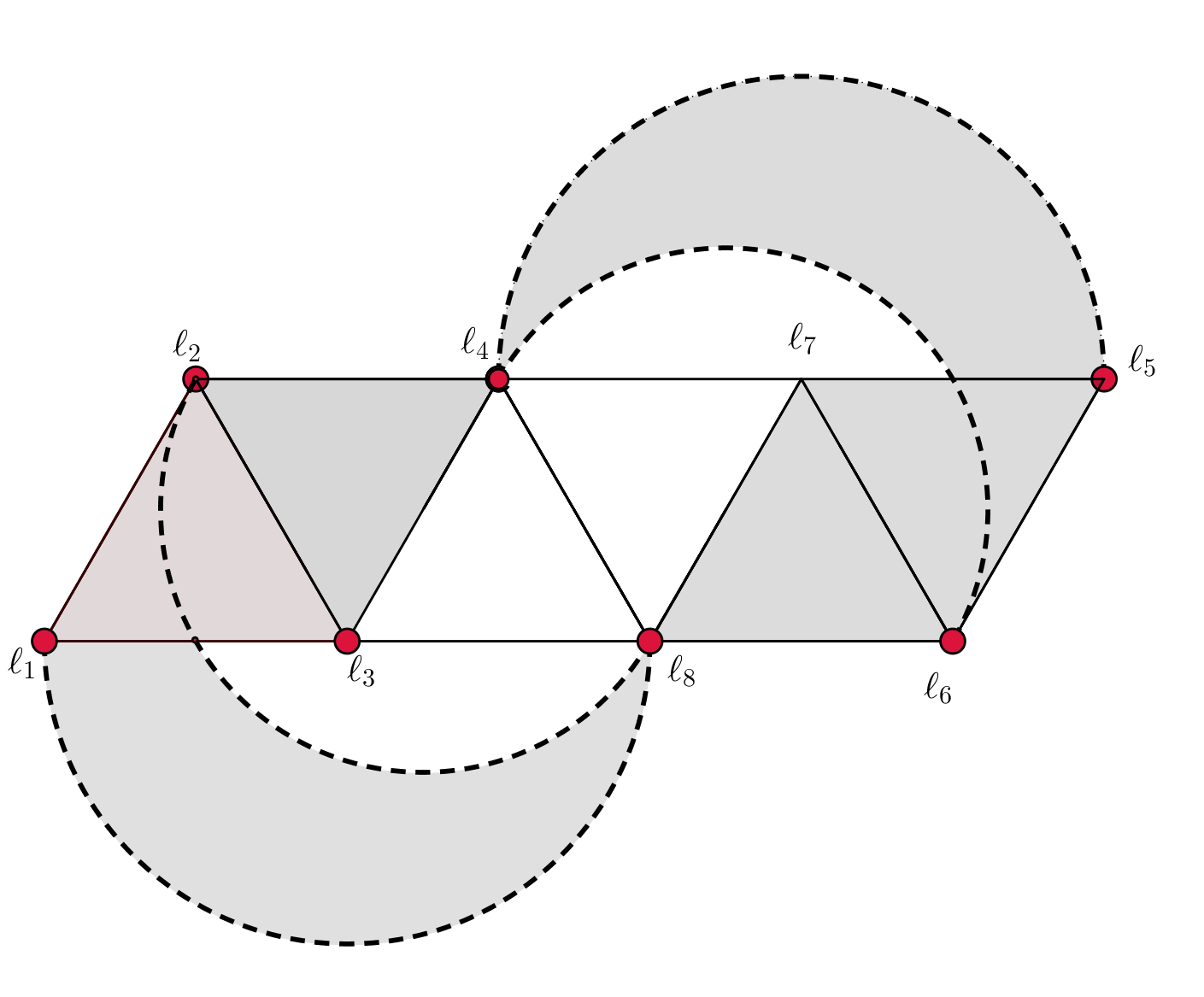}
\label{fig:OreA}
\caption{\textsc{left}: The dashed triangles are $S_5$ and $T_5$.  Were they both in $\Delta$, then one could relabel the vertices and create a weakly-Hamiltonian cycle (\textsc{right}).}
\end{center}
\end{figure}

Now there are two cases, both of which will result in a weakly-Hamiltonian cycle:  

\smallskip
\textsc{Case 1: } \emph{For some $i$, both $S_i$, $T_i$ are in $\Delta$.} 
We are going to introduce a new vertex labeling $\ell_1, \ldots, \ell_n$. The  ``consecutive facets of the new labeling'' will be called $L_1=\ell_1 \cdots \ell_d \ell_{d+1}$,  $L_2=\ell_2 \cdots \ell_{d+2}$, and so on. The following describes a weakly-Hamiltonian cycle:
\begin{compactitem}
\item Start with the first $i-1$ vertices in the same order: That is, set $\ell_1 \eqdef  1$, $\ldots$, $\ell_{i-1} \eqdef i-1$. Hence $L_1=H_1, L_2=H_2, \ldots, $, up until $L_{i-d-1}=H_{i-d-1}$, which (since $\Delta$ is traceable)  is the first of the $H_i$'s that contains the vertex $i-1$. 
\item Then set $L_{i} \eqdef T_i$. The vertices of $\tau$ are to be relabeled by $\ell_{i}, \ell_{i+1}, \ldots, \ell_{i+d}$: Specifically, label by $\ell_i$ the vertex that is in $H_{n-d}$ but not in $H_{n-d-1}$, by $\ell_{i+2}$ the vertex in  in $H_{n-d-1}$ but not in  in $H_{n-d-2}$, and so on. Facet-wise, we are traveling in reverse order across the \emph{last} facets of the original labeling. Stop until you get to relabel vertex $i$ by  $\ell_n$. (Or equivalently, if you prefer to think about facets, stop once you reach facet $H_i$.) 
\item The weakly-Hamiltonian cycle gets then concluded with $S_i$, which is adjacent to $L_1=H_1$ via $\sigma$. The facets previously called $H_{i-d}$, $H_{i-d+1}$, $\ldots$, $H_{i-1}$ are not part of the new weakly-Hamiltonian cycle.
\end{compactitem}

\smallskip
\textsc{Case 2: } \emph{For all $i$, at most one of $S_i$, $T_i$ is in $\Delta$.} Since the two sets 
$\{ i \in J : \sigma * i \in \Delta\}$ and $\{ i \in J : (i-1) * \tau \in \Delta\}$ are disjoint, the sum of their cardinalities is the cardinality of their union, which is contained in $J$. So
\begin{equation} \label{eq:Ore1}   |\{ i \in J : \sigma * i \in S\}| +   |\{ i \in J : (i-1) * \tau \in \Delta\}|  %=  |\{ i \in J : \sigma * i \in S\} \:  \cup \:   \{ i \in J : (i-1) * \tau \in \Delta\}|  \ 
\le \ | J | \; = \; n-2d+1.\end{equation}
Now, we claim that either $\sigma \ast n$ or  $1\ast \tau$ is a face of $\Delta$. From the claim the conclusion follows immediately, as such face creates a weakly-Hamiltonian cycle.  We prove the claim by contradiction. Suppose $\Delta$ contains neither  $\sigma \ast n$ nor  $1\ast \tau$. Every $d$-face containing $\sigma$ is of the form $\sigma \ast v$, where $v$ is either in $J$ or in the set $\{d, n-d+1, n-d+2, \ldots, n-1\}$ (which has size $d$).  So
\begin{equation} \label{eq:Ore2} d_{\sigma}  \le |\{ i \in J : \sigma * i \in \Delta\}|  \; +  d.\end{equation}
Symmetrically, the $d$-faces containing $\tau$ are of the form $w \ast \tau$, with $w$ either in $J$ or in the size-$d$ set $\{2, 3, \ldots, d, n-d+1\}$. % ($1 \ast \tau$ cannot be a face) 
So
\begin{equation} \label{eq:Ore3} d_{\tau}  \le |\{ i \in J : (i-1) * \tau \in \Delta\}| \; +  d.\end{equation}
 Putting together inequalities  \ref{eq:Ore1},  \ref{eq:Ore2} and \ref{eq:Ore3}, we reach a contradiction:
\[d_{\sigma}  + d_{\tau} \le  (n - 2d - 1) + d + d  = n-1. \qedhere \]
\end{proof}

\begin{corollary}\label{cor:Ore1}
Let $\Delta$ be a traceable $d$-dimensional  simplicial complex on $n$ vertices, $n > 2d$.
If for any two disjoint 
$(d-1)$-faces $\sigma$ and $\tau$ one has 
$d_{\sigma} + d_{\tau} \ge n$, then $\Delta$ is  weakly-Hamiltonian.  
\end{corollary}

\begin{corollary} \label{cor:Dirac}
Let $\Delta$ be a traceable  $d$-dimensional  simplicial complex on $n$ vertices, $n > 2d$. 
If every $(d-1)$-face of $\Delta$  belongs to  at least $\frac{n}{2}$
%$\frac{n}{2}-d+1$ 
facets of $\Delta$, then $\Delta$ is  weakly-Hamiltonian. 
\end{corollary}

\begin{example} Let $n > 2d$. Let $\Delta_4$ be the simplicial complex on $n$ vertices obtained from $\Sigma^d_n$ by removing the interior of the %triangles $126$ and $156$. 
$d$-faces $H_{n-d+1}$, $H_{n-d+2}$, $\ldots$, $H_{n}$. By construction $\Delta_4$ is traceable, but the given labeling (as well as any labeling obtained from it by reversing or cyclic shifting) fails to prove that $\Delta_4$ is weakly-Hamiltonian. Now, in $\Delta_4$, the %$d-1$ 
$(d-1)$-faces $ \mu_i \eqdef H_i \cap H_{i+1}$, with  $i \in \{ n-d+1, n-d+2, \ldots, n-1\}$,  have degree $n-d-2$.  All other $(d-1)$-faces $\nu_j$ contained in one of $H_{n-d+1}, H_{n-d+2}, \ldots, H_n$ have degree $n-d-1$.  Finally, all $(d-1)$-faces \emph{not} contained in any of   $H_{n-d+1}, H_{n-d+2}, \ldots, H_n$ have degree $n-d$.  Therefore: 
 \begin{compactitem}
\item If $n \ge 2d +4$, Corollary \ref{cor:Dirac} tells us that $\Delta_4$ is weakly-Hamiltonian, because  $n - d - 2 \ge \frac{n}{2} $.
\item If $n=2d+3$ or $n=2d+2$, 
any two of the $\mu_i$'s are incident, and any $\nu_j$ is incident to all of the $\mu_i$'s. Hence, for any two \emph{disjoint} $(d-1)$-faces $\sigma$ and $\tau$, we do have
$d_{\sigma}  + d_{\tau} \ge 2n - 2d - 2 \ge n$.   So we can still conclude that $\Delta_4$ is weakly-Hamiltonian via Corollary \ref{cor:Ore1}. 

\item If $n=2d+1$, then the assumptions of Corollaries \ref{cor:Dirac} and \ref{cor:Ore1} are not met, but Proposition~\ref{prop:Ore1} is still applicable. In fact, for the facets $\sigma$ resp. $\tau$ formed by the first resp. the last vertices of the given labeling, one has 
$d_{\sigma}  + d_{\tau} = (n - d ) + (n-d-1) = 2n - (2d + 1) = n$.
\end{compactitem}
So in all cases, $\Delta_4$ is weakly-Hamiltonian. The proof of  Proposition \ref{prop:Ore1} also suggests a relabeling that works:
$\ell_1\eqdef 1,\; \ell_2 \eqdef 2, \,  \ldots \,, \; \ell_{d+1}\eqdef d+1,  \; \ell_{d+2}=n, \;  \ell_{d+3}\eqdef n-1,   \,  \ldots \,, \; \ell_{n}\eqdef d+2$.
\end{example}

To see in what sense Proposition \ref{prop:Ore1} is a higher-dimensional version of Ore's theorem \ref{thm:OreGraphs}, part (A), the best is to give a proof of the latter using the former:

\begin{proof}[\bf Proof of Ore's theorem \ref{thm:OreGraphs}, part (A)] By contradiction, let $G$ be a non-Hamiltonian graph satisfying $\deg u + \deg v \ge n$ for all non-adjacent vertices $u,v$. Add edges to it until you reach a \emph{maximal} non-Hamiltonian graph $G^*$. Since any further edge between the existing vertices would create a Hamiltonian cycle, $G^*$ is traceable, and obviously it still  satisfies $\deg u + \deg v \ge n$. %By Proposition \ref{prop:maximal}, 
By Proposition \ref{prop:Ore1} $G^*$ is (weakly-)Hamiltonian, a contradiction. 
\end{proof}

It is possible that the bound of Proposition \ref{prop:Ore1} can be improved. But in any case, the possible improvement could only be small, as the following construction shows.

\begin{nonexample} 
Let $d < m$ be positive integers. Take the disjoint union of two copies $A', A''$ of $\Sigma^d_{m}$. Let $\mu$ be any facet of $\Sigma^d_{m}$ and let $\mu', \mu''$ be its copies in $A'$ and $A''$, respectively. Glue to $A' \cup A''$ a triangulation without interior vertices of the prism $\mu \times [0,1]$, so that the lower face $\mu \times \{0\}$ is identified with $\mu'$, and the upper face  $\mu \times \{1\}$ is identified with $\mu''$. Let $\Delta_5$ be the resulting $d$-complex on $n = 2m$ vertices. This $\Delta_5$ is traceable: the added prism, triangulated as a path of $d$-faces, serves as ``bridge'' to move between the two copies of $\Sigma^d_m$. However, this bridge can only be traveled once, so $\Delta_5$ is not weakly-Hamiltonian. For the labeling that makes it  traceable, $d_{\sigma}+d_{\tau}=(m-d)+(m -d) = n - 2d$. %, which is slightly less than $n$.
\end{nonexample}

Our next result is a ``P\'osa--type'' result, in the sense that it extends most of Nash--Williams' proof \cite{NashWilliams} of P\'osa's theorem \cite{Posa} to all dimensions. We focus on complexes $\Delta$ with the property that any labeling that makes them weakly-traceable, makes them also traceable. Such class is nonempty: for example, it contains all $1$-dimensional complexes and all trees of $d$-simplices (i.e. all triangulations of the $d$-ball whose dual graph is a tree).

\begin{proposition} \label{prop:PosaD} Let $\Delta$ be any traceable pure $d$-complex on $n$ vertices, $n > 2d$. Suppose that any labeling that makes $\Delta$ weakly-traceable makes it also traceable.

Let $\sigma_1, \sigma_2, \ldots, \sigma_{s}$ be an ordering of the $(d-1)$-faces of  $\Delta^*$, 
such that the respective degrees $d_i \eqdef d_{\sigma_i}$ are weakly-increasing, $d_1 \le d_2 \le \ldots \le d_{s}$. 
If for every $d \le  k < \frac{n}{2}$ one has  $ d_{k-d+1} >k$,
then $\Delta$ is weakly-Hamiltonian.
\end{proposition}

\begin{proof} Among all possible labelings that make  $\Delta$ weakly-traceable (and thus traceable, by assumption), choose one that maximizes $d_{\sigma} + d_{\tau}$, where $\sigma$ is the $(d-1)$-face of $H_1$ spanned by the first $d$ vertices (that is, $1,2,\cdots d$) and $\tau$ is the  $(d-1)$-face of $H_{n-d}$ spanned by the last $d$ vertices (that is, $n-d+1, \ldots, n$). Since $n>2d$, the faces $\sigma$ and $\tau$ are disjoint. 
If $d_{\sigma} + d_{\tau} \ge  n$, using the proof of Proposition \ref{prop:Ore1}  we get that $\Delta$ is weakly-Hamiltonian, and we are done.  If not, then one of $\sigma$, $\tau$ has degree  $< \frac{n}{2}$. Up to reversing the labeling, which would swap $\sigma$ and $\tau$, we can assume that $d_{\sigma} < \frac{n}{2}$.
Now let $J \eqdef \{d+2, d+3, \ldots, n-d\}.$
For every $i$ in $J$, which has cardinality $n-2d-1$, consider the two $d$-faces of $\Sigma^d_n$ 
\[ S_i \eqdef \sigma \ast i \quad \textrm{ and } \quad T_i \eqdef (i-1) \ast \tau.\] 
We may assume that at most one of these two faces is in $\Delta$, otherwise a weakly-Hamiltonian cycle arises, exactly as in the proof of Proposition \ref{prop:Ore1}. Now for each $i$ in  $J_1 \eqdef \{ i \in J \: : \:  \sigma \ast i \in \Delta\}$, consider the $(d-1)$-face $\rho_i$ with vertices $\{ i-d, i-d+1, \ldots, i-1\}$.

\begin{figure}[htbp] 
\begin{center}
\includegraphics[scale=0.63]{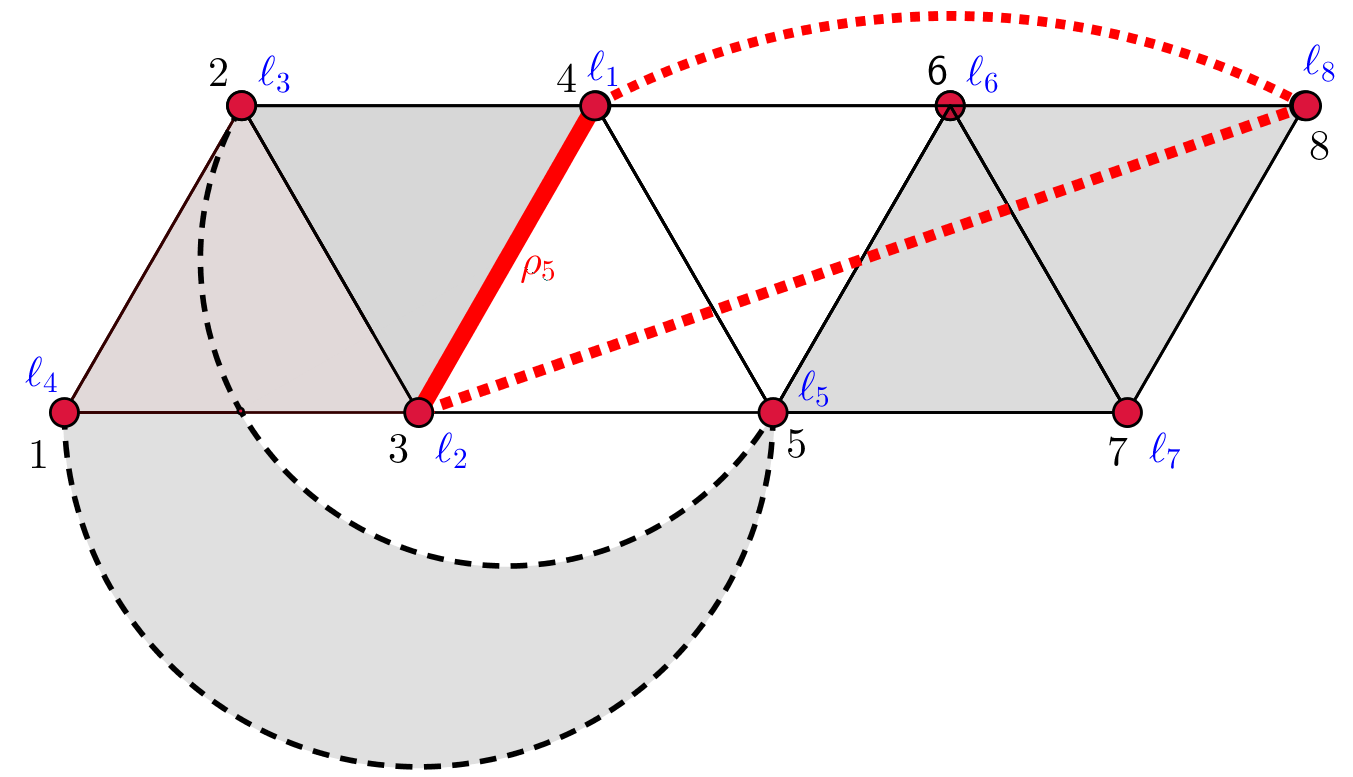} 
\caption{A higher-dimensional ``P\'osa flip'': Since $125$ is in $\Delta$, the vertex $5$ is in $J_1$. Now the red triangle $\rho_5 \ast w = 348$ cannot be in $\Delta$, or else we would have a weakly-Hamiltonian cycle with the blue labeling. The blue labeling makes $\Delta$ weakly-traceable, with $\rho_5$ playing  the role of the ``first'' $(d-1)$-face; the ``last'' $(d-1)$-face is the same as before. By how our original labeling was chosen,  $d_{\rho_5} \le d_{\sigma} < \frac{n}{2}$.}
\label{fig:Posa}
\end{center}
\end{figure}

If for some $i$ in $J_1$ the $d$-face $\rho_i \ast w$ is in $\Delta$, then there is a new relabeling $\ell_1, \ldots, \ell_n$ of the vertices for which we have a weakly-Hamiltonian cycle: see Figure \ref{fig:Posa} above. (The proof is essentially identical  to that of Proposition \ref{prop:Ore1}, up to replacing $T_i$ with $T'_i \eqdef \rho_i \ast w$, reversing the order, and permuting it cyclically, so that $\rho_i$ is the first face.) So also in this case, we are done.

It remains to discuss the case in which for all $i \in  J_1 \eqdef \{ i \in J \: : \:  \sigma \ast i \in \Delta\}$, the $d$-face $\rho_i \ast w$ is not in $\Delta$. In this case the relabeling $\ell_1, \ldots, \ell_n$ introduced above makes $\Delta$ weakly-traceable, and thus traceable by assumption. For such relabeling, the $(d-1)$-faces spanned by the first and the last $d$ vertices are $\rho_i$ and $\tau$, respectively.  So by the way  our original labeling was chosen, $d_{\rho_i} + d_{\tau} \le d_{\sigma} + d_{\tau}$, and in particular
\[d_{\rho_i} \le d_{\sigma} < \frac{n}{2}.\]  
Now, any $d$-face containing $\sigma$ is of the form $\sigma \ast v$, where $v$ is either in the set $J_1$ or in the set $Z \eqdef \{ d+1, n-d+1, \ldots, n-1\}$, which has cardinality $d$. So $d_{\sigma} \le |J_1 \cup Z |$. Since $J$ and $Z$ are disjoint, and $J_1 \subset J$, the sets $J_1$ and $Z$ are also disjoint and we have
\[d_{\sigma}-d = d_{\sigma}- |Z| \le  |J_1 \cup Z | - |Z| = |J_1| + |Z | - |Z| = | J_1 |.\] 
So the set $\{\rho_i \: : \: i \in J_1\}$ contains at least $d_{\sigma} - d\,$ faces of dimension $d-1$ and degree $\le d_{\sigma}$. If we count also $\sigma$, we have in $\Delta$ at least $d_{\sigma} - d +1\,$ faces of dimension $d-1$ and degree $\le d_{\sigma}$. But then, setting $k \eqdef d_{\sigma}$, we obtain
\[ d_{k-d+1} \le k < \frac{n}{2},\]
which contradicts the assumption.
\end{proof}

Again, to see in what sense Proposition \ref{prop:PosaD} is a higher-dimensional version of P\'osa's Theorem~\ref{thm:PosaGraphs}, perhaps the best is to see how easily the latter follows from the former:

\begin{proof}[\bf Proof of P\'osa's theorem \ref{thm:PosaGraphs}] By contradiction, if $G$ is not Hamiltonian, we can add edges to it until we reach a maximal non-Hamiltonian graph $G^*$, which still  satisfies the degree conditions and is traceable.  By Proposition \ref{prop:PosaD}, $G^*$ is (weakly-)Hamiltonian, a contradiction.
\end{proof}

A natural question is whether one can generalize to higher dimensions also part (B) of Ore's theorem \ref{thm:OreGraphs}. The answer is positive, although some extra work is required. In fact, for graphs part (B) of Ore's theorem can be quickly derived from part (A) by means of a coning trick. 
This trick however does not extend to higher dimensions, as we explained in Remark \ref{rem:trouble}, so we'll have to take a long detour, which makes the proof three times as long. The bored reader may skip directly to the next section.

\begin{definition} \label{def:quasi-traceable}
A $d$-dimensional complex $\Delta$ is \emph{quasi-traceable} if there exists a vertex labeling for which $\Delta \cup H_j$ is weakly-traceable, and moreover, with respect to the same labeling,
\begin{compactenum}[ (a) ]
\item if $j=1$, then $\Delta$ contains  all of $H_{2}, \ldots, H_{n-d}$ (i.e., $\Delta \cup H_1$ is traceable);
\item if $j \in \{2, \ldots, n-2d\}$, then $\Delta$ already contains all of $H_1, \ldots, H_{j-1}$ and $H_{j+d}, \ldots, H_{n-d}$ (i.e., $\Delta \cup H_j \cup \ldots \cup H_{j+d-1}$ is traceable);
\item if $j \in \{n-2d+1, \ldots, n-d-1\}$, then $\Delta$ contains all of $H_1, \ldots, H_{j-1}$ and also $H_{n-d}$ (i.e., $\Delta \cup H_j \cup \ldots \cup H_{n-d-1}$ is traceable);
\item if $j=n-d$, then $\Delta$ already contains all of $H_1, \ldots, H_{n-d-1}$  (i.e., $\Delta \cup H_{n-d}$ is traceable).
\end{compactenum}
\end{definition}

\begin{example} The complex $\Delta_6=123, 234, 567, 678, 789$ is quasi-traceable, although not weakly-traceable. In fact, $\Delta_6$ becomes weakly-traceable if we add one of the facets $345$ and $456$, and it becomes even traceable if we add both. 
\end{example}

Definition \ref{def:quasi-traceable} allows the ``added faces'' to be already present in $\Delta$. In particular, all traceable complexes are quasi-traceable.
Here comes our high-dimensional version of Theorem \ref{thm:OreGraphs}, part (B):

\begin{proposition}\label{prop:Ore2}
Let $\Delta$ be a quasi-traceable  $d$-dimensional  simplicial complex on $n$ vertices, $n > 2d$. 
If in some labeling that makes $\Delta$ quasi-traceable the two $(d-1)$-faces $\sigma$ and $\tau$ formed by the first $d$ and the last $d$ vertices  satisfy 
$d_{\sigma} + d_{\tau} \ge n-1$, then $\Delta$ is  weakly-traceable. 
\end{proposition}

\begin{proof} By contradiction, suppose $\Delta$ is not weakly-traceable; we treat the four cases  of Definition \ref{def:quasi-traceable} separately.

\textbf{Case (a)} is symmetric to Case (d), so we will leave it to the reader.
\smallskip

\textbf{Case (b)} is the main case. Since $j \in \{2, \ldots, n-2d\}$, by definition $\Delta$ contains all of $H_1, \ldots, H_{j-1}$ and also $H_{j+d}, \ldots, H_{n-d}$. Since $\Delta$ is not weakly-traceable, it does not contain $H_j$. Moreover, $\sigma \ast (d+j)$ cannot be a facet of $\Delta$, otherwise  the two ``halfpaths'' above would be connected into a weakly-Hamiltonian path. For the same reason, since $(d+j -1) \in H_{j-1}$, the $d$-face $(d+j-1) \ast \tau$ cannot be in $\Delta$. So let 
$J' \eqdef \{d+2, d+3, \ldots, n-d\} \setminus \{ d+j\}$. For every $i$ in $J'$, which has cardinality $n-2d-2$, consider the two $d$-faces of $\Sigma^d_n$
\[ S_i \eqdef \sigma \ast i \quad \textrm{ and } \quad T_i \eqdef (i-1) \ast \tau.\] 

Now there are two subcases:  Either there exists an $i$ such that $S_i, T_i$ are both in $\Delta$, or not.

\smallskip 
\textsc{Case} (b.1):  \emph{For some $i$, both $S_i$ and $T_i$ are in $\Delta$.}   
There are two subsubcases, according to whether $i$ is ``before the gap''  or ``after the gap''. 

\begin{itemize}
\item[ -- ] \textsc{Case} (b.1.1): $i<d+j$.
A weakly-Hamiltonian path arises from a relabeling as follows: We start at the beginning of the second halfpath, with the facets previously called $H_{j+d}, H_{j+d+1}, $ etc., until we reach $H_{n-d}$. Then we use $T_i$ to get back to the vertex previously labeled by $i-1$. Next, we use in reverse order the facets previously called $H_{i-d-1}, H_{i-d-2}, \ldots, H_2, H_1$. Finally  use $S_i$ to jump forward to the vertex previously called $i$, and conclude the path with the facets previously called $H_i, H_{i+1}, \ldots, H_{j-1}$.

\item[ -- ] \textsc{Case} (b.1.2): $i > d+j$.
A weakly-Hamiltonian path arises from a relabeling as follows: We start at the beginning of the second halfpath, with the facets previously called $H_{j+d}, H_{j+d+1}, $ etc., until $H_{i-d-1}$. Then we use $T_i$ to jump forward. As next faces, we use in reverse order the facets previously called $H_{n-d}, H_{n-d-1}, \ldots, H_2, H_i$. Finally, we use $S_i$ to jump back to $H_1$, and conclude the path with the facets previously called $H_1, H_{2}, \ldots, H_{j-1}$. So also in this case $\Delta$ is weakly-traceable, a contradiction. 
\end{itemize}

\textsc{Case} (b.2):  \emph{For all $i$, at most one of $S_i$ and $T_i$ is in $\Delta^*$.} Since the two sets 
$\{ i \in J' : \sigma * i \in \Delta\}$ and $\{ i \in J' : (i-1) * \tau \in \Delta\}$ are disjoint, we obtain a numerical contradiction:
\[ 
\begin{array}{ll}
d_{\sigma} + d_{\tau} &\le d + \;  |\{ i \in J' : \sigma * i \in \Delta\}| \: + d +  \:  | \{ i \in J' : (i-1) * \tau \in \Delta\}| =\\ 
&=  2d + \;  |\{ i \in J' : \sigma * i \in \Delta\} \cup \{ i \in J' : (i-1) * \tau \in \Delta\}| \le \\ 
 &\le 2d + | J' | \; = \; 2d  + n - 2d - 2 \; = \; n-2. \end{array} 
\]

\textbf{Case (c)} is the easiest. If $j \in \{n-2d+1, \ldots, n-d-1\}$, then  $H_{j-1}$ intersects $H_{n-d}$. Since $\Delta$ contains $H_1, \ldots, H_{j-1}$ and also $H_{n-d}$, it is weakly-traceable, a contradiction.
\smallskip

\textbf{Case (d)} is the last one.  So, assume $j=1$ and set
$J'' \eqdef \{d+2, d+3, \ldots, n-d\}$. For every $i$ in $J''$, which has cardinality $n-2d-1$, consider the two $d$-faces of $\Sigma^d_n$
\[ S_i \eqdef \sigma \ast i \quad \textrm{ and } \quad T_i \eqdef (i-1) \ast \tau.\] 
Now there are two subcases:  Either there exists an $i$ such that $S_i, T_i$ are both in $\Delta$, or not.

\bigskip 
\textsc{Case} (d.1):  \emph{For some $i$, both $S_i$ and $T_i$ are in $\Delta$.} Then we obtain a weakly-Hamiltonian path as follows: Starting with $\ell_1=1$, first we use the face $\sigma \ast i$, then $H_2, \ldots, H_{i-d-1}$ in their order, then we use $(i-1)\ast n$ to jump forward, and then we come back with $H_{n-d}, \ldots, H_{i}$. 

\smallskip
\textsc{Case} (d.2):  \emph{For all $i$, at most one of $S_i$ and $T_i$ is in $\Delta$.} We know by that $\sigma \ast d=H_1$ is not in $\Delta$ because we are treating the case $j=1$, and we know that $\sigma \ast n$ is not in $\Delta$ otherwise we would have a weakly-Hamiltonian path. Thus any $d$-face containing $\sigma$ is of the form $\sigma \ast v$, where $v$ is either in $J''$ or in the disjoint set $\{n-d+1, \ldots, n-1\}$, which has cardinality $d-1$. In contrast, any $d$-face containing $\tau$ is of the form $(i-1) \ast \tau$, where $i$ is either in $J''$ or in the set $\{2, \ldots, d+1\}$, which has cardinality $d$.  Since the two sets 
$\{ i \in J'' : \sigma * i \in \Delta\}$ and $\{ i \in J'' : (i-1) * \tau \in \Delta\}$ are disjoint, the sum of their cardinality is equal to the cardinality of their union, which is a subset of $J''$. So also in this case we obtain a contradiction
\[ 
\begin{array}{ll}
d_{\sigma} + d_{\tau} &\le d - 1 + \;  |\{ i \in J'' : \sigma * i \in \Delta\}| \: + d +  \:  | \{ i \in J'' : (i-1) * \tau \in \Delta\}| =\\ 
&=  2d - 1  + \;  |\{ i \in J'' : \sigma * i \in \Delta\} \cup \{ i \in J'' : (i-1) * \tau \in \Delta\}| \le \\ 
 &\le 2d - 1 + | J'' | \; = \; 2d -1  + n - 2d - 1 \; = \; n-2.  \end{array} \qedhere
\]
\end{proof}

\begin{example}
Let $\Delta_7$ be the simplicial complex on $5$ vertices obtained from $\Sigma^2_5$ by removing the interior of the two  triangles $123$ and $124$. Clearly $\Delta_7$ is quasi-traceable with $j=1$, because $\Delta_7 \cup H_1$ is traceable. 
Since $d_{12} + d_{45} = 4 = n-1$,
by Proposition \ref{prop:Ore2}  $\Delta_7$ is weakly-traceable.  In fact, the reader may verify that $\Delta_7$ is even Hamiltonian with the relabeling $\ell_1=1$, $\ell_2=2$, $\ell_3=5$, $\ell_4=3$, $\ell_5=4$.   
\end{example}

For completeness, we conclude this section by showing how Proposition \ref{prop:Ore2} implies part (B) of Ore's theorem \ref{thm:OreGraphs}:

\begin{proof}[\bf Proof of Ore's theorem \ref{thm:OreGraphs}, part (B)] By contradiction, let $G$ be a non-traceable graph satisfying $\deg u + \deg v \ge n-1$ for all non-adjacent vertices $u,v$. Add edges to it until we reach a maximal non-traceable graph $G^*$. This $G^*$ is quasi-traceable and  still  satisfies $\deg u + \deg v \ge n-1$. By  Proposition \ref{prop:Ore2} $G^*$ is (weakly-)traceable, a contradiction.
\end{proof}

\newpage

\section{Interval graphs and semiclosed complexes}

In the present section, 
\begin{compactenum}[ (1) ]
\item we introduce ``weakly-closed $d$-complexes'', generalizing co-comparability graphs; 
\item we create a hierarchy of properties between closed and weakly-closed complexes, among which a $d$-dimensional generalization of interval graphs; and
\item we connect such hierarchy to traceability and chordality.  
\end{compactenum}

\subsection{A foreword on interval graphs and related graph classes}
\textsc{Interval graphs} are the intersection graphs of intervals of $\mathbb{R}$. They have long been studied in combinatorics, since the pioneering papers by Lekkerkerker--Boland \cite{LekBol} and Gilmore--Hoffman \cite{GilmoreHoffman}, and have a tremendous amount of applications; see e.g.~\cite[Ch.~8, Sec.~4]{Golumbic} for a survey. \textsc{Unit-interval graphs}, also known as  ``indifference graphs'' \cite{Roberts} or ``proper interval graphs'', %\footnote{Throughout this paper, all graphs are finite. For infinite graphs, the definitions of ``indifference graph'', ``unit interval graphs'', and ``proper interval graphs'', may slightly differ.},  
are the intersection graphs of \emph{unit} intervals, or equivalently, the intersection graphs of sets of intervals no two of which are nested.
The claw $K_{1,3}$ is the classical example of a graph that can be realized as intersection of four intervals, three of which contained in the forth; but it  cannot be realized as intersection of unit intervals. 

Bertossi noticed in 1983 that connected unit-interval graphs are traceable \cite{Bertossi}, whereas connected interval graphs in general are not: The claw strikes. All $2$-connected unit-interval graphs are Hamiltonian \cite{ChChCh}\cite{PandaDas}; again, this does not extend to $2$-connected interval graphs.  
That said, for interval graphs (and even co-comparability graphs, see below for the definition) the Hamiltonian Path Problem and the Longest Path Problem can be solved in polynomial time \cite{DeogunSteiner} \cite{MerCor}, whereas for arbitrary graphs both problems are well known to be NP-complete, cf.~\cite{Karp}.

Given a  finite set of intervals in the horizontal real line, we can swipe them ``left-to-right'', and thus order them by increasing left endpoint. This so-called ``canonical labeling'' of the vertices of an interval graph obviously satisfies %a stronger property than (\ref{eq:Cocomparability}): 
the following property: for all $a<b<c$, 
\begin{equation} \label{eq:Interval} ac \in G  \Longrightarrow ab \in G.\end{equation}
This ``under-closure'' is a characterization: It is easy to prove by induction that any graph with $n$ vertices labeled so that (\ref{eq:Interval}) holds can be realized as the intersection graph of $n$ intervals. This result was first discovered by Olario, cf.~\cite[Proposition 4]{LO}.

There is a ``geometrically dual argument'' to the one above: Given a finite set of intervals in $\RR$, we could also swipe them right-to-left, thereby ordering the intervals by decreasing right endpoint. This yields a vertex labeling that again satisfies (\ref{eq:Interval}), for the same geometric reasons. In general, since some of the intervals may be nested,  this ``dual labeling'' bears no relation with the canonical one. But if we start with a finite set of \emph{unit} intervals, then the dual labeling is simply the reverse of the canonical labeling. Thus in \emph{unit-interval} graphs, not only the canonical labeling is under-closed, but also its reverse is. Or equivalently, in unit-interval graphs, the canonical labeling is closed `both below and above': in mathematical terms, for all $a<b<c$, 
\begin{equation} \label{eq:UnitInterval} ac \in G  \Longrightarrow ab, bc  \in G.\end{equation}
Again, it is not difficult to prove by induction that any graph with $n$ vertices, labeled so that (\ref{eq:UnitInterval}) holds, can be realized as the intersection graph of $n$ unit intervals  \cite[Theorem 1]{LO}; see Gardi \cite{Gardi} for a computationally-efficient construction.

Recently Herzog et al.~\cite{HerzogEtAl, EneEtAl} rediscovered unit-interval graphs from an algebraic perspective, which will be discussed in the next chapter. They called them \emph{closed graphs} and expanded the notion to higher dimensions as well (\emph{``closed $d$-complexes''}). 
Later Matsuda \cite{Matsuda} extended this algebraic approach to the broader class of ``co-comparability graphs'' (or ``weakly-closed graphs''), that we shall now describe in terms of their complement.

Any graph can be given an acyclic orientation by choosing a vertex labeling and then by directing all edges from the smaller to the larger endpoint. Every acyclic orientation can be induced this way. (This is not a bijection: different labelings may induce the same orientation). The \textsc{drawings of posets}, also called \emph{comparability graphs}, admit also \emph{transitive} orientations, namely, orientations such that  if $\vec{ab}$ and $\vec{bc}$ are present, so is $\vec{ac}$. Let us rephrase this in terms of a vertex labeling, which happens to be the same as a choice of a linear extension of the poset: Comparability graphs are those graphs~$\overline{G}$ that admit a labeling such that, for all $a<b<c$, 
\[  ab \in \overline{G} \textrm{ and } bc \in \overline{G} \Longrightarrow ac \in \overline{G}.\] 
Not all graphs admit transitive orientations: The pentagon, for example, does not. 

\textsc{Co-comparability graphs}, also called \textsc{weakly-closed graphs} in \cite{Matsuda}, are by definition the complements of comparability graphs. So they have a labeling that satisfies the contrapositive of the property above: Namely, for all $a<b<c$, 
\begin{equation}\label{eq:Cocomparability}  ac \in G  \Longrightarrow ab \in G \textrm{ or } bc \in G.\end{equation}

By comparing (\ref{eq:Interval}) and (\ref{eq:Cocomparability}), it is clear that all interval graphs are co-comparability. 

We should mention other two famous properties that all interval graphs enjoy. A graph is \textsc{perfect} if its chromatic number equals the size of the maximum clique. For example, even cycles are perfect, but odd cycles are not, because they have chromatic number $3$ and maximal cliques of size $2$. Note that in poset drawings,  a clique (resp. an independent set) is just a chain (resp. an antichain) in the poset, whereas a coloring represents a partition of the poset into antichains. %When we pass to the complement graphs, chains and antichains exchange roles. 
Thus Dilworth's theorem (``for every partially ordered set, the maximum size of an antichain equals the minimum number of chains into which the poset can be partitioned'' \cite{Dilworth} -- see Fulkerson \cite{Fulkerson} for an easy proof) can be equivalently stated as ``every co-comparability graph is perfect''.   
Not all perfect graphs are co-comparability, as shown by large even cycles.
\enlargethispage{3mm}

Last property: A graph is \textsc{chordal} if it has no induced subcycles of length $\ge 4$. 
One can characterize chordality in the same spirit of (\ref{eq:Interval}), (\ref{eq:UnitInterval}) and (\ref{eq:Cocomparability}): Namely, a graph is chordal if and only if it admits a labeling such that,  for all $a<b<c$,
\begin{equation} \label{eq:Chordal}
ac, bc \in G \Longrightarrow ab \in G.
\end{equation}
In fact, if a graph $G$ has a labeling that satisfies (\ref{eq:Chordal}), then $G$ is obviously chordal, because if $c$ is the highest-labeled vertex in any induced cycle, then its neighbors $a$ and $b$ in the cycle %(with $a < b$) 
must be connected by a chord by  (\ref{eq:Chordal}). The converse, first noticed by Fulkerson--Gross \cite{FulkersonGross}, follows recursively from Dirac's Lemma that every chordal graph has a ``simplicial vertex'', i.e. a vertex whose neighbors form a clique (cf.~\cite[p.~83]{Golumbic} for a proof). In fact, let us pick any simplicial vertex and label it by $n$. Then, in the (chordal!) subgraph induced on the unlabeled vertices, let us pick another simplicial vertex and label it by $n-1$; and so on.  The result is a labeling that satisfies (\ref{eq:Chordal}). See \cite[pp.~84--87]{Golumbic} for two algorithmic implementations.

Now, if the \emph{same} labeling satisfies   (\ref{eq:Cocomparability}) \& (\ref{eq:Chordal}), then it trivially satisfies (\ref{eq:Interval});  and conversely, if (\ref{eq:Interval}) holds, then also  (\ref{eq:Cocomparability}) \& (\ref{eq:Chordal}) trivially hold. Thus it is natural to conjecture that interval graphs are the same as the \emph{co-comparability chordal} graphs. The conjecture is true, although the `obvious' proof does not work: Some labelings on chordal graphs satisfy   (\ref{eq:Cocomparability}) but not (\ref{eq:Interval}), like $13, 23, 24$ on the three-edge path. 
However, Gilmore--Hoffman proved that  any labeling that satisfies   (\ref{eq:Cocomparability}) on a chordal graph (or more generally, on a graph that lacks induced $4$-cycles) can be modified  in a way that `linearly orders' all maximal cliques \cite[Theorem 8.1]{Golumbic} and thus satisfies (\ref{eq:Interval}).
For more characterizations, %including the one of chordal graphs as intersection graph of subtrees, 
and a proof that all chordal graphs are perfect, see Golumbic \cite[Chapter 4]{Golumbic}.

\subsection{Higher-dimensional analogs and a hierarchy}
A $d$-dimensional extension\footnote{Several different  $d$-dimensional generalizations of chordality exist in the literature, e.g. toric chordality \cite{ANS} or ridge-chordality, cf.~e.~g.~\cite{Bolo}. Emtander chose the name ``$d$-chordal'' for what here we call ``chordal''.} of Characterization (\ref{eq:Chordal}) of chordality was provided in 2010 by Emtander \cite{Emtander}, and is equivalent to the following:

\begin{definition}[chordal] Let $\Delta$ be a pure $d$-dimensional  simplicial complex with $n$ vertices. $\Delta$ is called \emph{chordal} if there exists a labeling $1, \ldots, n$ of its vertices (called a \emph{``PEO''} or ``\emph{Perfect Elimination Ordering}'') such that for any two facets $F=a_0 a_1\cdots a_d$ and $G= b_0\cdots b_d$ of $\Delta$ with $a_d=b_d$, the complex $\Delta$ contains the full $d$-skeleton of the simplex on the vertex set $F \cup G$.
\end{definition}

In 2013, Characterization (\ref{eq:UnitInterval})  of unit-interval graphs was generalized as well:

\begin{definition}[closed \cite{EneEtAl}] 
Let $\Delta$ be a pure $d$-dimensional  simplicial complex with $n$ vertices.  $\Delta$ is called \emph{closed} if there exists a labeling $1, \ldots, n$ of its vertices such that for any two facets $F=a_0 a_1\cdots a_d$ and $G= b_0\cdots b_d$ of $\Delta$ with $a_i=b_i$ for some $i$, the complex $\Delta$ contains the full $d$-skeleton of the simplex on the vertex set $F \cup G$.
\end{definition}

Obviously, closed implies chordal. We now present four notions that in the strongly connected case are progressive weakenings of the closed property (see Theorem \ref{thm:Hierarchy} and Proposition \ref{prop:closed2} for the proofs); the first property still implies chordality, whereas the last three do not. In Section \ref{sec:ClosedTraceable}, we connect all these notions to traceability (Theorem \ref{thm:Hi}). One of these properties is ``new'' even for $d=1$: We will see its importance in Chapter 3. % and to algebraic shifting (Remark  \ref{rem:kalai}).

\begin{definition}[unit-interval] 
Let $\Delta$ be a pure $d$-dimensional simplicial complex with $n$ vertices. The complex $\Delta$ is called \emph{unit-interval} if there exists a labeling $1, \ldots, n$ of its vertices such that  for any $d$-face $F=a_0 a_1 \cdots a_d$ of $\Delta$, the complex $\Delta$ contains the whole $d$-skeleton of the simplex with vertex set  
$\{a_0, a_0 +1, a_0 + 2, \ldots, a_d\}$.
\end{definition}

\begin{definition}[under-closed] 
Let $\Delta$ be a pure $d$-dimensional simplicial complex with $n$ vertices. The complex $\Delta$ is called \emph{under-closed} if there exists a labeling $1, \ldots, n$ of its vertices such that  for any $d$-face $F=a_0 a_1 \cdots a_d$ of $\Delta$ the following condition holds:
\begin{compactitem}
\item all faces $a_0i_1i_2\ldots i_d$ of $\Sigma^d_n$ with $i_1\leq a_1, i_2\leq a_2,\ldots , i_d\leq a_d$, are in $\Delta$.
\end{compactitem}
\end{definition}

\begin{definition}[semi-closed]
Let $\Delta$ be a pure $d$-dimensional simplicial complex with $n$ vertices. The complex $\Delta$ is called {\it semi-closed} if there exists a labeling  of its vertices such that for any $d$-face $F=a_0a_1\ldots a_d$ of $\Delta$, at least one of the two following conditions holds: 
\begin{compactenum} [  (i) ]
\item either all faces $a_0i_1i_2\ldots i_d$ of $\Sigma^d_n$ with $i_1\leq a_1, i_2\leq a_2,\ldots , i_d\leq a_d$, are in $\Delta$,
\item or all faces $i_0i_1\ldots i_{d-1}a_d$ of $\Sigma^d_n$ with $i_0\geq a_0, i_1\geq a_1,\ldots ,i_{d-1}\geq a_{d-1}$ are in $\Delta$.
\end{compactenum}
\end{definition}

\begin{definition}[weakly-closed] 
Let $\Delta$ be a pure $d$-dimensional  simplicial complex with $n$ vertices.  $\Delta$ is called \emph{weakly-closed} if there exists a labeling $1, \ldots, n$ of its vertices such that for each $d$-face $F=a_0 a_1 \cdots a_d \in \Delta$, for every integer $g \notin F$ with $a_0 < g < a_d$, there exists a $d$-face $G=b_0 b_1 \cdots b_d$ in $\Delta$ such that $G$ contains $g$, $G$ is adjacent to $F$, and at least one of the following two conditions hold:
\item \begin{compactenum}[ (i)]
\item either $b_d \ne a_d$,
\item or $b_0 \ne a_0$.
\end{compactenum}
\end{definition}

\begin{remark}\label{rem:matsuda} For $d=1$, and assuming connectedness:
\begin{compactitem} 
\item ``closed $1$-complexes'' and ``unit-interval $1$-complexes'' are the same as the  \emph{unit interval graphs}; compare Looges--Olario \cite[Theorem 1]{LO} and Matsuda \cite[Prop.~1.3]{Matsuda}.
% introduced by Herzog et al \cite{HerzogEtAl}; 
\item ``under-closed $1$-complexes'' are the same as the \emph{interval graphs}, cf.~\cite[Proposition 4]{LO}.
\item ``weakly-closed $1$-complexes'' are the same as the  \emph{co-comparability graphs}; this is clear from  the definition we gave, but a proof is also in Matsuda \cite[Theorem 1.9]{Matsuda}. %, it contains the claw graph, but not the $4$-cycle. 
\end{compactitem}
We will see that ``semi-closed $1$-complexes''  are an intermediate class  between the previous two. For example, such class contains the $4$-cycle but not the complement of long even cycles, as we will prove in  Theorem \ref{thm:Hierarchy}.  
\end{remark}

\begin{remark}[``unit-interval'' vs. ``chordal''] \label{rem:almost}
Suppose $F$ and $G$ are two faces of a complex $\Delta$ with $\min F = \min G$. Then any of the two conditions ``$\Delta$ is closed'', ``$\Delta$ is unit-interval'' forces $\Delta$ to contain the full $d$-skeleton of the simplex on the vertex set $F \cup G$. (Instead, the condition ``$\Delta$ is under-closed'' does not suffice: See Remark  \ref{rem:UCdC} below). 
Symmetrically, if $F$ and $G$ are $d$-faces of $\Delta$ with $\max F = \max G$, and $\Delta$ is either closed or unit-interval, then  $\Delta$ must contains the full $d$-skeleton of the simplex on the vertex set $F \cup G$. For this reason, all unit-interval $d$-dimensional complexes are chordal. 
\end{remark}

\begin{remark}[``Under-closed'' vs. ``chordal''] \label{rem:UCdC}
Not all chordal complexes are under-closed: Alread for $d=1$, the chordal graph $G=12, 13, 14, 23, 25, 36$, known as ``$3$-sun'' or ``net graph'', is neither interval nor co-comparability. However, while all interval graphs are chordal (and co-comparability), the statement ``all under-closed $d$-complexes are chordal'' is false for $d > 1$. In fact, we leave it to the reader to verify that the smallest counterexample is the $2$-complex
\[ \Delta \eqdef 123, 124, 234, 235.\]
The other direction in Gillmore--Hoffman's theorem (namely, ``all chordal co-comparability graphs are interval graphs'') does not extend to $d>1$ either, as the next Proposition shows.
\end{remark}

\begin{proposition} \label{prop:GilmoreHoffman}
\begin{compactenum}[\rm (i) ]
\item Some chordal simplicial complexes are semi-closed, but not under-closed. 
\item If a simplicial complex is chordal and semi-closed with respect to the \emph{same} labeling, then with respect to that labeling the complex is also under-closed. 
\end{compactenum}
\end{proposition}

\begin{proof}
\begin{compactenum}[\rm (i) ] 
\item The example we found is the complex
\[ \Sigma = 123, 124, 134, 135, 167, 234, 246.\]
The labeling above is a PEO, so $\Sigma$ is chordal. A convenient relabeling (we leave it to the reader to figure out the bijection from the vertex degrees) allows us to rewrite it as 
\[ \Sigma= 123, 256, 345, 346, 347, 356, 456.\]
With this new labeling we see that $\Sigma$ is weakly- and semi-closed.  However, with the help of a software designed by Pavelka \cite{Pavelka}, we verified that $\Sigma$ is not under-closed.  
\item Let $\Delta$ be a simplicial complex with a labeling that is a PEO and makes $\Delta$ semi-closed. Let $F=a_0 \cdots a_d$ be a face of $\Delta$ with $\operatorname{gap} F >0$. 
Let $G=a_0 b_1 \cdots b_d$ be a different $d$-face of $\Sigma^d_n$ such that $G \le F$ (componentwise) and $\min G = \min F$. 
We claim that for any $b_i$ not in $F$, there exists a $d$-face $A_i$ of $\Sigma^d_n$ that contains $b_i$, such that $A_i \ge F$ (componentwise) and $\max A_i = \max F$. In fact, by construction $a_0 < b_i \le  b_d \le a_d$.  Since $b_i$ is not in $F$, 
%$a_d$ is in $F$ and $b_i$ is not, the case $b_i = a_d$ is excluded; so we have $a_0 < b_i < a_d$. Hence 
there exists a unique $j \in \{0, \ldots, d-1\}$ such that $a_{j} < b_i < a_{j+1}$. Thus if we set
\[ A_i \eqdef a_0 \cdots a_{j-1} b_i a_{j+1} \cdots a_{d}\]
the claim is proven. Now, either  $F$ satisfies condition (i) of the semi-closed definition, and then $G \in \Delta$; or $F$ satisfies condition (ii),  in which case all $A_i$'s are in $\Delta$.  But by construction, the maximum of all these $A_i$'s is $a_d$, the same maximum of $F$. So by chordality, $\Delta$ must contain all the $d$-faces of $\Sigma^d_n$ with vertex set contained in 
\[F \:  \cup \:   \bigcup_{i \textrm{ s.t. } b_i \notin F} A_i \: \: = \: \{a_0, a_1, \ldots, a_d\} \cup \{b_1, \ldots, b_d\} \: = \: F \cup G.\] So also in this case $G \in \Delta$. \qedhere
\end{compactenum}
\end{proof}

\begin{remark}
Part (ii) of Proposition \ref{prop:GilmoreHoffman} is false if one replaces the assumption ``semi-closed'' with ``weakly-closed'': The subcomplex $\Sigma' = 123, 124, 134, 135, 234$  of $\Sigma$ is weakly-closed and chordal with respect to this labeling, but to prove it under-closed, we need to change labeling.
\end{remark}

\begin{remark}[``Under-closed'' vs.~``Shifted'']\label{rem:kalai} 
Recall that a simplicial complex $\Delta$ on $n$ vertices is called \emph{shifted} if for every face $F$ of $\Delta$, and for every face $G$ of the simplex on $n$ vertices, if $\dim F = \dim G$ and $F \le G$ componentwise, then also $G \in \Delta$. Shifted complexes are obviously under-closed. The converse is false, as shown by the graph $12, 23, 34$. 
\end{remark}

\begin{remark}\label{rem:cones} 
Being shifted is maintained under taking cones, by assigning label $1$ to the new vertex. In contrast, $G=12,13,23$ is closed and chordal, but the cone over it is neither closed nor chordal.  In fact, \emph{none} of the five properties (closed, unit-interval, under-closed, semi-closed, weakly-closed) is maintained under taking cones. A counterexample for all is the unit-interval graph $G= 12, 34, 56, 78$. The cone over $G$ is the $U^2_4$ of Lemma \ref{lem:bouquet} below.
\end{remark}

 Let us start exploring the relations between all the new properties with some Lemmas.

\begin{lemma} \label{lem:skeleton}
Let $d \ge k \ge 1$  be integers. If a pure $d$-dimensional simplicial complex is unit-interval (resp. under-closed, resp. semi-closed, resp. weakly-closed), then its $k$-skeleton is also unit-interval (resp. under-closed, resp. semi-closed, resp. weakly-closed).
\end{lemma}

\begin{proof} It suffices to prove the claim for $k=d-1$; the general claim follows then by iterating. We prove only the weakly-closed case; the others are easier. Let $\Delta$ be a pure weakly-closed $d$-complex. Let $\sigma=a_0 \cdots a_{d-1}$ be a $(d-1)$-face of $\Delta$. Let $g \notin \sigma$ be an integer such that $a_0 < g < a_{d-1}$. Since $\Delta$ is pure, there exists a $d$-face $F$ of $\Delta$ that contains $\sigma$. Let $v$  be the vertex of $F$ not in $\sigma$.  If $v=g$, i.e. if $F = \{g\} \cup \sigma$, then all the $d$ facets of $\Delta$ different than $\sigma$ are adjacent to $\sigma$ and contain $g$; if we choose one of these $d$ facets that has either different minimum or different maximum than $\sigma$, we are done. So let us assume that $ v \ne g$, or equivalently, that $F$ does not contain $g$. By the weakly-closed assumption, there exists a $d$-face $G$ in $\Delta$ such that $G$ contains $g$, $G$ is adjacent to $F$, and $G$ and $F$ do not have same minimum and maximum. If $G$ contains the entire face $\sigma$, i.e. $G = \sigma \cup g$, then again we could conclude as above, choosing some facet of $G$ different than $\sigma$. So we can assume that $G$ does not contain the whole of $\sigma$, or in other words, that the vertex $v$ is present in $G$. Let $\tau$ be the unique face of $G$ that does not contain $v$. By construction, $\sigma$ and $\tau$ are adjacent, and $g \in \tau$. If $\sigma$ and $\tau$ had same minimum and maximum, then also $F$ and $G$ would, because $F$ and $G$  are obtained by  adding to $\sigma$ and $\tau$, respectively, the same element $v$. Hence, the $(d-1)$-skeleton of $\Delta$ is weakly-closed.
\end{proof}

\begin{lemma} \label{lem:Bd} Let $d \ge 2$. Let $C^{d+1}$  be the $(d+1)$-dimensional simplicial complex with facets $H_1$ and $H_2$. The boundary $S^d$ of $C^{d+1}$ is strongly-connected, semi-closed, but not under-closed. The $d$-skeleton $B^d$ of $C^{d+1}$ is traceable, strongly-connected, unit-interval, but not closed. \\In particular, the $k$-skeleton of a closed complex need not be closed. 
\end{lemma}

\begin{proof} 
Note that $S^d$ is $B^d$ minus a $d$-face, so since $d \ge 2$ the $1$-skeleta of $B^d$ and of $S^d$ coincide. The vertices of $B^d$ (respectively, of $S^d$) can be partitioned with respect to the number of edges containing them, as follows:  exactly two  vertices have degree $d+1$, and we shall call them ``apices''; the remaining $d+1$ have degree $d+2$, and we shall call them ``basepoints''.  The crucial remark is that in $B^d$ (resp. $S^d$) the two apices are not connected by any edge.  We claim that any labeling that makes $B^d$ or $S^d$ closed  \emph{must} assign labels $1$ and $d+3$ to the two apices.  In fact: 
\begin{compactitem}
\item If the label $1$ is assigned to a basepoint, let $b_1, \ldots, b_d$ be the other $d$ basepoints and let $v,w$ be the apices, with $v<w$. Then $B^d$ (resp.~$S^d$) contains a $d$-face $F$ of vertices $\{1,b_1, \ldots, b_{d-1}, v\}$ and a $d$-face $G$ of vertices $\{1,b_1, \ldots, b_{d-1}, w\}$.  
Note that $1$ is in the same position in $F$ and $G$, 
yet $B^d$ (or $S^d$) does not contain the whole $d$-skeleton of the simplex on $F \cup G$, because $vw$ is missing. So the closed condition is not satisfied.
\item Symmetrically, if $d+3$ is assigned to a basepoint, call $b_1, \ldots, b_d$ the other basepoints and $v,w$ the apices, with $v<w$.  Then $B^d$ (resp.~$S^d$) contains a $d$-face $F$ of vertices $\{v, b_1, \ldots, b_{d-1}, d+3\}$ and a $d$-face $G$ of vertices $\{v, b_1, \ldots, b_{d-1}, d+3\}$. So $d+3$ is the maximum of both faces, and again $B^d$ (resp.~$S^d$) does not contain the edge $vw$, so the closed condition is not met. 
\end{compactitem}
Next, we claim that any labeling that makes $S^d$ under-closed  \emph{must} assign labels $1$ and $d+3$ to the two apices. (Caveat: This claim is valid only for $S^d$, since already $B^2$ is under-closed with the labeling $123, 124, 134, 234, 125, 135, 235$, where the apices are $4$ and $5$.) 
 In fact: 
\begin{compactitem}
\item If the label $1$ is assigned to a basepoint, then any other vertex is contained in a facet that contains also $1$. The same is true if $d+3$ is assigned to a basepoint. So either way,  there is a face $H$ containing both $1$ and $d+3$. Thus $\operatorname{gap}H =2$. But then if the labeling is under-closed, the complex must contain all three facets
$12 \cdots d \, j $, with $j \in \{d+1, d+2, d+3\}$.  So we found in $S^d$ three different facets containing the $(d-1)$-face $\sigma \eqdef 12 \cdots d$. This is a contradiction because $S^d$ is topologically a sphere: Every $(d-1)$-face in it lies in exactly two $d$-faces.  
\end{compactitem}

Thus the two claims are proven. So up to a rotation that does not affect the list of facets, both for $B^d$ and $S^d$ we may focus on the labeling that we introduced from the start. With respect to that labeling, $S^d$ is clearly semi-closed, but it is not under-closed, because the $d$-face with vertices $2, 3, \ldots, d+1, d+2$ is missing. Similarly, with respect to that labeling, $B^d$ is traceable and unit-interval, but it is not closed for the following reason. Let $F$ (resp. $G$) be the face of vertices $1, 3, 4, \ldots, d+1, d+2$ (resp.  $2, 3, 4, \ldots, d+1, d+3$). Since $F$ (resp. $G$) is contained in the facet $H_1$ (resp. $H_2$) of $C^{d+1}$,  it is in  $B^d$. Yet vertex $3$ appears in second position in both $F$ and $G$. However,  the face $H_3$ of vertices $1, 3, 4, \ldots, d+1, d+3$ contains the edge connecting the two apices, so $H_3$ is not in $B^d$.
\end{proof}

\begin{figure}[htbp] 
\begin{center}
\hskip-4mm \includegraphics[scale=0.37]{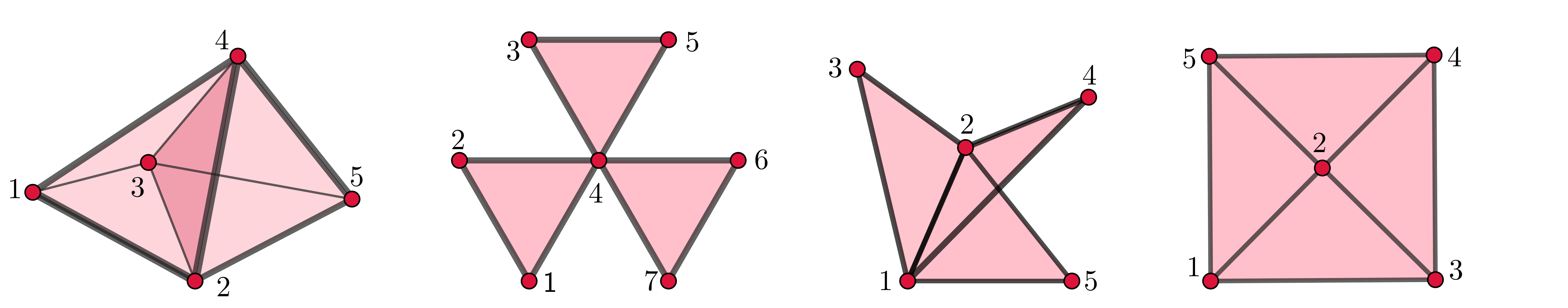} 
\caption{(i) A $2$-complex $B^2=123, 124, 134, 234, 235, 245, 345$ that is unit-interval, but not closed; if we remove the triangle $234$, we get a $2$-complex $S^2$  that is semi-closed, not under-closed, cf.~Lemma~\ref{lem:Bd}.
(ii) A $2$-complex $U^2_3 = 124, 345, 467$ that is closed, but not weakly-closed, cf.~Lemma~\ref{lem:bouquet}. \\
(iii) A~$2$-complex $\Delta^2_3 = 123, 124, 125$ that is under-closed, but not unit-interval, cf.~Lemma~\ref{lem:WheelComplex}.\\
(iv) A $2$-complex $Q^2=123, 125, 234, 245$ that is weakly-closed, but not semi-closed, cf.~Lemma~\ref{lem:StarredSquare}.}
\label{fig:ToyExamples1}
\end{center}
\end{figure}

\begin{lemma} \label{lem:bouquet}
Let $d$ and $k$ be positive integers. Let $U^{d}_k$ be a one-point union of $k$ copies of $\Sigma^d$. Then $U^{d}_{k}$ is closed if and only if $k \le d+1$, and it is weakly-closed if and only if $k \le 2$. In particular, for all $d \ge 2$, the $d$-complex $U^d_{d+1}$ is closed, but not weakly-closed.
\end{lemma}

\begin{proof} Let $v$ be the vertex common to all facets. When $k>d+1$, by the pigeonhole principle there are two facets in which $v$ appears in the same position; were $U^{d}_{k}$ closed, its dual graph would have to contain a clique, which is not the case. When $k \le d+1$, we force the closed property by giving $v$ a label so that $v$ appears in a different position in all facets. We show an algorithm to do this in case $k=d+1$, leaving the case $k < d+1$ to the reader. We label $v$ by $f_d \eqdef \binom{d+1}{2} + 1$.  We label the vertices of the first facet by $123\cdots d\, f_d$: so in the first facet, $v$  comes last. 
Then for all $i=2, 3, \ldots, k=d+1$, we label the $i$-th facet by using the next available $d-i$ integers below $f_d$,  then $f_d$, then the first $i-1$ available integers after $f_d$. This way in the $i$-th facet, $v$ comes ``$i$-th last''.  
For example, the labeling we construct for $U^3_4$, since $f_3 \eqdef \binom{4}{2} + 1 = 7$, $\, $ is $\: U^{3}_{4} \, = \, 1237, \; 4578, \; 679\,10, \; 7 \, 11 \, 12 \, 13$. \\
Finally, suppose that  $U^{d}_{k}$  is weakly-closed. No face of   $U^{d}_{k}$  has an adjacent facet. Hence, the labeling satisfying the weakly-closed condition must consist only of gap-$0$ faces. But labeling all facets with consecutive vertices is possible if and only if $k \le 2$. 
\end{proof}

\begin{lemma} \label{lem:WheelComplex}
Let $k\ge 1$ and $d\ge 2$ be integers. 
Let $\Delta^{d}_k$ be the $d$-dimensional complex on $d+k$ vertices obtained by joining the $(d-1)$-simplex $\Sigma^{d-1}$ to a $0$-complex consisting of $k$ points. 
Then 
\begin{compactenum}[ \rm (a)]
\item $\Delta^d_k$ is under-closed for all $k$. 
\item $\Delta^d_k$ is  closed, if and only if it is unit-interval, if and only if it is (weakly) traceable, if and only if $k \le 2$.
\end{compactenum}
\end{lemma}

\begin{proof} 
Let us label the vertices of $\Sigma^{d-1}$ by $1, 2, \ldots, d$. This labeling immediately shows that $\Delta^d_k$ is under-closed.  
Moreover, the $d$-complex $\Delta^d_k$  is strongly-connected. It has exactly $d+k$ vertices and $k$ facets. When $k \le 2$ its dual graph is a path, so clearly the obvious, consecutive labeling makes  $\Delta^d_k$ a closed, unit-interval, and traceable complex. But when $k \ge 3$, the ``path of $k$  $d$-simplices'' is not a subcomplex of $\Delta_d$. Hence, for $k \ge 3$ the complex $\Delta^d_k$ is not traceable,  not weakly-traceable, and not weakly-Hamiltonian. The fact that  $\Delta^d_k$ is neither unit-interval nor closed can be verified either directly, or using  Proposition \ref{prop:closed2} and Theorem \ref{thm:CTSC} below. 
\end{proof}

\begin{remark} \label{rem:G5} The $1$-skeleton of  $\Delta^2_3 = 123, 124, 125$ (cf.~Figure \ref{fig:ToyExamples1}) is the graph
\[ G_5 = 12, 13, 14, 15, 23, 24, 25 \]
which is under-closed by Lemma \ref{lem:skeleton}. It is not difficult 
to see that $G_5$ is the smallest $2$-connected interval graph that is not Hamiltonian. 
\end{remark}

\begin{lemma} \label{lem:StarredSquare}
Let $d\ge 2$ be an integer. Let $Q^{d}$ be the $d$-dimensional complex on $d+3$ vertices obtained by taking $d-1$ consecutive cones over the square. 
Then  $Q^{d}$ is weakly-closed, but not semi-closed.
\end{lemma}

\begin{proof}
Both $Q^2 = 123, 125, 234, 245$ and  $Q^3=1236, 1256, 2346, 2456$ are  weakly-closed. 
If we label further coning vertices using consecutive labels after $6$, we claim that the weakly-closed property is maintained. (This is not obvious, as the weakly-closed property is not maintained under arbitrary cones, cf.~Remark \ref{rem:cones}.) 
 In fact, since every face $F$ of $Q^3$ contains $6$, the gap of $F$ equals the gap of $F \cup \{7\}$, and the missing integers are the same, so the calculations proving weakly-closedness end up being the same for $Q^3$ and $Q^4$.  
 For the same reasons, one can show that if some $Q^{d}$ is semi-closed with a labeling that assigns consecutive labels to two apices, then $Q^{d-1}$ is semi-closed too. But if $d \ge 7$, $Q^d$ has $\ge 10$ vertices, and only $4$ of them are not apices; so necessarily two apices are assigned consecutive labels. So to complete the proof we only need to show that  $Q^2, Q^3, Q^4, Q^5$ and $Q^6$ are not semi-closed, which can be verified with \cite{Pavelka}.  
\end{proof}
 
\vskip3mm

\begin{lemma}\label{lem:SemiClosedGraph}
Let $\Delta$ be a pure $d$-complex where every vertex is in at most $k$ facets. 
\begin{compactenum}[\rm (1)]
\item In any labeling that makes $\Delta$ weakly-closed, every facet has gap $\le 2k-2$. 
\item In any labeling that makes $\Delta$ semi-closed, every facet has gap $\le k-1$. 
\\ If in addition $d=1$ and $\Delta$ is a $k$-regular graph, then in any labeling that makes $\Delta$ semi-closed,  the $k$ edges of the type 1j , with $2 \le j \le k+1$,  are all in $\Delta$; and so are all the $k$ edges of the type $\,in$, with $ n-k \le i \le n-1$.
\item In any labeling that makes $\Delta$ unit-interval, every facet has gap $\le g$, where $g$ is the largest integer such that $\binom{g+d}{d} \le k$; in particular, every facet has gap $ \le  \sqrt[d]{k d!} - 1$.
\end{compactenum}
\end{lemma}

\begin{proof} For any vertex $v$ of $\Delta$, let $\deg v$ be the number of facets of $\Delta$ containing it. For any facet $F$ of $\Delta$, let $S_F$ be the set of integers $i \notin F$ such that $\min F < i < \max F$. By definition, $S_F$ has cardinality  equal to $\operatorname{gap} F$. For brevity, set $a \eqdef \min F$ and $b \eqdef \max F$. 
\begin{compactenum}[(1)]
\item For every $i$ in $S_F$, there is  a face $G_i$ adjacent to $F$ that contains the vertex $i$ and exactly $d$ vertices of $F$, among which exactly one of $a, b$. Clearly as $i$ ranges over $S_F$, the $G_i$'s are all different. So $\deg a + \deg b \ge \operatorname{gap} F + 2$. (The summand $2$ is due to the fact that we should count also $F$ itself, once contributing to $\deg a$ and once to $\deg b$). Since $k \ge \deg a$ and $k \ge \deg b$, we conclude that $\operatorname{gap} F \le 2k - 2$.
\item For every $i$ in $S_F$, either $\Delta$ contains the $n_a \ge  \operatorname{gap} F+1$  facets (including $F$ itself) with minimum $a$ that are componentwise $\le F$, or  $\Delta$ contains the $n_b\ge  \operatorname{gap} F + 1$  facets (including $F$ itself) with maximum $b$ that are componentwise $\ge F$. Either way, there is a vertex $v$ (either $a$ or $b$) with $\deg v \ge  \operatorname{gap} F+1$. Since $\deg v \le k$ by assumption, we conclude that $\operatorname{gap} F \le k -1$.
So the first claim is settled. From this applied to $d=1$, it follows that 
\[ \{ \textrm{ edges of $\Delta$ containing $1$ } \} \ \subseteq \   \{ \; 1j \textrm{ such that }  2 \le j \le k+1 \;\}.\]
The two sets above have size $\deg 1$ and $k$, respectively. If $\Delta$ is $k$-regular, the two quantities are equal, hence the sets coincide. 
The same argument applies to the edges containing $n$. 
\item For every  $i$ in $S_F$, by definition of unit-interval, $\Delta$ contains the $\binom{\operatorname{gap} F + d}{d}$ $d$-faces that contain vertex $i$ and have vertices in $\{a, a+1, \ldots, b\}$.  So we must have $\binom{\operatorname{gap} F + d}{d} \le k$. In particular, since  $\binom{g+d}{d} \ge  \frac{(g+1)^d}{d!}$ for all positive integers $g,d$, we cannot have $\frac{(\operatorname{gap} F+1)^d}{d!} > k$. \qedhere
\end{compactenum}
\end{proof}

Our next Lemma is a $d$-dimensional version of the well-known fact that cycles of length $5$ or more are not co-comparability, cf.~Matsuda \cite{Matsuda}. 

\begin{lemma} \label{lem:dMatsuda}
For $n \ge 2d + 3$, the $d$-dimensional annulus $A^d_n \eqdef H_1, H_2, \ldots, H_n$ and
any $k$-skeleton of it are not weakly-closed.
\end{lemma}

\begin{proof}
 By Lemma \ref{lem:skeleton}, it suffices to prove that the $1$-skeleton $G$ of $ A^d_n$ is not weakly-closed. By contradiction, let $a_1, \ldots, a_n$
be a re-labeling of the vertices $1, \ldots, n$ (respectively) that proves $G$ weakly-closed. Up to rotating the labeling cyclically, we can assume that $a_1$ is the smallest of the $a_i$'s. Since $n \ge 2d + 3$, in particular $n-d> d+2$, so the labels $a_{n-d}, a_{n-d+1}, \ldots, a_{n}, a_1, a_2, \ldots, a_{d+2}$ are all distinct. Were $a_{d+2} < a_n$, we would have a contradiction  with the weakly-closed assumption: $a_1a_n$ is in $G$, but neither $a_1a_{d+2}$ nor $a_{d+2}a_n$ is. So $a_n < a_{d+2}$. Symmetrically, were 
$a_{d+1}>a_{n-d}$, we would have a contradiction: $a_1a_{d+1}$ is in $G$, but neither $a_1a_{n-d}$ nor $a_{n-d}a_{d+1}$ is. So $a_{d+1} < a_{n-d}$. Now let us compare $a_{d+1}$ and $a_n$:
\begin{compactitem}[--]
\item If $a_{d+1} > a_n$, then   $a_n < a_{d+1} < a_{n-d}$ by what we said above; so we get a contradiction, because the edge $a_n a_{n-d}$ is in $G$, but neither $a_na_{d+1}$ nor $a_{d+1} a_{n-d}$ is. 
\item  If $a_{d+1} < a_n$, then   $a_{d+1} < a_n < a_{d+2}$ by what we said above; so symmetrically we get another contradiction, because  $a_{d+1}a_{d+2}$ is in $G$, but neither $a_{d+1}a_n$ nor $a_n a_{d+2}$ is. \endproof
\end{compactitem}
\end{proof}

\begin{remark}
$A^2_6$ is not weakly-closed, even if its $1$-skeleton is semi-closed \cite{Pavelka}. ($A^2_5$ instead is weakly-closed.) So the bound $n \ge 2d+3$ of Lemma \ref{lem:dMatsuda} is best possible in general, but  if one only cares about $A^d_n$ and not about its skeleta, then it can be improved.
\end{remark}

\vskip3mm

\begin{theorem} \label{thm:Hierarchy}
For each $d \ge 1$, for (pure) simplicial $d$-complexes, one has the hierarchy
\[  \{ \textrm{ \rm unit-interval } \} \subsetneq \{ \textrm{ \rm under-closed } \} \subsetneq \{ \textrm{ \rm semi-closed } \} 
\subsetneq \{ \textrm{ \rm weakly-closed } \} \subsetneq \{ \textrm{ \rm all } \} .\] 
\end{theorem}

\begin{proof} All inclusions are obvious except perhaps the third one. Let $F=a_0a_1\ldots a_d$ be a face of~$\Delta$. 
If $F$ satisfies condition (i) in the definition of semi-closed, and there is a $g$ such that  $a_i < g < a_{i+1}$, then $G' \eqdef a_0 a_1 \cdots a_i \: g \: a_{i+1} \cdots a_{d-1}$
is componentwise $\le F$ and thus belongs to $\Delta$; moreover, since  $\max G' < \max F$, the face $G'$ satisfies condition (i) in the definition of weakly-closed. If instead $F$ satisfies condition (ii) in the definition of semi-closed, and  $a_i < g < a_{i+1}$ for some $g$, then 
$G'' \eqdef a_1 \cdots a_i \: g \: a_{i+1} \cdots a_d$
is componentwise $\ge F$, so $G''$ is in $\Delta$; and since $\min G'' > \min F$, this $G''$ 
satisfies condition (ii) in the definition of weakly-closed.

\vskip2mm
\begin{figure}[htbp] 
\begin{center}
\includegraphics[scale=0.43]{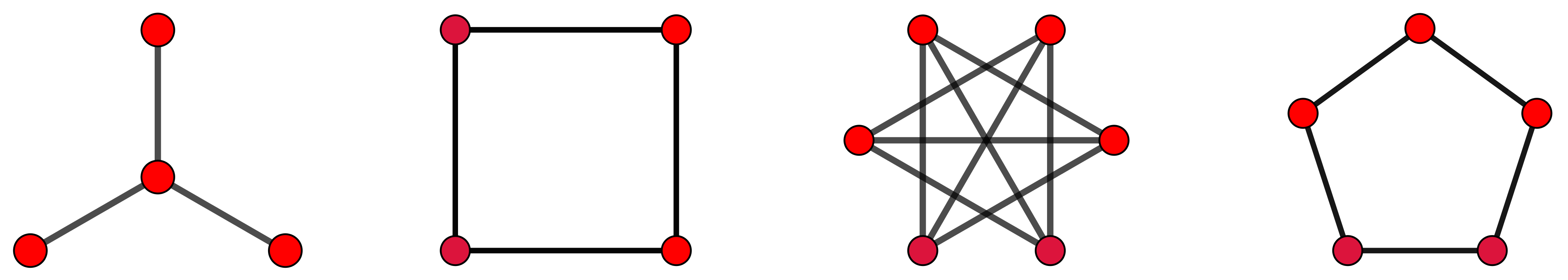} 
\caption{One-dimensional simplicial complexes that are: (i) Not unit-interval, but under-closed. (ii) Not under-closed, but semi-closed. (iii) Not semi-closed, but weakly-closed. (iv) Not even weakly-closed.}
\label{fig:Hierarchy}
\end{center}
\end{figure}

Next, we discuss the strictness of the inclusions, which is the interesting part of the theorem.  
\begin{compactenum}[ (i) ] 
\item For $d=1$, the claw graph $12, 13, 14$ is under-closed only with this labeling, which is not unit-interval because for example $23$ is missing. \\For $d \ge 2$,  strictness follows by Lemma~\ref{lem:WheelComplex}.  
\item For $d=1$, the $4$-cycle is semi-closed with the labeling $12, 13, 24, 34$. By Lemma \ref{lem:SemiClosedGraph}, part (2), \emph{only} this labeling makes the $4$-cycle semi-closed. This labeling is not under-closed, because $24$ is an edge, but $23$ is not. More generally, for any $n \ge 4$, one can show that the graph $\operatorname{susp}(A_{n-2})$ of Remark \ref{rem:An} is semi-closed (with the suspension apices labeled by $1$ and $n$), but not under-closed. \\For $d \ge 2$, the strictness of the inclusion follows by  Lemma \ref{lem:Bd}. 
\item %Set $a \eqdef 2k-2$, $b\eqdef 2k-1$, $c \eqdef 2k$. %Since the $2k$-cycle is $2$-regular, its complement $\overline{C_{2k}}$ is $(a-1)$-regular. 
For $d=1$: Since $C_{2k}$ is a comparability graph (it is the nonempty-face poset of the $k$-gon),  $\overline{C_{2k}}$ is co-comparability. %For example, $\overline{C_{6}}$, depicted in Figure \ref{fig:Hierarchy}(iii),  is weakly-closed with the labeling \[\overline{C_6}=12, 13, 16, 23, 24, 35, 45, 46, 56.\] 
We claim that $\overline{C_{2k}}$ is not semi-closed for any $k \ge 3$. For notational simplicity, we give the proof for $k=3$; the case of arbitrary $k$ has a completely analogous proof. Suppose by contradiction that $\overline{C_{6}}$ has a semi-closed labeling. Since $C_6$ is $2$-regular,  its complement is $(6-1-2)$-regular, i.e. $3$-regular. By Lemma \ref{lem:SemiClosedGraph}, part (2), all of $12, 13, 14$ and $36, 46, 56$ are edges. In contrast, $15$, $16$ and $26$ are not edges, again by  Lemma~\ref{lem:SemiClosedGraph}. But then $25$ must be an edge of $\overline{C_{6}}$, for otherwise $15$, $16$, $26$ and $25$ would form a $4$-cycle inside the complement, which is $C_{6}$. We claim that this edge $25$ cannot satisfy the semi-closed condition. In fact, if all of $23, 24, 25$ were edges, together with $12$ we would have $4$ edges containing vertex $2$, contradicting $3$-regularity; and similarly, if all of $25, 35, 45$ were edges, counting also $56$ we would have $4$ edges containing vertex $5$.\\
This shows strictness of the inclusion for $d=1$; the case $d \ge 2$ is settled by Lemma \ref{lem:StarredSquare}. 
\item
For any $d \ge 1$, this is settled by Lemma \ref{lem:dMatsuda}. \qedhere
\end{compactenum}
\end{proof}

\vskip4mm
\subsection{Shortest dual paths and relation with traceability} \label{sec:ClosedTraceable}
As we saw in Lemma \ref{lem:bouquet}, there exist complexes like $U^2_3= 124, 345, 467$ that are closed but not weakly-closed. So at this point we owe the reader some explanation: Why did we (and before us, Matsuda \cite{Matsuda} and others, in the $1$-dimensional case) choose to call ``weakly-closed'' a property not implied by ``closed''? Here is the reason. 
We are going to show that \emph{all strongly-connected closed complexes are unit-interval} (Proposition \ref{prop:closed2}), so in particular under-closed, semi-closed, and weakly-closed. We will then prove that all such complexes are traceable (Theorem \ref{thm:CTSC}), which can be viewed as a higher-dimensional generalization of the graph-theoretical results by Bertossi \cite{Bertossi} and Herzog et al's \cite[Proposition 1.4]{HerzogEtAl}. The key to our generalization is  to focus on shortest paths in the \emph{dual} graph.

\begin{definition}
Let $F$ be a facet a pure $d$-dimensional simplicial complex $\Delta$. Let $v$ be a vertex of $\Delta$. A \emph{shortest path between $F$ and $v$} is  a path in the dual graph of $\Delta$ of minimal length from $F$ to some facet containing $v$. The \emph{distance} between $F$ and $v$ is the length of a shortest path, if any exists, or $+ \infty$, otherwise.
\end{definition}

\begin{definition} Let $\Delta$ be a pure $d$-dimensional simplicial complex, with vertices labeled from $1$ to $n$. 
A path $ F_0, \; F_1,  \: \ldots, \; F_\ell$ in the dual graph of $\Delta$ is called  \emph{ascending}, if each $F_i$ is obtained from $F_{i-1}$ by replacing the smallest vertex of   $F_{i-1}$, with a vertex greater than all remaining vertices of  $F_{i-1}$. A path is called \emph{descending},  if the reverse path is ascending.
\end{definition}

For example, suppose that a $2$-complex $\Delta$ contains the facets $124, 245, 456$, and $356$. The dual path they form is not ascending  -- or better, it is ascending, except for the last step. Such dual path demonstrates that the vertex $v=3$ is at distance $\le 3$ from $124$.  Now suppose that we know in advance that $\Delta$ is closed: Then from $356, 456 \in \Delta$, we immediately derive that $\Delta$ must contain the whole $2$-skeleton of the simplex $3456$. Note that the same conclusion could be reached also if we knew in advance that $\Delta$ is \emph{unit-interval}, rather than closed.  Either way: $\Delta$ contains the facet $G=345$ which contains $3$ and is adjacent to $245$. So  $124, 245, 345$ yields a ``shortcut'' to the original path, thereby proving that $v=3$ is actually at distance $\le 2$ from $124$. And it gets even better: Since $245$ and $345$ are in $\Delta$, by the closed assumption (or the unit-interval assumption) on $\Delta$, we may conclude that  $\Delta$ contains the whole $2$-skeleton of the simplex $2345$. So also $234$ is in $\Delta$, which means that $v=3$ is at distance $1$ from $124$. 

This example generalizes as follows, in what can be viewed as a higher-dimensional version of Cox--Erskine's narrowness property \cite{CoxErskine}:

\begin{lemma} \label{lem:closed}
Let $\Delta$ be a pure $d$-dimensional simplicial complex, with a labeling that makes it either closed or unit-interval. Let $F=a_0a_1\cdots a_d$ be a facet of $\Delta$. Let $v$ be  a vertex. If the distance between $F$ and $v$ is a finite number $\ell \ge 2$, then 
\begin{compactitem}
\item either there is a shortest path from $F$ to $v$ that is ascending (and thus $v > a_d$),
\item or there is a shortest path from $F$ to $v$ that is descending (and thus $v < a_0$).
\end{compactitem}
If instead $a_0 < v < a_d$, and some facet containing $v$ is in the same strongly-connected component of $F$, then the distance between $F$ and $v$ is at most one, and $\Delta$ contains the whole $d$-skeleton of the simplex on the vertex set $F \cup \{ v \}$.
\end{lemma}

\begin{proof} Let  \[F= F_0, \ldots, F_{i-1}, \; F_i, \; F_{i+1}\] be a shortest path from $F$ to a vertex $v \in F_{i+1}$. Suppose the path is ascending until $F_i$, but it stops being ascending when passing from $F_i$ to $F_{i+1}$. This means that $\max F_i = \max F_{i+1}$. By Remark~\ref{rem:almost}, $\Delta$ contains the whole $d$-skeleton of the simplex with vertex set $F_i \cup F_{i+1}$. In particular, if we set $\gamma \eqdef F_{i-1} \cap F_i$, the complex $\Delta$ contains  $G \eqdef \gamma \cup v$. But since $G$ is a $d$-face that contains $v$ and is already adjacent to $F_{i-1}$, \[F= F_0, \ldots, \; F_{i-1}, \; G\]
is a shorter path from $F$ to $v$ than the one we started with, a contradiction. 
The same argument applies to descending paths. If instead $a_0 < v < a_d$, clearly there cannot be any ascending or descending path from $F$ to $v$. So either $v \in F$, in which case the distance from $F$ to $v$ is $0$ and there is nothing to prove, or $v \notin F$, in which case the distance is $1$. In the latter case, $F$ and the adjacent face $G$ containing $v$ have same maximum, so again by Remark \ref{rem:almost} the complex $\Delta$ contains the $d$-skeleton of the simplex on $F \cup G = F \cup \{v \}$. 
\end{proof}

\begin{proposition} \label{prop:closed2} All strongly-connected closed simplicial complexes are unit-interval. 
\end{proposition}

\begin{proof}  Let $\Delta$ be a strongly-connected $d$-dimensional simplicial complex that is closed with respect to some-labeling. Let $F=a_0 a_1 \cdots a_d \in \Delta$. We claim the following:
\begin{compactitem}[ (*) ]
\item If there exist $m  \in \{1, \ldots, d\}$  and  $g_1, \ldots, g_m$ not in $F$, with $a_0 < g_1 < g_2 < \ldots < g_m < a_d$, then $\Delta$ contains the $d$-skeleton of the simplex with vertex set  $\{a_0, \ldots, a_d, g_1, \ldots, g_m\}$.
\end{compactitem}
If $\operatorname{gap}(F) = 0$, then the implication is trivially true, because the antecedent is never verified. So suppose $\operatorname{gap}(F) > 0$, and let us proceed by induction on $m$. 

For $m=1$: Pick a vertex $g$ of $\Delta$ not in $F$, with $a_0 < g < a_d$.  Since $\Delta$ is strongly connected, by the second part of Lemma~\ref{lem:closed} the complex $\Delta$ has a facet $G$ that contains $g$ and is adjacent to $F$.  Had $G$ neither same minimum nor same maximum of $F$, then either $G=a_1a_2 \cdots a_d g$ or $G=ga_0 a_1 \cdots a_{d-1}$. But both cases contradict the assumption $a_0 < g < a_d$. Hence, $F$ and $G$ have either same minimum or same maximum (or both), so they share at least one vertex in the same position. Since $\Delta$ is closed, $\Delta$ contains the $d$-skeleton of the simplex on $F \cup G = F \cup \{g \}$. 

For $m>1$: let $H$ be a subset of $\{a_0, \ldots, a_d, g_1, \ldots, g_m\}$ of cardinality $d+1$. If $H$ contains at most $m-1$ elements of $\{g_1 \ldots, g_m\}$, then we know that $H \in \Delta$ by the inductive assumption. If $g_1, \ldots, g_m$ are all vertices of $H$, let us consider a new face $H'$ with exactly the same vertices of $H$, except for one replacement, to be decided as follows: 
\begin{compactitem}
\item If $\min H = a_0$ and $\max H = a_d$, we shall replace $g_1$ with any vertex $v$ of $F$ that is not in~$H$. This way, since $a_0 \le v \le a_d$, we have that as real intervals \[(\min H , \max H)  = (a_0, a_d) = (\min H' , \max H').\] 
\item If $\min H = g_1$, or if $\min H = a_i$ for some $i >0$, we shall replace $g_1$ with $a_0$. This way \[(\min H , \max H) \subsetneq (a_0, \max H) = (\min H', \max H').\]
\item If $\max H = g_m$, or $\max H = a_i$ for some $i<d$, we shall replace $g_m$ with $a_d$. This way \[(\min H , \max H) \subsetneq (\min H, a_d) = (\min H', \max H').\]
\end{compactitem}
In all three cases, if $w$ is the only element that belongs to $H$ but not to $H'$, then $w$ is either $g_1$ or $g_m$, and we have  \[\min H' < w < \max H'.\] Moreover, $H'$ contains at most $m-1$ elements of $\{g_1 \ldots, g_m\}$, %(because $H$ contained $m$ of them, and we got rid of one). 
so by the inductive assumption  $H'$ is in $\Delta$. But since $\min H' < w < \max H'$, by the second part of Lemma \ref{lem:closed} we conclude that also $H$ is in $\Delta$. By the genericity of $H$, this proves Claim (*). From the Claim the conclusion follows immediately, by choosing $m$ maximal.
\end{proof}

\begin{remark}
The converse is false: The complex with $k$ disjoint $d$-simplices is obviously not strongly-connected, yet it is unit-interval with the natural labeling below:
\[\Delta = H_1, \: H_{d+2}, \: H_{2d+3},  \ldots, H_{(k-1)d+k}.\] 
For connected graphs, it is obvious that ``closed'' and ``unit-interval'' are the same: This is noticed also in Matsuda~\cite[Proposition 1.3]{Matsuda} and Crupi--Rinaldo \cite{Rinaldo}. However, as we saw in Lemma \ref{lem:Bd}, higher-dimensional complexes that are both strongly-connected and unit-interval might not be closed. 
\end{remark}

We have arrived to the main result of this section, the generalization of Bertossi's theorem:

\begin{theorem}[Higher-dimensional Bertossi] \label{thm:CTSC}
Let $\Delta$ be a pure $d$-dimensional simplicial complex that is either closed or unit-interval. Then 
\[ \Delta \textrm{ is strongly-connected }\:  \ \Longleftrightarrow \ \: \Delta \textrm{ is traceable.}  \] 
\end{theorem}

\begin{proof}\hspace*{\fill}
\begin{compactdesc}
\item{$\Leftarrow$:}  Let $F$ be a $d$-face of $\Delta$. We want to find a walk from $\Delta$ to $H_1$ in the dual graph. If $\operatorname{gap} F= 0$, then $F=H_j$ for some $j$, and $H_1, H_2, \ldots, H_j$ is the desired path. If $\operatorname{gap} F> 0$, let $i \eqdef \min F$. Since $F$ and $H_i$ have same minimum, by Remark  \ref{rem:almost} $\Delta$ contains the whole $d$-skeleton of the simplex on $F \cup H_i$. But the $d$-skeleton of a higher-dimensional simplex is strongly-connected, which means that in the dual graph of $\Delta$ we can walk from $F$ to $H_i$.  And since $H_i$ has gap $0$, we can walk from it to $H_1$.
\item{$\Rightarrow$:}   Fix a labeling for which $\Delta$ is (almost-)closed. We are going to show by induction on $j$ that with the \emph{same} labeling, every $H_j$ is in $\Delta$. For $j=1$, since $\Delta$ is pure, it contains a face $F=a_0 a_1 \cdots a_d$ with $a_0=1$, and then it is easy to derive (either directly, or using that the labeling satisfies the under-closed condition by Theorem \ref{thm:Hierarchy}) that $H_1$ is in $\Delta$.
Now suppose that $\Delta$ contains $H_j$ and let us show that $\Delta$ contains  $H_{j+1}$. By Lemma \ref{lem:closed}, $\Delta$ has a $d$-face $H'$ that contains $d+j+1$ and is adjacent to $H_j$. Such $H'$ has the same vertices of $H_j$, with the exception of a single vertex $i$ that was replaced by $d+j+1$. Now either $i=j$, in which case $H' = H_{j+1}$ and we are done; or $i > j$. If  $i> j$, then $j$ was not replaced, so it is still present in $H'$. Hence $H'$ and $H_j$ are adjacent faces with the same minimum, namely, $j$. By Remark \ref{rem:almost}, this implies that $H_{j+1}$ is in $\Delta$.  \qedhere
\end{compactdesc}
\end{proof}

\begin{remark} \label{rem:WCareSC} If the ``unit-interval'' assumption is weakened to ``under-closed'', then the direction ``$\Rightarrow$'' of Theorem \ref{thm:CTSC} no longer holds, with $K_{1,3}$ playing the usual role of the counterexample. The direction ``$\Leftarrow$'' instead is still valid. We claim in fact that \textit{all weakly-closed traceable complexes are strongly-connected}. To see this, it suffices to show that from any $d$-face $F$ of positive gap we can walk in the dual graph to some gap-$0$ face. But the weakly-closed definition tells us how to move in the dual graph from $F$ to a face $F'$ of smaller gap than $F$. So if we iterate this, eventually we get from $F$ to a gap-0 face. (The same type of argument is carried out in details in the proof of Theorem \ref{thm:Hi}, item (5),  below.) That said, the ``weakly-closed'' assumption is needed for ``$\Leftarrow$''. In fact, for any $d \ge 2$, if $G_d \eqdef \{1, d+2, 2d+3, \ldots, (k-1)d+k, kd+(k+1), \ldots, d^2+d+1\}$, then the traceable $d$-complex with $d^2+d+1$ vertices  $\Delta = H_1, H_2, \ldots, H_{d^2 }\,, H_{d^2+1 }\,, G_d\; $ is not strongly-connected. Its dual graph is a path of length $d^2+1$ plus an isolated vertex.
\end{remark}

Generalizing a result by Chen, Chang, and Chang \cite[Theorem 2]{ChChCh}, we can push Theorem \ref{thm:CTSC} a bit further. If $D$ is a simplicial complex obtained from $\Delta$ by deleting some vertices $v_1, \ldots, v_k$, then any labeling of $\Delta$ naturally induces a \emph{compressed labeling} for $D$, just by ordering the vertices of $D$ in the same way as they are ordered inside $\Delta$.  For example, if  $\Delta=123, 134, 345$, the compressed labeling for $D=\operatorname{del}(2, \Delta)$ is $123, 234$. A priori, this $D$ need not be pure. 

\begin{lemma} \label{lem:compressed}
Let $\Delta'$ be a $d$-dimensional simplicial complex obtained by deleting some vertices from a $d$-dimensional simplicial complex $\Delta$. 
If $\Delta$ is unit-interval (resp. under-closed, resp. semi-closed), then so is $\Delta'$. 
\end{lemma}

\begin{proof} 
If the original labeling satisfied the unit-interval (resp. under-closed, resp. semi-closed) condition, so does the compressed labeling.
\end{proof}

\begin{lemma} \label{lem:ChChCh} Let $\Delta$ be a $d$-dimensional strongly-connected simplicial complex, with a labeling that makes it unit-interval. The following are equivalent:
\begin{compactenum}[ \rm (a)]
\item The deletion of $d$ or less vertices, however chosen, yields a $d$-complex that is strongly connected. %whose pure part is 
\item The deletion of $d$ or less vertices, however chosen, yields a pure $d$-complex  that  with the compressed labeling is traceable.  %whose pure part is 
\item $\Delta$ contains all faces of gap $\le d$.
\end{compactenum}
\end{lemma}

\enlargethispage{4mm}

\begin{proof} 
\begin{compactdesc}
\item{(a) $\Leftrightarrow$ (b):} By Lemma \ref{lem:compressed} the compressed labeling satisfies the unit-interval condition. Via Theorem \ref{thm:CTSC}, we conclude.
\item{(b) $\Rightarrow$ (c):} By deleting zero vertices we notice that $\Delta$ is itself traceable. Let $F=a_0 \cdots a_d $ be any $d$-face of $\Sigma^d_n$ that has gap $\le d$. If $\operatorname{gap}(F)=0$, then $F$ is one of $H_1, \ldots, H_{n-d}$, so  $F$ is in $\Delta$ by definition of traceable. Otherwise,  set $S_F \eqdef \{ j \notin F \textrm{ such that } a_0 < j < a_d\}$. Let $\Delta'$ be the complex obtained from $\Delta$ by deleting the vertices in $S_F$, which are at most $d$. By assumption, $\Delta'$ is traceable with the ``compressed labeling''. So $\Delta'$ contains a gap-0 face of minimum $a_0$. But by how the compressed labeling is defined, this face has exactly the vertices that  in the original labeling for $\Delta$ were called $a_0, a_1, \ldots, a_d$. So $F$ is in $\Delta$. 
\item{(c) $\Rightarrow$ (b):} Let $\Delta'$ be the $d$-complex resulting from the deletion. 
With the compressed labeling, $\Delta'$ is traceable, because any gap-$0$ $d$-face of $\Delta'$ with the compressed labeling, is a $d$-face of $\Delta$ that had gap $\le d$ in the original labeling. It remains to see that $\Delta'$ is pure. We prove that $\Delta'$ has no facets of dimension $d-1$, leaving the case of facets of even lower dimensions to the reader. We claim that every $(d-1)$-face $\sigma$ of $\Delta$ lies in at least $d+1$ distinct $d$-faces of $\Delta$. From the claim the conclusion follows via the pigeonhole principle: If we delete $d$ vertices, however chosen, then at least one of the $d$-faces containing $\sigma$ will survive the deletion, which implies that $\sigma$ is not a facet in $\Delta'$. \\
So let us prove the claim. Let $\sigma= b_0 \cdots b_{d-1}$. If $b_{d-1} - b_0 -d +1 \eqdef \operatorname{gap} (\sigma) \le  d$, then  $b_{d-1} + 1 \le b_0 + 2d$. So for each $i$ in the $(d+1)$-element set 
\[ T_{\sigma} \eqdef \{ b_0, b_0 + 1, \ldots, b_{d-1}, b_{d-1} + 1, \ldots, b_0 + 2d \} \:  \setminus \: \{ b_0, b_1, \ldots, b_{d-1}\}\]  
the $d$-face $\sigma \cup \{i\}$  has gap $\le d$, and thus is in $\Delta$ by assumption. If instead  $\operatorname{gap} (\sigma) \ge  d+1$, we use the unit-interval assumption: 
for every $i$ in $S_{\sigma} \eqdef \{ i \notin \sigma \textrm{ such that } b_0 < i < b_{d-1}\}$, the $d$-face $\sigma \cup \{i\}$ is in $\Delta$. So either way the claim is proven. \qedhere
\end{compactdesc}
\end{proof}

\begin{theorem}[Higher-dimensional Chen--Chang--Chang] \label{thm:Hi0}
Let $\Delta$ be a pure $d$-dimensional simplicial complex.
\begin{compactitem}
\item If $\Delta$ is unit-interval and the deletion of $\le d$ vertices, however chosen, yields a strongly-connected $d$-complex, then $\Delta$ is Hamiltonian.
\item If $\Delta$ is weakly-closed and Hamiltonian, the deletion of $\le 1$ vertices, however chosen, yields a strongly-connected $d$-complex.
\end{compactitem}
\end{theorem}

\begin{proof} For the second claim: Up to a cyclic reshuffling, the vertex we wish to delete is $n$. The argument of Remark \ref{rem:WCareSC} yields a dual path in $\Delta$ from each $d$-face $F$ to $H_1$. If $F$ does not contain $n$, none of the $d$-faces in such dual path does, so the path belongs to the dual graph of the deletion of $n$ from $\Delta$.

Now we prove the first claim. By Lemma \ref{lem:ChChCh}, $\Delta$ contains all $d$-faces of gap $\le d$. In particular: 
\begin{compactitem}
\item for any odd $i$ such that $1 \le i \le n-2d$, $\Delta$ contains the gap-$d$ face $O_i$ formed by $i$ and by the first $d$ consecutive odd integers after~$i$; 
\item for any even $j$ such that $2 \le j \le n-2d$, $\Delta$ contains the gap-$d$ face $E_j$ formed by $j$ and by the first $d$ consecutive even integers after $j$; 
\item $\Delta$ contains  the gap-$(d-1)$ face $F=1,2,4,  \ldots, 2d$
formed by $1$ and by the $d$ smallest even natural numbers;
\item $\Delta$ contains  the gap-$(d-1)$ face $G$ formed by the largest even integer $\le n$ and by the $d$ largest odd integers $\le n$. 
\end{compactitem}
Now consider the following sequence $\mathfrak{C}$ of $d$-faces in $\Delta$: First all $O_i$'s in increasing order, then $G$, then all $E_j$'s in decreasing order, then $F$. Note that any two $O_i$'s are adjacent, and  the last of them is adjacent to $G$; symmetrically, any two $E_j$'s are adjacent, and $F$ is adjacent to $E_2$. We claim that this sequence would form a weakly-Hamiltonian cycle if we relabeled the vertices of $\Delta$ first by listing the odd ones increasingly, and then the even ones decreasingly. 

Formally, if $n$ is odd, we introduce the new labeling
\[ \ell_1 \eqdef 1, \, \ell_2 \eqdef 3, \,  \ell_3  \eqdef  5, \ldots, \, \ell_{\frac{n+1}{2}} \eqdef n, \,   \ell_{\frac{n+1}{2} + 1} \eqdef  n-1,  \,  \ell_{\frac{n+1}{2} + 2}\eqdef n-3, \ldots,  \,   \ell_{n-1} \eqdef 4,  \,  \ell_{n} \eqdef 2.\]
And if instead $n$ is even,  we introduce the new labeling
 \[ \ell_1 \eqdef  1, \,  \ell_2 \eqdef 3, \,  \ell_3  \eqdef  5, \ldots, \, \ell_{\frac{n}{2}} \eqdef n-1,  \, \,  \ell_{\frac{n}{2} + 1} \eqdef n,  \,\,  \ell_{\frac{n+1}{2} + 2} \eqdef  n-2, \ldots,  \,   \ell_{n-1} \eqdef 4,  \,  \ell_{n} \eqdef 2.\]
Let us set $L_1 \eqdef \ell_1 \ell_2 \cdots \ell_{d+1}$,  $L_2 \eqdef \ell_2 \ell_3 \cdots \ell_{d+2}$, and so on.
Then the sequence $\mathfrak{C}$ described above is equal (whether $n$ is even or odd) to
 \[ L_1, \, L_2, \, \ldots, \, L_{\lfloor \frac{n+1}{2} \rfloor- (d-1)}, \, L_{\lfloor \frac{n+1}{2} \rfloor + 1}, \, L_{\lfloor \frac{n+1}{2}\rfloor + 2}, \ldots, L_{n - d}, L_{n - (d-1)}.\]

This shows that with the new labeling $\Delta$ is \emph{weakly-}Hamiltonian. It remains to show for $d \ge 2$ that our weakly-Hamiltonian cycle can indeed be `completed' to a Hamiltonian cycle, in the sense that the $L_i$'s that were not mentioned in $\mathfrak{C}$ are anyway contained in $\Delta$. First of all, note that $\Delta$ with the original labeling contained all the  $d$-faces of gap $\le d$, so  in particular it contained all $d$-faces containing $1$ and with vertex set contained in $F \cup O_1$. This shows that with the new labeling, $L_{n-(d-2)}$, $\ldots$, $L_{n}$ are all in $\Delta$. So it remains to consider the missing $L_i$'s from the `center' of the sequence  $\mathfrak{C}$. For the ``$n$ odd'' case (the  case for $n$ even is analogous), we have to see whether $\Delta$ contains also the $d-1$ facets
\[  L_{\frac{n+1}{2} - d + 2},  \: L_{\frac{n+1}{2} - d+ 3}, \:  \ldots, \: L_{\frac{n+1}{2}}.\]
When we translate these $d$-faces back into the \emph{old} labeling, it is easy to see that the face with the largest gap is the last one, which has gap $d-1$. So all these faces are in $\Delta$ by assumption. 
\end{proof}

\begin{example} %n odd 
Let $\Delta$ be an unit-interval $3$-complex on $n=9$ vertices that contains all tetrahedra with gap $\le 3$. With the notation of Theorem \ref{thm:Hi0} the complex  $\Delta$ contains the sequence $\mathfrak{C}$ below:
\[O_1=1357, \: O_2=3579, \:G = 5789, \:E_2=2468,  \:F=1246.\] If we relabel the vertices as in the proof of Theorem \ref{thm:Hi0}, the list above becomes
\[L_1, \:L_2, \:L_3, \:L_6, \:L_7.\] 
Thus $\Delta$ is weakly-Hamiltonian. To prove that it is Hamiltonian, we need to check that $L_4, L_5$ and $L_8, L_9$ are in $\Delta$. Translated into the original labeling, this means checking that $6789, 4689$ and $1234, 1235$ are in $\Delta$, which is clearly the case because they all have gap $\le 2$. 
\end{example}

\begin{remark}
For $d=1$, Theorem~\ref{thm:Hi0} boils down to Chen--Chang--Chang's result that ``unit interval graphs are Hamiltonian if and only if they are $2$-connected'' \cite[Theorem 2]{ChChCh}. The $G_5$ of Remark \ref{rem:G5} is $2$-connected and not Hamiltonian; hence the ``unit-interval'' assumption in the first claim of Theorem~\ref{thm:Hi0} is necessary. As for the second claim, the ``weakly-closed'' assumption is necessary for $d>1$, because we saw in Remark  \ref{rem:HnotSC} that some Hamiltonian $d$-complexes are not strongly-connected. 
\end{remark}

We may condense most of the results of this chapter in the following summary:

\begin{theorem} \label{thm:Hi}
Let $\Delta$ be a $d$-dimensional simplicial complex. 
\begin{compactenum}[\rm (1) ]
\item  If $\Delta$ is closed (or unit-interval) and strongly connected, then $\Delta$ is traceable. % (with the same labeling). 
\item If $\Delta$ is closed (or unit-interval), and the deletion of $d$ or less vertices, however chosen, yields a strongly connected complex, then  $\Delta$ is Hamiltonian.
\item If $\Delta$ is under-closed, it contains $H_1$. If in addition $\Delta$ has a face of minimum $i$ for each $i\in \{2, \ldots, n-d\}$, then $\Delta$ is traceable. % (with the same labeling). %has a face with minimum $i$ for each $i\in \{2, \ldots, n-d\}$, then $\Delta$ is traceable.
\item If $\Delta$ is semi-closed, then for every face $F=a_0 \cdots a_d$ of $\Delta$  either $H_{a_0}$ or $H_{a_d-d}$ is in $\Delta$. % (with the same labeling). %However, it could be that $|delta$ 
\item If $\Delta$ is weakly-closed, then $\Delta$ contains at least one of the  $H_i$'s. \\If in addition $\Delta$ contains $H_1$, plus a face with minimum $i$ and of  gap smaller than $d$  for each $i$ in  $\{2, \ldots, n-d\}$, then $\Delta$ is weakly-traceable. % (with the same labeling).
\end{compactenum}
\end{theorem}

\begin{proof} 
\begin{compactenum}[\rm (1) ]
\item This is given by Proposition \ref{prop:closed2} and Theorem \ref{thm:CTSC} above. 
\item This is given by Proposition \ref{prop:closed2} and Theorem \ref{thm:Hi0} above. 
\item By definition of under-closed, if $\Delta$ has a face of minimum $i$, then $\Delta$ contains $H_i$. The fact that $\Delta$ has  a face of minimum $1$ follows from the assumption that $\Delta$ is pure. 
\item This is straightforward from the definition of semi-closed. 
\item  Let $F=a_0 a_1 \cdots a_d$ be any facet of $\Delta$ with $\operatorname{gap}(F)>0$. Let  $g \notin F$ such that $a_0 < g < a_d$. By definition of ``weakly-closed'', some face $G=b_0 b_1 \cdots b_d$ of $\Delta$ contains $g$, is adjacent to $F$, and has either $b_0 \ne a_0$ or $b_d \ne a_d$. Thus $\operatorname{gap} G < \operatorname{gap} F$. Iterating the process, eventually we find in $\Delta$ a gap-$0$ face, which has to be one of  
\[H_{a_0},  \; H_{a_0+1}, \ldots, H_{a_d - d}. \] 
As for the second claim: By assumption, $\Delta$ contains $H_1$. Also, $\Delta$ contains $H_{n-d}$, because no other face has minimum $n-d$. Now let  $H'=2a_1\cdots a_d$ be a face of $\Delta$ with minimum $2$ and gap $\le d-1$. By the argument above, we know that $\Delta$ must contain at least one of 
\[H_2, H_3, \ldots, H_{a_d-d}.\]
Let us call this face $H_{i_2}$.  By how $H'$ was chosen, 
\[ 2 \le i_2 \le a_d - d = \operatorname{gap}(H')+2 \le d+1.\]
But since $H_1$ contains all vertices from $1$ to $d+1$, in particular it contains $i_2$. So $H_{i_2}$ is incident with $H_1$. Now let $H''=a_0a_1\cdots a_d$ be a face of $\Delta$ with gap smaller than $d$, and minimum $a_0=i_2+1$. Repeating the argument above, $\Delta$ contains one of 
\[H_{i_2+1}, H_{i_2+2}, \ldots, H_{a_d-d}.\]
Call this facet $H_{i_3}$; as above, it must intersect $H_{i_2}$. And so on. Eventually, we obtain a list $H_1 = H_{i_1}, H_{i_2} \ldots, H_{i_{k-1}}, H_{i_k} = H_{n-d}$ of facets of $\Delta$ that makes it weakly-traceable.
 \qedhere
\end{compactenum}
\end{proof}

\begin{remark}
In the previous theorem, a relabeling was necessary only to prove item (2). For  all other items, the \emph{original} labeling was already suitable for the desired conclusion. So for item (1) we proved a slightly stronger statement: ``If $\Delta$ is strongly-connected, then any labeling that makes $\Delta$ unit-interval automatically makes $\Delta$ traceable''. Same for items (3), (4), (5).
\end{remark}

\newpage

\section{Algebraic motivation}
In this section, we review Ene et al's definition of determinantal facet ideals \cite{EneEtAl}. We find out a large class of them that are radical. In fact, we prove the following:
\begin{compactitem}
\item If a simplicial complex is semi-closed, then its determinantal facet ideal has a square-free Gr\"obner degeneration (and in particular is radical), and the quotient by such ideal in positive characteristic is $F$-pure (Theorem \ref{t:s-c-f}).
\item If the simplicial complex is unit-interval, then the natural generators of its determinantal facet ideal form a Gr\"obner basis with respect to a diagonal term order  (Theorem \ref{t:a-c-gb}). Moreover, the converse is true if with respect to the same labeling, the simplicial complex is traceable (Theorem \ref{t:a-c-gb1}). %toglierei ``not necessarily semi-closed'' perche' alla fine e' implicato da unit-interval, si crea un po' di confusione, sembra che ne esistano di non-semi-closed
\end{compactitem}

\subsection{A foreword on $F$-pure rings, $F$-split rings, and Knutson ideals}

Let $p$ be a prime number. Let $R$ be a ring of characteristic $p$. Recall that the \emph{Frobenius map} is the  ring homomorphism from $R$ to itself that maps an element $r \in R$ to $r^p$.
We denote by $F_*R$ the $R$-module defined as follows:  $F_*R \eqdef R$ as additive group, and
$r\cdot x \eqdef r^px$ for all $r\in R$ and $x\in F_*R$.  This allows us to view the Frobenius map as a map of $R$-modules,
\begin{eqnarray*}
F: & R\longrightarrow F_*R \\
& r\mapsto r^p.
\end{eqnarray*}

The ring $R$ is reduced if and only if $F$ is injective. So the following definitions are natural:

\begin{definition}
$R$ is \emph{$F$-pure} if $F\otimes 1_M:M\rightarrow F_*R\otimes_R M$ is injective for any $R$-module $M$. 
\end{definition}

\begin{definition}
$R$ is \emph{$F$-split} if there exists a homomorphism $\theta:F_*R\rightarrow R$ of $R$-modules such that $\theta\circ F=1_R$. Such a $\theta$ is called an $F$-splitting of $R$. 
\end{definition}

If a ring is $F$-split, it is clearly $F$-pure. The converse does not hold in general. However, the two concepts are equivalent in a number of cases, for example:

\enlargethispage{4mm}
\begin{lemma}\label{l:split=pure}
Let $R=\bigoplus_{i\in\ZZ}R_i$ be a Noetherian graded ring of characteristic $p$ having a unique homogeneous ideal $\mm$ that is maximal with respect to inclusion. Furthermore, assume that the Noetherian local ring $R_0$ is complete. Then the following are equivalent:
\begin{compactenum}[ \rm (a)]
\item $R$ is $F$-split.
\item $R$ is $F$-pure.
\item $F\otimes 1_E:  \: E \, \longrightarrow \, F_*R\otimes_RE$ is injective, where $E$ is the injective hull of $R/\mm$.
\end{compactenum}
\end{lemma}
\begin{proof}
$(a)\implies (b)\implies (c)$ are obvious implications. To see $(c)\implies (a)$: the map 
\[F\otimes 1_E: \: E \,\longrightarrow \: F_*R\otimes_RE\] is injective if and only if the corresponding map 
\[\Hom_R(F_*R,\Hom_R(E,E))\cong \Hom_R(F_*R\otimes_RE,E)\, \longrightarrow \, \Hom_R(E,E)\]
is surjective. Hence, by \cite[Corollary 3.6.7, Proposition 3.6.16, Theorem 3.6.17]{BH93}, the corresponding map $\alpha:\Hom_R(F_*R,R)\rightarrow R$ is surjective. So there exists $\theta\in \Hom_R(F_*R,R)$ such that $\alpha(\theta)=1$. On the other hand, by construction $\alpha(\theta)=\theta(F(1))$, so $\theta \circ F=1_R$. %, i.e. $\theta$ is an $F$-splitting of $R$.
\end{proof}

Since we want to study homogeneous quotients of a polynomial ring over a field, by Lemma~\ref{l:split=pure}  we may as well regard  the $F$-split notion and the $F$-pure notion as equivalent. 

\medskip
In the following the concept of {\it Knutson ideal} will be fundamental. The name arises from the work of Knutson \cite{Knutson}, later systematically investigated by the second author \cite{Se1}, who extended several properties from $\ZZ/p\ZZ$ to any field. The result from \cite{Se1} that we shall need is the following:

\begin{theorem}[{Seccia \cite{Se1}}]\label{t:lisa}
Let $K$ be a field. Let $g\in S=K[x_1,\ldots ,x_n]$ be a polynomial with $\init_<(g)$ square-free for some term order on $S$. Let $\C_g$ be the smallest set of ideals of $S$ containing $(g)$ and such that: 
\begin{compactenum}
\item $I\in \C_g \implies I:h\in \C_g$ whenever $h\in S$,
\item $I,J\in \C_g \implies I+J\in \C_g, \ I\cap J\in \C_g$.
\end{compactenum}
If $I\in \C_g$, then $\init_<(I)$, and therefore $I$, is radical. Furthermore, if $I,J\in \C_g$, then $\init_<(I+J)=\init_<(I)+\init_<(J)$ and $\init_<(I\cap J)=\init_<(I)\cap\init_<(J)$. Finally, if $K$ has positive characteristic, $S/I$ is $F$-pure whenever $I\in\C_g$.
\end{theorem}

\begin{example}
It can be shown that, if $g=x_1x_2\cdots x_n$,  then $\C_g$ is the set of squarefree monomial ideals.
\end{example}
 
\subsection{Determinantal facet ideals: basic properties}
Let $d,n$ be positive integers with $d+1\leq n$. Let $S \eqdef K[x_{ij}:i=1,\ldots ,n, j=0,\ldots ,d]$ be a polynomial ring in $(d+1)n$ variables over some field $K$. Set
\[X=\begin{bmatrix}
x_{01}&x_{02}&\ldots &x_{0n}\\
x_{11}&x_{12}&\ldots &x_{1n}\\
\vdots&\vdots&\ldots&\vdots\\
x_{d1}&x_{d2}&\ldots &x_{dn}
\end{bmatrix}.\]
Given $1\leq r\leq d$, and integers $0\leq a_0<a_1<\ldots <a_r\leq d$ and $1\leq b_0<\ldots <b_r\leq n$, an \emph{$(r+1)$-minor of $X$} is any element of the form \[ [a_0a_1\ldots a_r|b_0b_1\ldots b_r] \ \eqdef  \ \det \; 
\begin{bmatrix}
x_{a_0b_0}&x_{a_0b_1}&\ldots & x_{a_0b_r}\\
x_{a_1b_0}&x_{a_1b_1}&\ldots & x_{a_1b_r}\\
\vdots&\vdots&\ldots&\vdots\\
x_{a_rb_0}&x_{a_rb_1}&\ldots & x_{a_rb_r}\\
\end{bmatrix}.\]
If $r=d$, the row indices are forced to be $a_0=0,a_1=1,\ldots ,a_d=d$. For this reason we denote $[01\ldots d|b_0b_1\ldots b_d]$ simply by $[b_0b_1\ldots b_d]$. The ideal of $S$ generated by the $r+1$-minors of $X$ is denoted by $I_{r+1}(X)$. This ideal defines the variety of $(d+1)n$ matrices with entries in $K$ and with rank at most $r$. 
The set $\Pi$ of all the minors of $X$ can be partially ordered by the relation
\[[a_0a_1\ldots a_r|b_0b_1\ldots b_r]\leq [c_0c_1\ldots c_s|d_0d_1\ldots d_s] \iffdef r\geq s, \ \ a_i\leq c_i \mbox{ and } b_i\leq d_i \ \forall \ i=0,\ldots ,s.\]
In particular, for maximal minors the previous definition restricts to 
\[[a_0a_1\ldots a_d]\leq [b_0b_1\ldots b_d] \iff a_0\leq b_0, a_1\leq b_1,\ldots ,a_d\leq b_d.\]
It is not our intent to review the theory of {\it Algebras with Straightening Law} here, as the interested reader can learn it directly from the standard source \cite{BrunsVetter}. However, we wish to introduce a few concepts for the sake of clarity. The starting observation is that the polynomial ring $S$ is generated by $\Pi$ as a $K$-algebra. In fact, a basis of $S$ as $K$-vector space is given by
\[\{\pi_1\cdots \pi_m:m\in\NN, \ \pi_i\in\Pi, \ \pi_1\leq \pi_2\leq \ldots \leq \pi_m\}.\]
The elements of this $K$-basis are called {\it standard monomials}. It may happen that the product of two standard monomials is not a standard monomial. However, such product will be uniquely writable as $K$-linear combination of standard monomials, which is in some sense compatible with the poset structure on $\Pi$. This is what is known as `Straightening Law'; compare \cite[Theorem 4.11]{BrunsVetter}. What we wish to outline is that the ideals of $S$ generated by poset ideals of $\Pi$ (i.e. subsets $\Omega\subset \Pi$ such that for all $\omega\in \Omega$, $\pi\in\Pi$, $\pi\leq \omega \implies \pi\in\Omega$) are particularly nice. %We make an example to gain a bit of confidence.

\begin{example}
For any $r\leq d$, the ideal $I_{r+1}(X)$ is generated by the poset ideal $\Omega_{\geq r+1}$ of all $t$-minors of $X$ with $t\geq r+1$. This $\Omega_{\geq r+1}$ has a unique maximal element, $[d-r\ldots d|n-r\ldots n]$.
\end{example}

Some new notation: if $1\leq i<j\leq n$, by $X_{[i,j]}$ we mean the matrix
\[X_{[i,j]}=\begin{bmatrix}
x_{0i}&x_{0,i+1}&\ldots &x_{0j}\\
x_{1i}&x_{1,i+1}&\ldots &x_{1j}\\
\vdots&\vdots&\ldots&\vdots\\
x_{di}&x_{d,i+1}&\ldots &x_{dj}
\end{bmatrix},\]
so $I_{r+1}(X_{[i,j]})$ is the ideal of $S$ generated by the $r+1$-minors of $X_{[i,j]}$, whenever $r\leq \min\{d,j-i\}$.

Eventually, we say that a term order $<$ on $S$ is a {\it diagonal term order} if, for all $1\leq r\leq d$ and integers $0\leq a_0<a_1<\ldots <a_r\leq d$ and $1\leq b_0<\ldots <b_r\leq n$, $\init_<([a_0a_1\ldots a_r|b_0b_1\ldots b_r])=x_{a_0b_0}x_{a_1b_1}\cdots x_{a_rb_r}$. For example, the lexicographic term order on $S$ extending the linear order of the variables given by $x_{ij}>x_{hk}$ if and only if $i<h$ or $i=h$ and $j<k$ is a diagonal term order. We will use the following result from \cite{St}:

\begin{theorem}[Sturmfels \cite{St}]\label{t:sturmfels}
If $<$ is a diagonal term order, $1\leq i<j\leq n$ and $r\leq\min\{d,j-i\}$, then $\{[a_0a_1\ldots a_r|b_0b_1\ldots b_r]: 0\leq a_0<a_1<\ldots <a_r\leq d \mbox{ and } i\leq b_0<\ldots <b_r\leq j\}$ is a Gr\"obner basis of the $I_{r+1}(X_{[i,j]})$.
\end{theorem}

So far, by a ``simplicial complex on $n$ vertices'' we have always implicitly assumed that each vertex $i=1,\ldots ,n$ appears in the complex. From now on, we will drop this convention, i.e. henceforth a simplicial complex on a set $A$ is also a simplicial complex on any finite set $B\supset A$.

\begin{definition} Let $\Delta$ be a $d$-dimensional simplicial complex on $n$ vertices. Let $K$ be any field. Let $S = K[x_{ij}:i=1,\ldots ,n, j=0,\ldots ,d]$. The \emph{determinantal facet ideal} of $\Delta$ is the ideal
\[J_{\Delta}:=([a_0a_1\ldots a_d]:a_0a_1\ldots a_d \in \Delta)\subset S.\]
\end{definition}

When $d=1$, then $\Delta$ is a graph, and $J_{\Delta}$ is the {\it binomial edge ideal} of $\Delta$. Binomial edge ideals have been intensively studied in the recent literature: Among the many papers on this topic, see for example \cite{HerzogEtAl}, \cite{Ohtani0}, \cite{MaMu}, \cite{Matsuda}. Unlike binomial edge ideals, determinantal facet ideals are not always radical -- not even if the complex is weakly-closed:

\begin{example}\label{ex:notradical}
Consider the weakly-closed $2$-dimensional simplicial complex on five vertices
\[ \Delta = 124, 145, 234, 345.\]
Thus in the polynomial ring with $15$ variables $x_{i,j}$, for $i \in \{0,1,2\}$ and $j \in \{1, \ldots, 5\}$,  the ideal $J_{\Delta}$ is generated by the four degree-3 polynomials
\[  -x_{0,4}x_{1,2}x_{2,1} + x_{0,2}x_{1,4}x_{2,1}+x_{0,4}x_{1,1}x_{2,2}-x_{0,1}x_{1,4}x_{2,2}-x_{0,2}x_{1,1}x_{2,4}+x_{0,1}x_{1,2}x_{2,4} \:, \]
\[    -x_{0,5}x_{1,4}x_{2,1}+x_{0,4}x_{1,5}x_{2,1}+x_{0,5}x_{1,1}x_{2,4}-x_{0,1}x_{1,5}x_{2,4}-x_{0,4}x_{1,1}x_{2,5}+x_{0,1}x_{1,4}x_{2,5} \:, \]
\[    -x_{0,4}x_{1,3}x_{2,2}+x_{0,3}x_{1,4}x_{2,2}+x_{0,4}x_{1,2}x_{2,3}-x_{0,2}x_{1,4}x_{2,3}-x_{0,3}x_{1,2}x_{2,4}+x_{0,2}x_{1,3}x_{2,4} \: , \]
\[    -x_{0,5}x_{1,4}x_{2,3}+x_{0,4}x_{1,5}x_{2,3}+x_{0,5}x_{1,3}x_{2,4}-x_{0,3}x_{1,5}x_{2,4}-x_{0,4}x_{1,3}x_{2,5}+x_{0,3}x_{1,4}x_{2,5} \: . \]
It can be checked using the software Macaulay 2 \cite{m2} that $J_{\Delta}$ is not radical. 
\end{example}

Determinantal facet ideals are multi-graded. To see this, we endow $S$ with the multi-grading defined by $\deg(x_{ij})={\bf e_j}\in\NN^n$ for all $i=0,\ldots ,d, \ j=1,\ldots ,n$. Here ${\bf e_j}$ is the vector with a one in position $j$, and zeroes everywhere else.
With such grading $J_{\Delta}$ is homogeneous, and $S/J_{\Delta}$ admits a multi-graded minimal free resolution
\[0\rightarrow \bigoplus_{{\bf v}\in\NN^n}S(-{\bf v})^{\beta_{p,{\bf v}}}\rightarrow \ldots \rightarrow \bigoplus_{{\bf v}\in\NN}S(-{\bf v})^{\beta_{1,{\bf v}}}\rightarrow S\rightarrow S/J_{\Delta}\rightarrow 0,\]
where $p$ is the projective dimension of $S/J_{\Delta}$. We set $|{\bf v}|=v_1+\ldots +v_n$ for each ${\bf v}=(v_1,\ldots ,v_n)\in \NN^n$; this way the {\it graded Betti numbers} with respect to the standard grading are 
\[\beta_{i,j}=\sum_{\substack{{\bf v}\in\NN^n \\ |{\bf v}|=j}}\beta_{i,{\bf v}}.\]
In particular, $\mathrm{reg}(S/J_{\Delta})=\max\{|{\bf v}|-i:\beta_{i,{\bf v}}\neq 0\}$. In the next result, inspired by \cite[Lemma 2.1]{MaMu} $\supp({\bf v})=\{i:v_i\neq 0\}\subset [n]$ for each ${\bf v}=(v_1,\ldots ,v_n)\in \NN^n$. For each subset $W\subset 
\{1,\ldots ,n\}$, by $\Delta_W$ we denote the subcomplex of $\Delta$ induced on $W$.

\begin{proposition}\label{p:betti}
Let $\Delta$ be a $d$-dimensional simplicial complex on $n$ vertices and $W\subset [n]$. Whenever ${\bf v}\in\NN^n$ is such that $\supp({\bf v})\subset W$, 
\[\beta_{i,{\bf v}}(S/J_{\Delta})=\beta_{i,{\bf v}}(S/J_{\Delta_W}) \ \ \ \forall \ i\in\NN.\]
In particular, $\mathrm{reg}(S/J_{\Delta})\geq \mathrm{reg}(S/J_{\Delta_W})$.
\end{proposition}
\begin{proof}
Let $\mathbb{F}$ be the multi-graded minimal free resolution of $S/J_{\Delta}$:
\[\mathbb{F}: 0\rightarrow \bigoplus_{{\bf v}\in\NN^n}S(-{\bf v})^{\beta_{p,{\bf v}}}\rightarrow \ldots \rightarrow \bigoplus_{{\bf v}\in\NN^n}S(-{\bf v})^{\beta_{1,{\bf v}}}\rightarrow S\rightarrow 0.\]
Consider the complex of multi-graded $S$-modules
\[\mathbb{F}':0\rightarrow \bigoplus_{\substack{{\bf v}\in\NN^n \\ \supp({\bf v})\subset W}}S(-{\bf v})^{\beta_{p,{\bf v}}}\rightarrow \ldots \rightarrow \bigoplus_{\substack{{\bf v}\in\NN^n \\ \supp({\bf v})\subset W}}S(-{\bf v})^{\beta_{1,{\bf v}}}\xrightarrow{\phi} S\rightarrow 0.\]
The cokernel of $\phi$ is $S/J_{\Delta_W}$, hence all we need to show is that $\mathbb{F}'$ is acyclic. But since the minimal generators of the free $S$-modules in $\mathbb{F}'$ involve only the variables $x_{ij}$ with $j\in W$, to show that $\mathbb{F}'$ is acyclic is enough to show that $\mathbb{F}'_{{\bf u}}$ is acyclic for any ${\bf u}\in\NN^n$ with $\supp({\bf u})\subset W$.
On the other hand, for any ${\bf v}\in\NN^n$, $S(-{\bf v})_{{\bf u}}$ is nonzero if and only if ${\bf u}-{\bf v}\in\NN^n$: in particular $S(-{\bf v})_{{\bf u}}\neq 0$ implies $\supp({\bf v})\subset\supp({\bf u})\subset W$, hence $\mathbb{F}'_{{\bf u}}=\mathbb{F}_{{\bf u}}$ whenever $\supp({\bf u})\subset W$. We conclude since $\mathbb{F}_{{\bf u}}$ is acyclic for any ${\bf u}\in\NN^n$.
\end{proof}

\subsection{Many radical  and many F-pure determinantal facet ideals}

Let us warm up by studying the algebraic counterpart of the traceability of $\Delta$:

\begin{proposition}
Let $\Delta$ be a traceable $d$-dimensional simplicial complex on $n$ vertices. Then $\height(J_{\Delta})=n-d$. Furthermore, if $J_{\Delta}$ is radical and unmixed, then it admits a square-free initial ideal. If in addition $K$ has positive characteristic, then $S/J_{\Delta}$ is even $F$-pure.
\end{proposition}
\begin{proof}
Let us fix a labeling for which $\Delta$ is traceable. Set  \[C \eqdef ([1\ldots d+1],[2\ldots d+2],\ldots ,[n-d\ldots n])\subset J_{\Delta}.\] Let us fix a diagonal term order $<$ on $S$. Note that 
\[ \init_<([i\ldots i+d])=x_{0i}x_{1(1+i)}\cdots x_{d(d+i)} \quad \textrm{ and } \quad \init_<([j\ldots j+d])=x_{0j}x_{1(1+j)}\cdots x_{d(d+j)}\] 
are coprime if $i\neq j$. So $\{[1\ldots d+1],[2\ldots d+2],\ldots ,[n-d\ldots n]\}$ is a Gr\"obner basis of $C$ and
\[\init_<(C)=(x_{01}x_{12}\cdots x_{d (d+1)}, \: x_{02}x_{13}\cdots x_{d (d+2)}, \; \ldots \; , \: x_{0(n-d)}x_{1(1+n-d)}\cdots x_{dn})\]
is a complete intersection of height $n-d$. Hence $C$ is a complete intersection of height $n-d$ inside $J_{\Delta}$, which implies $\height(J_{\Delta})\geq n-d$. On the other hand $\height(J_{\Delta})\leq n-d$ because $J_{\Delta}$ is contained in $I_{d+1}(X)$, which has height equal to $n-d$. As for the final claim, set $g=[1\ldots d+1]\cdots [n-d\ldots n]$. Notice that $\init_<(g)$ is square-free. Obviously, we also have $C\in \C_g$. But if $J_{\Delta}$ is radical and unmixed, since $\height(J_{\Delta})=\height(C)$ by the previous part, then $J_{\Delta}$ must be of the form $C:h$ for some $h\in S$. Thus $J_{\Delta}\in \C_g$ and we conclude via Theorem \ref{t:lisa}.
\end{proof}

The next lemma will help us identify a large class of complexes whose determinantal facet ideal is indeed radical.

\begin{lemma}\label{l:poset}
Let $1\leq a_0<a_1<\ldots <a_d\leq n$, and $\Gamma_{{\bf a}}$ the simplicial complex generated by the facets $a_0i_1\ldots i_d$ with $i_j\leq a_j$ for all $j=1,\ldots ,d$. Then
\[J_{\Gamma_{{\bf a}}}=I_{d+1}(X_{[a_0,a_d]})\cap I_{d}(X_{[a_0,a_{d-1}]}) \cap I_{d-1}(X_{[a_0,a_{d-2}]})\cap \ldots \cap I_1(X_{[a_0,a_0]}).\]
Analogously, if $\Gamma^{{\bf a}}$ is the simplicial complex generated by the facets $i_0i_1\ldots a_d$ with $i_j\geq a_j$ for all $j=0,\ldots ,d-1$, then
\[J_{\Gamma^{{\bf a}}}=I_{d+1}(X_{[a_0,a_d]})\cap I_{d}(X_{[a_1,a_{d}]}) \cap I_{d-1}(X_{[a_2,a_{d}]})\cap \ldots \cap I_1(X_{[a_{d-1},a_d]}).\]
\end{lemma}
\begin{proof}
Since the two identities are symmetric,  we will  only prove the first one. The containment `$\subseteq$' is obvious; so let us show `$\supseteq$'. 
To make the notation lighter, we make the harmless assumption that $a_0=1$. Note that $J_{\Gamma_{{\bf a}}}$ is generated by a poset ideal, namely by \[\Omega=\{\pi \in \Pi: \pi \leq [a_0\ldots a_d]\}.\] Similarly, for all $j=0,\ldots ,d$, the ideal $I_{j+1}(X_{[1,a_{j}]})$ is generated by the poset ideal
\[\Omega_j=\{\pi \in \Pi: \pi \leq [d-j\ldots d|a_j-j\ldots a_j]\}.\]
Since it is easy to check that $\Omega=\cap_{j=0}^d\Omega_j$, via \cite[Proposition (5.2)]{BrunsVetter} we obtain
\[J_{\Gamma_{{\bf a}}}=I_{d+1}(X_{[1,a_d]})\cap I_{d}(X_{[1,a_{d-1}]}) \cap I_{d-1}(X_{[1,a_{d-2}]})\cap \ldots \cap I_1(X_{[1,1]}). \qedhere\]
\end{proof}

Now, let $\mmu \in S$ be the product of the minors whose main diagonals are illustrated in the $7\times 13$ matrix below.

\centerline{\includegraphics[height=5.5cm]{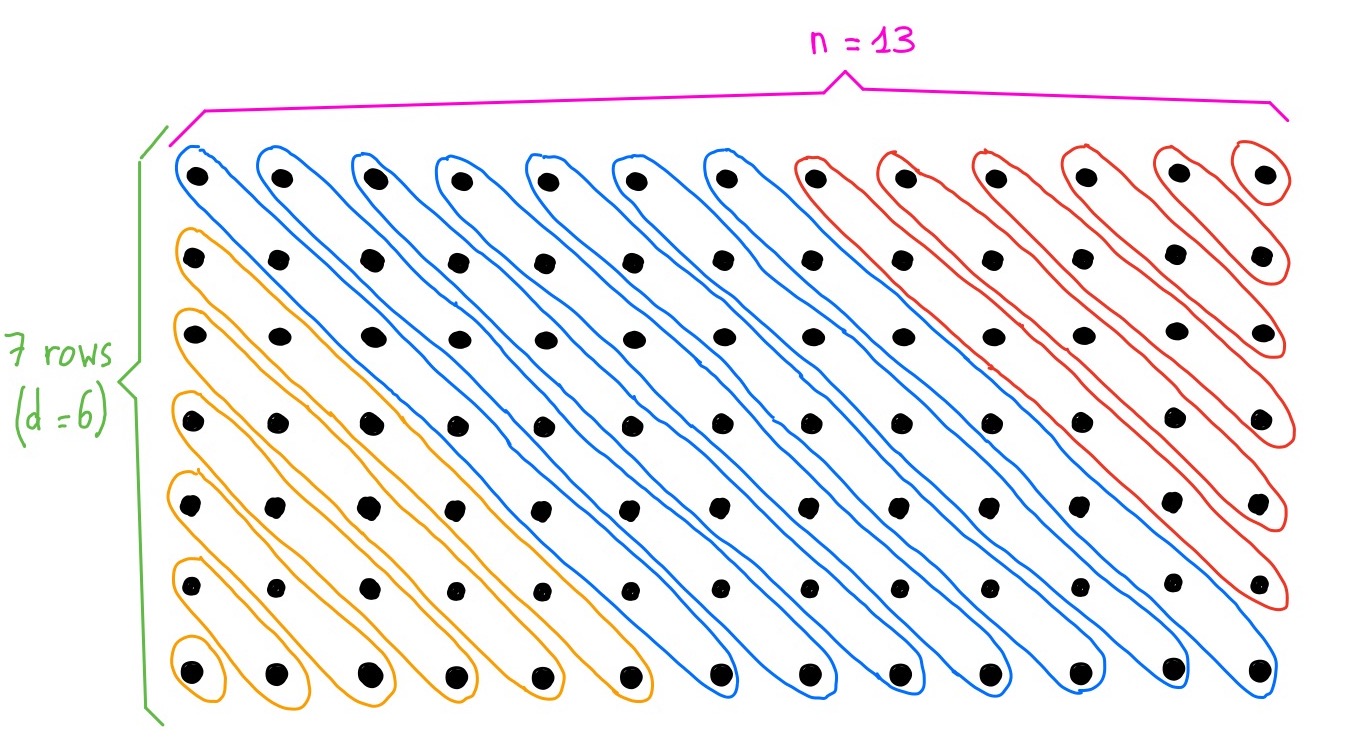}}

More precisely,
\begin{align*}
\mmu  \eqdef & [d|1][d-1,d|1,2]\cdots [1,2,\ldots ,d-1,d|1,2,\ldots ,d-1,d]\cdot \\
& [1,2,\ldots ,d,d+1]\cdots [n-d,n-d+1,\ldots ,n-1,n]\cdot \\
& [n-d+1,n-d+2,\ldots ,n-1,n|0,1,\ldots ,d-2,d-1]\cdots [n-1,n|0,1][n|0].
\end{align*}

The reason we defined $\mmu$ this way is that if $<$ is a diagonal term order, we have
\[\init_{<}(\mmu)=\prod_{i=0}^d\prod_{j=1}^nx_{ij}.\]
Using this $\mmu$, we are now ready to prove the first main result of this Chapter.

\begin{theorem}\label{t:s-c-f}
Let $\Delta$ be a $d$-dimensional semi-closed simplicial complex on $n$ vertices. Then $J_{\Delta}$ is a radical ideal. Moreover:
\begin{compactenum}[ \rm (1)]
\item For any diagonal term order (compatible with the labeling which makes $\Delta$ semi-closed), $\init(J_{\Delta})$ is a squarefree term ideal.
\item If the field $K$ has positive characteristic, $S/J_{\Delta}$ is $F$-pure.
\end{compactenum}
\end{theorem}

\begin{proof}
We will prove that if $\Delta$ is semi-closed with respect to the given labeling then $J(\Delta)\in \C_\mmu$, whence both claims follow by Theorem \ref{t:lisa}.
Let $1\leq a_0<a_1<\ldots <a_d\leq n$. Using the notation of Lemma \ref{l:poset}, since $\Delta$ is semi-closed, either $\Gamma_{{\bf a}}$ or $\Gamma^{{\bf a}}$ is contained in $\Delta$ whenever $a_0a_1\cdots a_d\in \Delta$. For any $a_0a_1\cdots a_d\in \Delta$, set
$\Delta_{{\bf a}}=\Gamma_{{\bf a}}$ if $\Gamma_{{\bf a}}\subset \Delta$, and $\Delta_{{\bf a}}=\Gamma^{{\bf a}}$ otherwise. Then
\[\Delta \ = \ \bigcup_{a_0a_1\cdots a_d\in \Delta}\Delta_{{\bf a}}.\]
In particular, 
\[ J(\Delta) \ = \ \displaystyle\sum_{a_0a_1\cdots a_d\in \Delta}J(\Delta_{{\bf a}}).\]  
Since $\C_\mmu$ is closed under sums, in order to show that $J(\Delta)\in \C_\mmu$ we only need to check that each $J(\Delta_{\bf a})\in \C_\mmu$. To verify this, we use a result in \cite{Se2}: The ideal $I_{r+1}(X_{[ij]})\in \C_\mmu$ whenever $1\leq i<j\leq n$ and $0\leq r\leq \min\{d,j-i\}$. Since $\C_\mmu$ is closed under intersections, Lemma \ref{l:poset} guarantees that $J(\Delta_{{\bf a}})\in \C_\mmu$, as desired.
\end{proof}

\begin{remark}
The assumption ``semi-closed'' is best possible: if we replace it with ``weakly-closed'', the theorem no longer holds, cf.~Example \ref{ex:notradical}. That said, the converse of Theorem  \ref{t:s-c-f} is false. To see this, consider the non-weakly-closed complex $U^2_3=124,345,467$ of Figure \ref{fig:ToyExamples1}. If $g=[124][345][467]$ then for a diagonal term order $\init(g)=x_{01}x_{12}x_{24}x_{03}x_{14}x_{25}x_{04}x_{16}x_{27}$, which is squarefree. Obviously $[124],[345],[467]\in \C_g$, hence $J_{\Delta}\in \C_g$. So $\init(J_{\Delta})$ is squarefree, and, in the positive characteristic case, $S/J_{\Delta}$ is $F$-pure by Theorem \ref{t:lisa}. On the other hand, when $d=1$ Theorem \ref{t:s-c-f} is true for all weakly closed graphs, via the main result of Matsuda \cite{Matsuda}. This shows that the techniques used in \cite{Matsuda} do not generalize to higher dimensions.
\end{remark}

\begin{remark}
Suppose that $K$ has positive characteristic. Theorem \ref{t:s-c-f} implies that, whenever $\Omega$ is a poset ideal of $\Pi$ consisting only of maximal minors, then the corresponding ASL is $F$-pure. On the other hand, some ASLs are not $F$-pure, as explained in \cite[Remark 5.2]{KoVa}. We do not know whether all the ASLs on a poset ideal of $\Pi$ are $F$-pure.
\end{remark}

\begin{remark}
In positive characteristic, having a square-free initial ideal or an $F$-pure quotient are unrelated properties.  Many ideals, like $I=(x^2+xy+y^2)\subset S \eqdef \mathbb{Z}/p\mathbb{Z}[x,y]$, for $p$ prime, have the property that $S/I$ is $F$-pure even if $\init(I)$ is not square-free for any term order. On the other hand, the binomial edge ideal of a 5-cycle is not $F$-pure in characteristic 2 \cite[Example 2.7]{Matsuda}, even if it admits a squarefree initial ideal.
See \cite{KoVa} for a discussion on the relationship between the two properties of being $F$-pure and having a squarefree initial ideal.
\end{remark}

Theorem \ref{t:s-c-f} allows us to characterize the determinantal facet ideals having a linear resolution: It turns out that there is only one. This extends to all dimensions the result for graphs by Saeedi-Madani and Kiani \cite{SMK}.

\begin{corollary} \label{cor:SMK}
Let $\Delta$ be a pure $d$-dimensional simplicial complex on $n$ vertices. 
\[ J_{\Delta} \textrm{ has a linear resolution } \ \Longleftrightarrow \ \Delta = \Sigma^d_n. \]
\end{corollary}
\begin{proof}  ``$\Leftarrow$'': If $\Delta$ is the $d$-skeleton of the $(n-1)$-simplex, $J_{\Delta}$ is the ideal of maximal minors of the matrix $X$. This ideal is resolved by the Eagon-Northcott complex \cite{EagonNorthcott}, which is linear. \\
``$\Rightarrow$'': By contradiction, suppose there is a subset $W\subset [n]$ of cardinality $d+2$ such that $\Delta_W$ is not the $d$-skeleton of the $(d+1)$-simplex on $W$. We can re-label the vertices so that $W=\{1,2,\ldots ,d+2\}$ and 
\[\Delta_W=12\ldots (d+1), \ 12\ldots d(d+2), \ \ \ \ldots \ \ \ ,\ 1\ldots i(i+2)(i+3)\ldots (d+2)\]
where $2\leq i\leq d$. With respect to such a labeling $\Delta_W$ is semi-closed. So by Theorem \ref{t:s-c-f}, $\init(J_{\Delta_W})$ is a squarefree monomial ideal for any diagonal term order. Hence, by the work of Conca--Varbaro \cite{CV}, $\mathrm{reg}(S/J_{\Delta_W})=\mathrm{reg}(S/\init(J_{\Delta_W}))$.  But by Lemma \ref{l:poset}
\[J_{\Delta_W}=I_{i}(X_{[1,i]})\cap I_{d+1}(X_{[1,d+2]}),\]
so by Theorem \ref{t:lisa} $\init(J_{\Delta_W})=\init(I_{i}(X_{[1,i]}))\cap \init(I_{d+1}(X_{[1,d+2]}))$. Via Theorem \ref{t:sturmfels}, it is easy to check that the monomial $(x_{d-i+1,1}x_{d-i+2,2}\cdots x_{d,i})(x_{0,2}x_{1,3}\cdots x_{d,d+2})$ is a minimal generator of $\init(I_{i}(X_{[1,i]}))\cap \init(I_{d+1}(X))$. Hence $\init(J_{\Delta_W})$ has a minimal generator of degree $i+d+1$. In particular,
\[ \mathrm{reg}(S/J_{\Delta})\ge \mathrm{reg}(S/J_{\Delta_W})=\mathrm{reg}(\init(S/J_{\Delta_W}))\geq i+d>d.\]
So by Proposition \ref{p:betti}, $\mathrm{reg}(S/J_{\Delta})\geq \mathrm{reg}(S/J_{\Delta_W})>d$. So $J_{\Delta}$ cannot have a linear resolution.
\end{proof}

\subsection{Determinantal facet ideals defined by a Gr\"obner basis}
If $\Delta$ is a closed simplicial complex, it is easy to see that the minors generating $J_{\Delta}$ form a Gr\"obner basis with respect to a diagonal monomial order, corresponding to the labeling that makes $\Delta$ closed: See \cite{EneEtAl}. In \cite{EneEtAl} it has been incorrectly claimed that the converse of the above statement holds true. The following result, which is a consequence of \cite[Corollary 2.4]{Se2}, shows that there are  many other complexes $\Delta$ for which the minors generating $J_{\Delta}$ form a Gr\"obner basis:

\begin{theorem}\label{t:a-c-gb}
Let $\Delta$ be a $d$-dimen\-sional simplicial complex, with a labeling that makes it unit-interval. 
The set $\{[a_0\ldots a_d]:a_0\ldots a_d]\}$ is a Gr\"obner basis of $J_{\Delta}$ with respect to any diagonal term order.  If in addition the field $K$ has positive characteristic, then $S/J_{\Delta}$ is $F$-pure.
\end{theorem}

\begin{proof}
By definition, $\Delta$ is the union of $d$-skeleta of simplices on consecutive vertices. We can choose these $d$-skeleta to be maximal with respect to inclusion. This yields a decomposition 
\[\Delta \ = \ \Sigma_{[i_1,j_1]}^{d} \: \cup \:  \Sigma_{[i_2,j_2]}^{d} \: \cup \:  \ldots \: \cup \: \Sigma_{[i_l,j_l]}^{d},\]
where $\Sigma_{[i_k,j_k]}^{d}$ denotes the $d$-skeleton of the simplex on vertices $i_k,i_k +1, i_k+2, \ldots, j_k$. Therefore 
\[J_{\Delta} \ = \ I_{d+1}(X_{[i_1,j_1]}) \: + \: I_{d+1}(X_{[i_2,j_2]}) \:  + \:  \ldots \:  + \:  I_{d+1}(X_{[i_l,j_l]}).\] 
So by \cite[Corollary 2.4]{Se2}
\[\init_<(J_{\Delta}) \ = \ \init_< (I_{d+1}(X_{[i_1,j_1]})) \: + \:  \init_< (I_{d+1}(X_{[i_2,j_2]})) \: +\:  \ldots \: +  \: \init_< (I_{d+1}(X_{[i_l,j_l]})).\]
By Theorem \ref{t:sturmfels}, $\{ [a_0, \ldots, a_d] \mid a_0\ldots a_d \in \Delta\}$ is a Gr\"obner basis for $J_\Delta$. Finally, the $F$-purity claim in the case of positive characteristic follows again from \cite[Corollary 2.4]{Se2}.
\end{proof}

\begin{remark}
That the set $\{[a_0\ldots a_d]:a_0\ldots a_d \in \Delta]\}$ is a Gr\"obner basis when $\Delta$ is unit-interval has been independently proved, using a completely different method, in Almousa--Vandebogert \cite[Theorem 2.16]{AV21}. They also obtained the analogous result for $r$-determinantal facet ideals (a more general concept than determinantal facet ideals) of unit-interval simplicial complexes. We were not aware of the paper \cite{AV21} of Almousa and Vandebogert before posting the first version of the present work on the arXiv. (We coordinated efforts to adopt the same name ``unit-interval complexes'' in the two papers.)
For the sake of completeness, we point out that \cite[Corollary 2.4]{Se2} implies that also $r$-determinantal facet ideals of unit-interval simplicial complexes define $F$-pure quotients in positive characteristic. We do not  know, however, whether the ($r$-)determinantal facet ideals of ``lcm-closed'' complexes, as defined in \cite{AV21}, or whether those of ``closed complexes'', as defined here, are all $F$-pure. 
\end{remark}

\begin{remark} \label{rem:closedNotAlmost}
The converse of Theorem \ref{t:a-c-gb} is false:  as explained above, any closed but not unit-interval complex is a counterexample. For a more interesting example, consider
\[ W = 123, \; 124,  \;  134,  \; 234,  \; 235,  \; 245,  \; 345,  \; 568,  \; 789,  \; 8 \, 10 \, 11\]
corresponding to a one-point union of the $B^2$ and the $U^2_3$ of Figure \ref{fig:ToyExamples1}. This complex $W$ is not unit-interval, not closed, and not even weakly-closed \cite{Pavelka}. 
However, one can verify with Macaulay2 \cite{m2} or via \cite[Theorem 2.15]{AV21} that $\{[a_0, a_1, a_2]: a_0\ldots a_d \in \Delta]\}$ form a Gr\"obner basis of $J_{W}$ for any diagonal term order. 
\end{remark}

\begin{remark}\label{rem:correction}
 Two of the results of \cite{EneEtAl} are incorrect because of the following counterexamples. As we already mentioned, the complex $B_d$
of Lemma \ref{lem:Bd} (cf.~Figure \ref{fig:ToyExamples1}) is not closed, but the set of all the minors $[abc]$, where $abc$ ranges over all facets of $B^d$,  is a Gr\"obner basis of $J_{B^d}$ for any diagonal term order by Theorem \ref{t:a-c-gb}. Thus one direction of \cite[Theorem 1.1]{EneEtAl}  is incorrect for all $d>1$. Moreover,  the graph $G_0 = 12, 13, 23, 24, 34$
is closed, but one can verify that $S/J_{G_0}$ is not Cohen-Macaulay. Thus \cite[Corollary 1.3]{EneEtAl} is incorrect already for $d=1$. 
\end{remark}

The final part of our work is dedicated to the delicate quest for some partial converse for Theorem \ref{t:a-c-gb}. 
To increase the chances of success, we restrict ourselves to \emph{traceable} complexes. The traceable assumption is rather natural in this case, 
as we have anyway seen in Theorem \ref{thm:CTSC} that all strongly-connected unit-interval complexes are traceable. 
We start off with a Lemma:

\begin{lemma}\label{l:gb-ac}Let $\Delta$ be a simplicial complex such that  $\mathcal{GB} \eqdef \left\lbrace [a_0,\ldots,a_d] \mid a_0\ldots a_d \in \Delta \right\rbrace$ is a Gr\"obner basis of $J_{\Delta}$ for some diagonal term order. Let $F= a_0\ldots a_d$ and $G=b_0 \ldots b_d$ be two facets of $\Delta$. If for some integer $l \in \{0,\ldots d-1\}$ 
\begin{compactitem}
\item[\rm (i)] $a_i=b_i$ for all $i \in \{0,\ldots, l\}$,
\item[\rm (ii)]$a_{l+1}> a_l +1$,
\item[\rm (iii)] $b_{l+k}= b_l +k$ for all $k\geq 1$,
\end{compactitem}
then the facet $a_0 \ldots a_{l-1}(a_l +1) \ a_{l+1} \ldots a_d$ is also in $\Delta$.
Symmetrically, if for some $l \in \{1,\ldots d\}$
\begin{compactitem}
\item[\rm (iv)] $a_i=b_i$ for all $i \in \{l,\ldots, d\}$,
\item[\rm (v)]$a_{l-1}< a_l -1$,
\item[\rm (vi)] $b_{l-k}= b_l -k$ for all $k \in \{1, \ldots, l\}$, 
\end{compactitem}
then the facet $a_0 \ldots a_{l-1} (a_l -1) \ a_{l+1} \ldots a_d$ is also in $\Delta$.
\end{lemma}

\begin{proof} 
It is harmless to assume that the term order $<$ is the lexicographic term order defined before, cf.~Theorem \ref{t:sturmfels}.
Let $F$ and $G$ be two facets of $\Delta$ satisfying (i), (ii) and (iii). Let us compute the initial term of the polynomial \[f\eqdef  [l+1 \ldots d \mid a_{l+1} \ldots a_d][b_0 \ldots b_d]- [l+1 \ldots d \mid b_{l+1} \ldots b_d][a_0 \ldots a_d].\] 

If we set 
\begin{equation*}
\begin{split}
p&\eqdef [a_0 \ldots a_d], \ \ p'\eqdef [l+1 \ldots d \mid a_{l+1} \ldots a_d]\\
q&\eqdef [b_0 \ldots b_d], \ \ \ q' \eqdef [l+1 \ldots d \mid b_{l+1} \ldots b_d]
\end{split}
\end{equation*}

then $f=p'q-pq'$, and by Laplace expansion we have
\begin{equation*}
\begin{split}
p'q&=\underbrace{(x_{0 b_0}\cdots x_{l-1 b_{l-1} } x_{l b_{l}} p'q')}_{h}+ g_1, \quad \mu_1 < \alpha \ \forall \ \mu_1 \in \supp (g_1), \  \forall \ \alpha \in \supp (h), \\
pq'&=\underbrace{(x_{0 a_0}\cdots x_{l-1 a_{l-1} } x_{l a_{l}} p'q')}_{h}+ g_2,\quad \mu_2 < \alpha \ \forall \ \mu_2 \in \supp (g_2), \  \forall \ \alpha \in \supp (h).\\
\end{split}
\end{equation*}
Furthermore
\begin{equation*}
\begin{split}
\init_< (g_1)=( x_{l+1 a_{l+1}} \cdots x_{d  a_d})(x_{0 b_0}\cdots x_{l-1 b_{l-1} } x_{l b_{l+1}} x_{l+1 b_l} x_{l+2 b_{l+2}} \ldots x_{d b_d}),\\
\init_<(g_2)=( x_{l+1 b_{l+1}} \cdots x_{d  b_d})(x_{0 a_0}\cdots x_{l-1 a_{l-1} } x_{l a_{l+1}} x_{l+1 a_l} x_{l+2 a_{l+2}} \ldots x_{d a_d}).
\end{split}
\end{equation*}
Since $\init_<(g_2)$ is smaller than $\init_<(g_1)$, we conclude that\[\init_< (f)=\init_< (g_1-g_2)=( x_{l+1 a_{l+1}} \cdots x_{d  a_d})(x_{0 b_0}\cdots x_{l-1 b_{l-1} } x_{l b_{l+1}} x_{l+1 b_l} x_{l+2 b_{l+2}} \ldots x_{d b_d}).\]
In addition $f \in J_\Delta$ because $F, G \in \Delta $. Thus, there must be a minor $g=[c_0 \ldots c_d]$ in $\mathcal{GB}$ such that $\init_<(g)$ divides $\init_<(f)$. Note that for $c_0, \ldots,c_l$ we only have one option, namely,
\begin{equation*}
\left\{
\begin{array}{ll}
c_0 &= b_0=a_0\\
& \ \vdots \\
 c_{l-1}&=b_{l-1}=a_{l-1}\\
c_l&=b_{l+1}=b_l+1=a_l+1.
\end{array} \right.
\end{equation*}
For $c_{l+1}$ we have a priori two possibilities: either $c_{l+1}= b_l$ or  $c_{l+1}=a_{l+1}$. But $b_l < b_{l+1}=c_l$, so it must be $c_{l+1}=a_{l+1}$. Similarly, for $c_{l+2}$ we have a priori  two options: Either $c_{l+2}= b_{l+2}$, or  $c_{l+2}=a_{l+2}$. But by the assumptions, we have that $b_{l+2}\leq a_{l+1}=c_{l+1}$, so since $c_{l+2}>c_{l+1}$ it must be $c_{l+2}=a_{l+2}$. 
In general, for any $k\geq 2$ we have $b_{l+k}\leq a_{l+k-1}=c_{l+k-1}$. Since $c_{i}>c_{i-1}$, arguing recursively we obtain that the only possible option is $c_{l+k}=a_{l+k}$ for all $k\geq 2$.
Hence we have proved that
\[g=[c_0,\ldots,c_d]=[a_0\ldots a_{l-1} (a_l+1) a_{l+1} \ldots a_d].\]
Since $g$ is an element of $\mathcal{GB}$, we conclude that  $a_0\ldots a_{l-1} (a_l+1) a_{l+1} \ldots a_d \in \Delta$.

The proof of the second part of the lemma is symmetric; namely, one considers the polynomial  
\[f' \eqdef  [0 \ldots l-1 \mid a_{0} \ldots a_{l-1}][b_0 \ldots b_d]- [0 \ldots l-1 \mid b_{0} \ldots b_{l-1}][a_0 \ldots a_d] \in J_{\Delta}\] whose leading term is 
\[\init_< (f')=( x_{0 a_{0}} \cdots x_{l-1  a_{l-1}})(x_{0 b_0}\cdots x_{l-2 b_{l-2} } x_{l-1 b_{l}} x_{l b_{l-1}} x_{l+1 b_{l+1}} \ldots x_{d b_d}),\]
and one proceeds analogously to the argument above.
\end{proof}

\begin{theorem}\label{t:a-c-gb1}
Let $\Delta$ be a $d$-dimensional simplicial complex. If with respect to the same labeling $\Delta$ is traceable and the set $\{[a_0\ldots a_d]:a_0\ldots a_d]\}$ is a Gr\"obner basis of $J_{\Delta}$ with respect to some diagonal term order, then such labeling makes $\Delta$ unit-interval.
\end{theorem}

\begin{proof} 
Let $F=a_0 \ldots a_d$ be a facet of $\Delta$ with $\operatorname{gap} (F)=k$. We proceed  by induction on $k$. For $k=0$ there is nothing to prove, so we assume $k>0$. 
Let $g_1, \ldots, g_k$ be the vertices not in $F$, and such that $a_0<g_1< \ldots <g_k<a_d$. We want to show that $\Delta$ contains the $d$-skeleton of $\{a_0, \ldots,a_d,g_1,\ldots,g_k\}$. The strategy is to first show that $\Sigma_{[a_0+1,a_d]}^{d}, \Sigma_{[a_0,a_d-1]}^{d} \subseteq \Delta$ by inductive assumption, and then to prove that $\Delta$ contains also the facets of the form $a_0 c_1 \ldots c_{d-1} a_d$. So let us proceed. Let $l$ be the greatest integer such that $a_l <g_1$, so that $g_1=a_l+1$. Consider the two facets $F$ and $H_{a_0}$ of $\Delta$.
They satisfy the assumptions of Lemma \ref{l:gb-ac}, so 
\[F'=a_0 \ldots a_{l-1} \ g_1 \ a_{l+1} \ldots a_d \in \Delta.\]
If $l=0$, then $\operatorname{gap}(F')=k-1$, so by the inductive assumption $\Sigma_{[g_1,a_d]}^{d}=\Sigma_{[a_0+1,a_d]}^{d}\subset \Delta$. Otherwise, since $\operatorname{gap} (F')=k$, we cannot apply the inductive assumption yet. However, we have \lq \lq shifted" the first gap to the left and now the first missing vertex is $a_l=a_{l-1}+1$. We can apply again Lemma \ref{l:gb-ac} to the facets $F'$ and $H_{a_0}$ and we get
\[F''=a_0 \ldots a_{l-2} \ a_l \ g_1 \ a_{l+1} \ldots a_d \in \Delta.\]
If $l=1$, then $\operatorname{gap}(F'')=k-1$, so by the inductive assumption $\Sigma_{[a_1,a_d]}^{d}=\Sigma_{[a_0+1,a_d]}^{d}\subset \Delta$. Otherwise, once again $\operatorname{gap} (F'')=k$ and the first missing vertex $a_{l-1}=a_{l-2}+1$ has been shifted by one to the left. 
Iterating this procedure, we eventually get that \[(a_{0}+1)\ldots a_{l}\  g_1 \ a_{l+1} \ldots a_d \in \Delta.\]
This face has gap equal to $k-1$ and we can finally apply induction: We get $\Sigma_{[a_0+1,a_d]}^{d} \subseteq \Delta$.\par 
To prove that $\Sigma_{[a_0,a_d-1]}^{d} \subseteq \Delta$ we  use a similar argument. Let $l$ be the smaller integer such that $g_k<a_l$, so that $g_k=a_l-1$, and consider the two facets of $\Delta$
\begin{equation*}
\begin{split}
F&=\ \ \ a_0 \ldots \ a_{l-1}\ a_l\ldots a_d\\
H^{a_d} \eqdef H_{a_d-d}&=(a_d-d)(a_d-d+1) \ldots g_k \ a_l\ldots \  a_d.
\end{split}
\end{equation*}
Iteratively applying the second part of Lemma \ref{l:gb-ac}, we can shift the last missing vertex to the right until we reach the facet
\[a_0 \ldots a_{l-1} g_k a_l \ldots a_d -1 \in \Delta,\]
which has gap $k-1$. So by induction $\Sigma_{[a_0,a_d-1]}^{d} \subseteq \Delta$.\par
It remains to prove that all the facets of the form $G=a_0 c_1 \ldots c_{d-1} a_d $ are in $\Delta$. To do so, we start from $F=a_0 a_1 \ldots a_d$ and we replace one by one each $a_i$ with the corresponding $c_i$. In detail: For $i=1$, we have three possibilities:
\begin{compactitem}
\item $c_1=a_1$, or
\item $a_0< c_1 < a_1$, or
\item $c_1 > a_1$.
\end{compactitem}
If $c_1=a_1$ there is nothing to do. If $a_0< c_1 < a_1$, consider the two facets 
\begin{equation*}
\begin{split}
F&= \ \ \  a_0 \ \  \ a_1 \ldots a_d\\
\widetilde{F}&=(a_1-1) \ a_1 \ldots  a_d.
\end{split}
\end{equation*}
Since $a_0< c_1 < a_1$, we have that  $a_1-1 >  a_0$. Hence $\widetilde{F}  \in \Sigma_{[a_0+1 ,a_d]}^{d} \subseteq \Delta$. So by Lemma \ref{l:gb-ac} $a_0  (a_1-1 ) a_2 \ldots a_d \in \Delta$. If $c_1=a_1-1$ we stop,  otherwise we repeat the same argument. At each iteration of Lemma \ref{l:gb-ac}, the second vertex in the facet decreases by one unit; eventually, we obtain that $a_0 c_1 a_2 \ldots a_d \in \Delta$.

As for the third possibility ($c_1 > a_1$), we claim that we can simply dismiss it without loss of generality. 
In fact, for every  $i\in \{1,\ldots,d\}$, we can always  ``flatten all the vertices after $a_{i-1}$ to the right'': that is, we can always replace  $F$ with another face in $\Delta$ of the form
\[F_i=a_0 a_1\ldots a_{i-1} (a_d-d-i) \ldots (a_d-1) a_d.\] 
To see it, let $0\leq l\leq d-1$ be the largest index for which $a_l+1<a_{l+1}$ (such an $l$ must exist because $\operatorname{gap} (F)>0$). Applying Lemma \ref{l:gb-ac} to the facets $F$ and $\widetilde{F}=a_0\ldots a_{l}(a_{l}+1)(a_l+2)\ldots a_l+(d-l)$ in  $\Sigma_{[a_0,a_d-1]}^{d}\subseteq \Delta$, we get that the facet $a_0 \ldots a_{l-1}(a_l+1)a_{l+1}\ldots a_d$ is in $\Delta$. Proceeding this way we end up with the face 
\[F_l=a_0 \ldots a_{l-1}(a_{l+1}-1)a_{l+1}\ldots a_d\in \Delta.\]
Replacing $F$ with $F_{l}$, and arguing the same way, we infer that $F_i\in\Delta$ for all $i=0,\ldots ,d-1$.
In particular, for $i=1$, we could replace $F$ with a face with same minimum and maximum \[ F_1=a_0 \  (a_d-d+1) \ (a_d-d+2)  \ldots \  a_d\in \Delta.\] Note that $c_1 \le a_d-d+1$. So our claim is proven: Up to replacing $F$ with $F_1$, we can assume that $c_1 \le a_1$.
 
So the case $i=1$ is settled. Consider now $i=2$.  If $c_2=a_2$, there is nothing to do. Otherwise, flattening the vertices after $c_1$ of $a_0c_1a_2\ldots a_d$ to the right, we may assume that $c_2 < a_2$. Consider the two facets
\begin{equation*}
\begin{split}
F &= \ \ \  a_0 \ \ \ \  \ c_1 \ \ \ \ a_2 \ldots a_d \in \Delta\\
\widetilde{F} &=(a_2-2) \ (a_2-1) \ a_2\ldots  a_d \in \Delta.
\end{split}
\end{equation*} 
Since $c_2 < a_2$, we have that $c_1 < a_2-1$, so applying Lemma \ref{l:gb-ac} we obtain that \[a_0 \, c_1 \, (a_2-1) \, a_3 \ldots a_d \in \Delta.\]
If $c_2=a_2-1$ we stop, otherwise we repeat the same argument. At every iteration of Lemma \ref{l:gb-ac}, the third vertex in the facet decreases by one unit; eventually, we obtain that $a_0 c_1 c_2 a_3 \ldots a_d \in \Delta.$

Iterating this procedure for all $i$'s, we conclude that 
\[G=a_0 c_1 c_2 \ldots c_{d-1} a_d \in \Delta. \qedhere\]
\end{proof}

\begin{remark} \label{rem:lcm} Very recently  Almousa and Vandebogert   \cite{AV21} introduced a technical property of simplicial complexes, called ``lcm-closed'', that simultaneously generalizes the two properties of being ``closed'' and being ``unit-interval''. They asked \cite[Question 2.19]{AV21} whether such property for simplicial complexes would characterize the fact that the minors of the determinantal facet ideal form a Gr\"obner basis with respect to any diagonal term order. With a little ingenuity, one can see that for traceable complexes, ``lcm-closed'' is simply equivalent to ``unit-interval''. Thus  Theorem \ref{t:a-c-gb1} answers Almousa--Vandebogert's question positively, for complexes that with respect to the same labeling are traceable.
\end{remark}

%%%%%%%%%%%%%%%%%%%%%%%%%%%%%%%%%%%%%%%%%%%%%%%%%%%%%%%%%%%%%%%%%%%%%%%%%

\end{document}